\documentclass[a4paper,11pt]{article}
\usepackage[all]{xy}

\usepackage{typearea}
\typearea{13}

\usepackage{mymacros2}
\usepackage{mathbbol}
\newcommand{\DR}{{\mathrm{DR}}}

\newcommand{\forlim}{{\text{``${\displaystyle{\lim}}$''}}}
\newcommand{\forlimit}[1]{\underset{\substack{\longrightarrow\\ {#1}}}{\text{``${\displaystyle{\lim}}$''}}}
\newcommand{\bRz}{\bR_{\geq 0}}
\renewcommand{\Mod}{\mathrm{mod}}
\newcommand{\Modic}{\Mod_{ic}}
\newcommand{\im}{\mathrm{im}}
\newcommand{\coim}{\mathrm{coim}}
\newcommand{\Lambdaaff}{\Mod^{\frakI}(\Lambda)}
\newcommand{\Open}{\mathrm{Open}}

\newcommand{\Ierv}{\mathrm{Ierv}}
\newcommand{\Perv}{\mathrm{Perv}}
\renewcommand{\Mod}{\mathrm{Mod}}

\newcommand{\pD}{{^p}D}

\newcommand{\hol}{\mathrm{hol}}
\newcommand{\alg}{\mathrm{alg}}
\newcommand{\Sol}{\mathrm{Sol}}
\newcommand{\ModIp}{\Mod^\frakI_{pre}}
\newcommand{\ModIps}{{\Mod^\frakI_{ps}}}
\newcommand{\ModI}{\Mod^\frakI}
\newcommand{\tcV}{\tilde{\cV}}
\newcommand{\tcW}{\tilde{\cW}}
\newcommand{\tcX}{\tilde{\cX}}
\newcommand{\tcP}{\tilde{\cP}}
\newcommand{\tcI}{\tilde{\cI}}

\newcommand{\dotimes}{\otimes^\bL}
\newcommand{\tilh}{\tilde{h}}

\newcommand{\tf}{\tilde{f}}
\newcommand{\tg}{\tilde{g}}
\newcommand{\oR}{\overline{\bR}}
\newcommand{\oX}{{\overline{X}}}
\newcommand{\oY}{{\overline{Y}}}
\newcommand{\oU}{{\overline{U}}}
\newcommand{\oS}{\overline{S}}
\newcommand{\mI}{\mathrm{I}}
\newcommand{\dX}{{(\oX, D_X)}}
\newcommand{\dY}{{(\oY, D_Y)}}
\newcommand{\dS}{{(\oS, D_S)}}
\newcommand{\RcHom}{{\bR\cH om}}
\newcommand{\of}{\overline{f}}
\newcommand{\IC}{\mathrm{I}\bC}
\newcommand{\Ik}{\mathrm{I}\bk}
\newcommand{\tM}{\tilde{M}}
\newcommand{\Ind}{\mathrm{Ind}}
\newcommand{\an}{{\mathrm{an}}}
\newcommand{\potimes}{\overset{+}{\otimes}}
\newcommand{\pboxtimes}{\overset{+}{\boxtimes}}
\newcommand{\cIhom}{\cI hom}
\newcommand{\hO}{\widehat{\cO}}
\newcommand{\hi}{\mathchar`-}
\newcommand{\bk}{\mathbb{k}}
\title{Irregular perverse sheaves}
\author{Tatsuki Kuwagaki}

\date{}
\begin{document}

\maketitle

\begin{abstract}
We introduce irregular constructible sheaves, which are $\bC$-constructible with coefficients in a finite version of Novikov ring $\Lambda$ and special gradings. We show that the bounded derived category of cohomologically irregular constructible complexes is equivalent to the bounded derived category of holonomic $\cD$-modules by a modification of D'Agnolo--Kashiwara's irregular Riemann--Hilbert correspondence. The bounded derived category of cohomologically irregular constructible complexes is equipped with the irregular perverse t-structure, which is a straightforward generalization of usual perverse t-structure and we see its heart is equivalent to the abelian category of holonomic $\cD$-modules. We also develop the algebraic version of the theory. Furthermore, we discuss the reason of the appearance of Novikov ring by using a conjectural reformulation of Riemann--Hilbert correspondence in terms of certain Fukaya category.
\end{abstract}

\section{Introduction}
The regular Riemann--Hilbert correspondence (formulated and proved by Kashiwara \cite{KashiwaraRH} and another proof by Mebkhout \cite{MebkhoutRH}) states that the derived category of regular holonomic $\cD$-modules is equivalent to the derived category of $\bC$-constructible sheaves. Under this equivalence, the abelian category of regular holonomic $\cD$-modules is mapped to the abelian category of perverse sheaves introduced by Beilinson--Bernstein--Deligne--Gabber~\cite{Kashiwaramaster, BBD, GM}.

After many efforts including understanding of formal and asymptotic structures \cite{Majima, SabbahHLT, Mochizukiformal, Kedlaya}, Stokes phenomena and Riemann--Hilbert correspondence for meromorphic connections~\cite{Malgrange, Sibuya, Deligne, Mochizukiformal,SabbahStokes}, sophistication of the regular Riemann--Hilbert correspondence \cite{MR1827714}, and developments of ind-sheaves and the discovery of its relation to asymptotic behavior \cite{MR1827714, MR2003419}, in a seminal paper~\cite{D'Agnolo--Kashiwara}, D'Agnolo--Kashiwara formulated and proved the irregular Riemann--Hilbert correspondence for holonomic $\cD$-modules:
\begin{theorem}[{\cite[D'Agnolo--Kashiwara]{D'Agnolo--Kashiwara}}]
For a complex manifold $X$, there exists a fully faithful embedding
\begin{equation}\label{intDK}
D^b_{\hol}(\cD_X)\hookrightarrow E^b_{\bR\hi c}(\IC_X).
\end{equation}
where the left hand side is the derived category of cohomologically holonomic $\cD$-modules and the right hand side is the category of $\bR$-constructible $\bC$-valued enhanced ind-sheaves.
\end{theorem}
In the sequel~\cite{D'Agnolo--Kashiwara2}, they also introduced the notion of enhanced perverse t-structure on the right hand side of the embedding and proved that the embedding is t-exact in a slightly generalized sense. Moreover, Mochizuki~\cite{Mochizukicurvetest} proved that the image of the equivalence can be characterized by curve test.

In this paper, we modify the right hand side of the equivalence and make it more closer to the form of the regular Riemann--Hilbert correspondence.

As mentioned in their paper, D'Agnolo--Kashiwara's clever definition and use of enhanced sheaves are inspired from the construction of Tamarkin~\cite{Tam}. Tamarkin's idea of adding one extra variable originally aimed to realize Novikov ring action in sheaf theory as in Fukaya category~\cite{FOOO}. In this paper, we take a way which is more closer to this original idea instead of the use of enhanced sheaves. The replacement for the right hand side of (\ref{intDK}) is expressed as graded modules (sheaves) over the ``finite Novikov ring'' $\Lambda:=\bk[\bR_{\geq 0}]$ where $\bk\subset \bC$ is a field. An element of $\Lambda$ is expressed as a finite sum $\sum_{a\in \bR_{\geq 0}}c_aT^a$ where $T$ is the indeterminate. A priori, the hom-spaces $\Hom(\cV, \cW)$ of $\Lambda$-modules are defined over $\Lambda$. By taking the tensor product $\Hom(\cV, \cW)\otimes_{\Lambda}\bk$ where $\Lambda\rightarrow \bk$ is defined by $T^a\mapsto 1$, we obtain a new category $\ModIp(\Lambda_X)$. We will further modify this category to obtain $\ModI(\Lambda_X)$. Anyway, we can consider $\ModIp(\Lambda_X)$ as an approximate description of $\ModI(\Lambda_X)$.

The category $\ModI(\Lambda_X)$ is abelian and has enough injective and flat objects. We define an abelian subcategory of $\ModI(\Lambda_X)$: the category of irregular constructible sheaves $\Mod_{ic}(\Lambda_X)$. Then we set $D^b_{ic}(\Lambda_X)$ as the full subcategory of the bounded derived category $D^b(\ModI(\Lambda_X))$ consisting of cohomologically irregular constructible sheaves. The meaning of irregular constructibility is as follows: As usual there exists a $\bC$-Whitney stratification and we have a sheaf which is locally constant as $\Lambda$-module over each stratum, but moreover with particular gradings coming from Sabbah--Mochizuki--Kedlaya's Hukuhara--Levelt--Turrittin theorem~\cite{SabbahHLT, Mochizukiformal, Kedlaya}. Then we have the following:

\begin{theorem}
\begin{enumerate}
\item The category $D^b_{ic}(\Lambda_X)$ has functors $\cHom, \otimes, f^{-1}, f^!$ for any morphism $f$ and $f_!$ for proper $f$.
\item If $\bk=\bC$, there exists an equivalence 
\begin{equation}\label{RHintro}
D^b_{\hol}(\cD_X)\xrightarrow{\simeq} D^b_{ic}(\Lambda_X).
\end{equation}
\end{enumerate}
\end{theorem} 

In our formulation, the data of exponential factors of solutions of irregular differential equations are encoded in the grading of $\Lambda$-modules. We would like to apply the following trivial fact to our setting: For a graded ring $R$, the grading-forgetful functor from the abelian category of graded $R$-modules to the abelian category of $R$-modules is exact. Nevertheless our category $\Mod_{ic}(\Lambda_X)$ has a bit exotic modification of hom-spaces, we still have the following:
\begin{theorem}
There exists an exact functor $\frakF$ from $\Mod_{ic}(\Lambda_X)$ to the abelian category of $\bC$-constructible sheaves $\Mod_{c}(\bk_X)$.
\end{theorem}

By using $\frakF$, we can define the support of an irregular constructible sheaf $\cV$ by $\supp\cV:=\supp \frakF(\cV)$. By using this definition, we can define the irregular perverse t-structure by the same formula as in usual perverse sheaves: Let $\pD_{ic}^{\leq 0}(\Lambda_X)$ (resp. $\pD_{ic}^{\geq 0}(\Lambda_X)$) be the full subcategory of $D^b_{ic}(\Lambda_X)$ spanned by objects satisfying 
\begin{equation}
\dim\lc\supp H^j(\cV)\rc\leq -j \hspace{3mm} (\text{resp.}\dim\lc\supp H^j(\bD\cV)\rc\leq -j ) \text{ for any }j\in \bZ.
\end{equation}

\begin{theorem}
\begin{enumerate}
\item The pair $(\pD_{ic}^{\leq 0}(\Lambda_X), \pD_{ic}^{\geq 0}(\Lambda_X))$ defines a t-structure of $D^b_{ic}(\Lambda_X)$, which we call irregular perverse t-structure.
\item The heart of irregular perverse t-structure $\Ierv(\bC_X)$ over $\bC$ is equivalent to the abelian category of holonomic $\cD$-modules under the equivalence $(\ref{RHintro})$.
\end{enumerate}
\end{theorem}
We also prove the corresponding results in algebraic setting: Mostly, the statements are corollaries of analytic cases, although we also have $f_*$ and $f_!$ for any morphism and can prove more stronger commutativity results for the Riemann--Hilbert functor (as in the case of algebraic regular Riemann--Hilbert correspondence).

As perverse sheaves have vast applications to mathematics including Hodge theory, topology, geometric representation theory, and etc, one can expect irregular perverse sheaves have such applications too, which are possible future works.

We also discuss a conjectural explanation of the appearance of Novikov ring using Fukaya category, which makes D'Agnolo--Kashiwara's approach closer to Tamarkin's one. Our main conjecture is the following (a slightly more precise form is presented in Section 11):
\begin{conjecture}
\begin{enumerate}
\item There exists a version of Fukaya category $\Fuk_{icnov}(T^*X)$ defined over finite Novikov ring $\Lambda$.
\item After taking derived category and reducing coefficients $\Lambda$ to $\bk$, we denote the resulting category by $D\Fuk_{ic}(T^*X)$. Then we have an equivalence $D\Fuk_{ic}(T^*X)\simeq D^b_{ic}(\Lambda_X)$. In particular, over $\bk=\bC$, we have the Fukaya categorical Riemann--Hilbert correspondence
\begin{equation}\label{equiv1.3}
D^b_{\hol}(\cD_X)\simeq D\Fuk_{ic}(T^*X).
\end{equation}
\end{enumerate}
\end{conjecture}
If the conjecture is true, one can imagine $K$-theory classes of objects of $D\Fuk_{ic}(T^*X)$ as an irregular version of characteristic cycle. In the same vein, their supports can be considered as an irregular version of microsupports, which are no longer conic. Hence one can also imagine a generalization of microlocal analysis. Note that a version (real blowed-up version) of the equivalence (\ref{equiv1.3}) is already appeared if one fixes a formal type~\cite{STWZ} (see also Remark~\ref{final}). Also, another connection between Riemann--Hilbert correspondence and holomorphic Fukaya category is conjectured by Kontsevich~\cite{KonIMJ-PRG}, whose relation to our conjecture is also of interest.

The organization of this paper is as follows: In section 2, we define and discuss the preliminary version of the category of sheaves with coefficients in $\Lambda$. In section 3 and 4, we define the (derived) category of sheaves with coefficients in $\Lambda$ over topological spaces with boundary and consider various (derived) functorial operations as in usual sheaf theory. In section 5, we define our main objects irregular constructible sheaves and again see various functorial operations. We also note that irregular constructible sheaves are actually sheaves. In section 6, we construct the functor $\frakF$ which relates irregular to usual sheaves. In section 7, we see the relationship between enhanced sheaves and our $\Lambda$-modules, which enables us to establish our version of Riemann--Hilbert correspondence using D'Agnolo--Kashiwara's theorem in section 8. We also prove some commutativity results for Riemann--Hilbert functor in section 8. In section 9, we define irregular perverse sheaves by using $\frakF$ and import results in the theory of perverse sheaves to irregular perverse sheaves. In section 10, we discuss algebraic version of the above story. In section 11, we give some discussions around Fukaya category and Riemann--Hilbert correspondence.

\section*{Acknowledgment}
The author would like to thank Andrea D'Agnolo whose lectures three times (at Kashiwa, Berkeley, and Padova) gave him many insights about irregular Riemann--Hilbert correspondence. He also kindly pointed out some mistakes in an early draft. The author also would like to thank Takahiro Saito for having many discussions on many aspects of Riemann--Hilbert correspondence (at least once a week), and Takuro Mochizuki for kindly answering some questions. This work was supported by World Premier International Research Center Initiative (WPI), MEXT, Japan and JSPS KAKENHI Grant Number JP18K13405.

\section{$\Lambda_X$-modules}
In this section, we introduce the ``finite Novikov ring'' $\Lambda$ and its modules. We fix a field $\bk\subset \bC$ once and for all.
\subsection{The ring $\Lambda$}
Let us see the set of non-negative real numbers $\bRz$ as a semigroup by the addition. We denote the associated polynomial ring by $\Lambda:=\Lambda_\bk:=\bk[\bRz]$. For $a\in \bRz$, let us denote the corresponding indeterminate by $T^a$. We set $\Gr^a\Lambda:=\bk\cdot T^a\subset \Lambda$ for $a\geq 0$, which gives an $\bR$-grading on $\Lambda$.

Let $\Mod^0(\Lambda)$ be the abelian category of $\bR$-graded $\Lambda$-modules with degree $0$ morphisms. For an $\bR$-graded $\Lambda$-module $V$, let $V\la a\ra$ be the grading shift of $M$ i.e., $\Gr^bV\la a\ra:=\Gr^{a+b} V$. We set
\begin{equation}\label{hom-space}
\Hom_{\Mod^\bR(\Lambda)}(V,W):=\bigoplus_{a\in \bR}\Hom_{\Mod^0(\Lambda)}(V, W\la a\ra)
\end{equation}
for $\bR$-graded $\Lambda$-modules. The category $\Mod^\bR(\Lambda)$ is consisting of $\bR$-graded modules with the hom-spaces defined by (\ref{hom-space}). We set
\begin{equation}
\Hom_{\Mod^\bR(\Lambda)}^a(V,W):=\Hom_{\Mod^0(\Lambda)}(V, W\la a\ra).
\end{equation}

\subsection{$\Lambda_X$-modules}
Let $X$ be a complex manifold. Let $\Lambda_X$ be the constant sheaf valued in $\Lambda$.
\begin{definition}
A sheaf of $\bR$-graded $\Lambda_X$-module is a sheaf valued in $\Mod^0(\Lambda)$. 
\end{definition}
Let $\tcV$ be a sheaf of $\bR$-graded $\Lambda_X$-modules. For an open subset $U\subset X$, we have an $\bR$-graded $\Lambda$-module $\tcV(U)$. For an inclusion $U\hookrightarrow V$, we have a map $\tcV(V)\rightarrow \tcV(U)$ which respects the grading $\Gr^a\tcV(V)\rightarrow \Gr^a\tcV(U)$. Hence we have a sheaf of $\bk$-vector spaces $\Gr^a\tcV$ and an isomorphism $\tcV\cong \bigoplus_{a}\Gr^a\tcV$ as sheaves valued in $\bk$-vector spaces.

We denote the category of $\bR$-graded $\Lambda_X$-modules by $\Mod^0(\Lambda_X)$.

\begin{proposition}
The category $\Mod^0(\Lambda_X)$ is abelian.
\end{proposition}
\begin{proof}
This is because $\Mod^0(\Lambda)$ is abelian.
\end{proof}

\begin{notation*}
$\tcV\la a\ra$ for $a\in \bR$ is $a$-shift of $\tcV$ as in the previous subsection. For $f\colon \tcV\rightarrow \tcW$, $f\la a\ra$ means the shifted morphism $\tcV\la a\ra\rightarrow \tcW\la a\ra$ 
\end{notation*}

%Let $\cF\colon \Open_X\rightarrow \Mod^0(\Lambda)$ be an $\bR$-graded $\Lambda_X$-module where $\Open_X$ be the site of open subsets of $X$. By composing the forgetful functor $[\cdot]\colon \Mod^0(\Lambda)\rightarrow \Mod^\bA(\Lambda)$, we get a sheaf of $\bA$-graded $\Lambda$-module $[\cF]:=[\cdot]\circ \cF$.

%\begin{definition}
%An $\bA$-graded $\Lambda$-module $\cV$ is {\em liftable} if there exists an $\bR$-graded $\Lambda_X$-module %$\tcV$ such that $\cV= [\tcV]$
%\end{definition}

%It is clear from the definition:
%\begin{lemma}
%For liftable $\Lambda_X$-modules $\cV=[\tcV], \cW=[\tcW]$, we have an isomorphism
%\begin{equation}
%\Hom_{\Mod^\bA(\Lambda_X)}(\cV, \cW)\cong \bigoplus_{c\in \bR}\Hom_{\Mod^0(\Lambda_X)}(\tcV, \tcW\la c\ra).
%\end{equation}
%where $\tcW\la c\ra$ is the grading shift by $c$.
%\end{lemma}

We set
\begin{equation}
\Hom_{\Mod^\bR(\Lambda_X)}(\tcV,\tcW):=\bigoplus_{a\in \bR}\Hom_{\Mod^0(\Lambda_X)}(\tcV, \tcW\la a\ra)
\end{equation}
and
\begin{equation}
\Hom^a_{\Mod^\bR(\Lambda_X)}(\tcV,\tcW):=\Hom_{\Mod^0(\Lambda_X)}(\tcV, \tcW\la a\ra).
\end{equation}
Note that $\Hom_{\Mod^\bR(\Lambda_X)}(\tcV, \tcW)$ is a $\Lambda$-module. We see $\bk$ as a $\Lambda$-module by setting $f\cdot c:=f|_{T=1}c$ for $f\in \Lambda$ and $c\in \bk$. We set
\begin{equation}
\Hom_{\ModIp(\Lambda_X)}([\tcV], [\tcW]):=\Hom_{\Mod^\bR(\Lambda_X)}(\tcV, \tcW)\otimes_\Lambda \bk.
\end{equation}

\begin{definition}
The category $\Mod_{pre}^\frakI(\Lambda_X)$ is defined by the following data: the set of objects is the set of $\bR$-graded $\Lambda_X$-modules. For an $\bR$-graded $\Lambda_X$-module $\tcV$, the corresponding object in $\ModIp(\Lambda_X)$ is denoted by $[\tcV]$.

The hom-space between $[\tcV]$ and $[\tcW]$ is $\Hom_{\ModIp(\Lambda_X)}([\tcV], [\tcW])$ defined in the above.

There is a canonical functor $[\cdot]\colon \Mod^0(\Lambda_X)\rightarrow \ModIp(\Lambda_X)$ which is the identity on objects and takes a morphism $f$ to $f\otimes 1$.
\end{definition}

\begin{definition}
For an object $\cV$ in $\ModIp(\Lambda_X)$, a lift is a pair of an object $\tcV\in \Mod^0(\Lambda_X)$ and an isomorphism $[\tcV]\xrightarrow{\cong}\cV$. In the following, we usually do not write this isomorphism explicitly for simplicity.
\end{definition}

\begin{proposition}\label{1.9}
The category $\ModIp(\Lambda_X)$ is an abelian category.
\end{proposition}

To prove this proposition, we prepare some lemmas. 

\begin{lemma}\label{1.105}
Let $V$ be an $\bR$-graded $\Lambda$-module. Let $s$ be a homogeneous element of $V$. If $T^a\cdot s \neq 0$ for any $a\in \bR\geq 0$, then $s\otimes 1$ is nonzero in $V\otimes_\Lambda \bk$.
\end{lemma}
\begin{proof}
Note that $l\cdot s$ is nonzero for any $l\in \Lambda\bs\{0\}$. We have an inclusion $\Lambda\cdot s\hookrightarrow V$.  Since the LHS is a free $\Lambda$-module, the tensoring $(-)\otimes_\Lambda \bk$ preserves the inclusion. Hence $s\otimes 1$ is nonzero in $V\otimes_\Lambda \bk$.
\end{proof}

\begin{lemma}\label{1.10}
Suppose that $\cV=[\tilde{\cV}], \cW=[\tilde{\cW}]$. For $f\in \Hom^c_{\Mod^\bR(\Lambda_X)}(\tcV,\tcW)$, if $T^af\neq 0$ for any $a\in\bRz$, then $f$ is nonzero as an element in $\Hom_{\ModIp(\Lambda_X)}(\cV,\cW)$.
\end{lemma}
\begin{proof}
This is a case of Lemma \ref{1.105} by setting $V=\Hom_{\Mod^\bR(\Lambda_X)}(\tcV, \tcW)$ and $s:=f$ 
\end{proof}

\begin{lemma}\label{1.11}Suppose that $\cV=[\tilde{\cV}], \cW=[\tilde{\cW}]$.
For $f\in \Hom_{\ModIp(\Lambda_X)}(\cV,\cW)$, there exists $b\in \bR$ such that there exists 
\begin{equation}
f'\in\Hom^b_{\Mod^\bR(\Lambda_X)}(\tcV,\tcW)
\end{equation}
which is a lift of $f$.
\end{lemma}
\begin{proof}
Take a representative $f=\bigoplus_{c}f_c\in  \bigoplus_{c}\Hom^c_{\Mod^\bR(\Lambda_X)}(\tcV,\tcW)$. Since $f_c$ is zero except for finite $c$, we can take $b$ to be a real number which is greater than or equal to the maximum of $c$ for which $f_c$ is nonzero. Then we set 
\begin{equation}
f':=\bigoplus_{c}T^{b-c}f_c\in \Hom^b_{\Mod^\bR(\Lambda_X)}(\tcV,\tcW).
\end{equation}
Since $T^{b-c}=1$ on $\Hom_{\ModIp(\Lambda_X)}(\cV,\cW)$, the element $f'$ represents $f$.
\end{proof}

\begin{lemma}\label{1.12}Suppose that $\cV=[\tilde{\cV}], \cW=[\tilde{\cW}]$.
Let $f_i\in \Hom^{b_i}_{\Mod^\bR(\Lambda_X)}(\tcV,\tcW)$ $(i=1,2)$ be lifts of $f\in \Hom_{\ModI_{pre}(\Lambda_X)}(\cV, \cW)$. Then there exists $b_i\in \bRz$ such that $T^{b_1}f_1=T^{b_2}f_2$.
\end{lemma}
\begin{proof}
By multiplying some $T^a$'s we can assume that $a_1=a_2$. Since $f_1-f_2$ represents $0$ in $\Hom_{\Mod_{pre}^{\frakI}(\Lambda_X)}(\cV,\cW)$, there exists $b\in \bRz$ such that $T^b(f_1-f_2)=0$ by Lemma~\ref{1.10}.
\end{proof}

\begin{lemma}\label{1.13}
Suppose that $\cV=[\tilde{\cV}], \cW=[\tilde{\cW}]$. For $f\in \Hom_{\ModIp(\Lambda_X)}(\cV,\cW)$, let $f'\in \Hom^{b}_{\Mod^\bR(\Lambda_X)}(\tcV,\tcW)$ be a lift. We view $f'$ as a degree $0$ morphism between $\tcV$ and $\tcW\la b\ra$ in $\Mod^0(\Lambda_X)$. The objects $[\ker(f')]$, $[\im(f')]$, $[\coker(f')]$, and $[\coim(f')]$ in $\ModIp(\Lambda_X)$ only depend on $f$.
\end{lemma}
\begin{proof}
By Lemma \ref{1.12}, it suffices to prove the objects defined for $f'$ and $T^af'$ are isomorphic. We have morphisms $f'\colon \tcV\rightarrow \tcW\la a\ra$ and $T^af'\colon \tcV\rightarrow \tcW\la a+b\ra$ in $\Mod^0(\Lambda_X)$. Note that $\ker f'\hookrightarrow \ker T^af'$. Hence for any $\tcP\in \Mod^0(\Lambda_X)$, we have
\begin{equation}
\tilde{c}\colon \bigoplus_{b\in \bR}\Hom^b_{\Mod^\bR(\Lambda_X)}(\tcP, \ker f')\hookrightarrow\bigoplus_{b\in \bR}\Hom^b_{\Mod^\bR(\Lambda_X)}(\tcP, \ker T^af'),
\end{equation}
which induces a comparison morphism $c\colon \Hom_{\ModIp(\Lambda_X)}([\tcP], [\ker f'])\rightarrow \Hom_{\ModIp(\Lambda_X)}([\tcP], [\ker T^af'])$. It suffices to show that $c$ is an isomorphism.

For any $g\in \Hom^b_{\Mod^\bR(\Lambda_X)}(\tcP, \ker T^af')$, consider $T^ag\in \Hom^{a+b}_{\Mod^\bR(\Lambda_X)}(\tcP, \ker T^af')$. Since $f'\la a+b\ra\circ T^a g=T^af'\la b\ra \circ g=0$, $T^ag$ factors through $\ker f'\la a+b\ra$ i.e. $T^ag\in \Hom^{a+b}_{\Mod^\bR(\Lambda_X)}(\tcP, \ker f')$. Hence $T^ag$ is in the image of $\tilde{c}$. Since $g$ and $T^ag$ represents the same morphism in $\ModIp(\Lambda_X)$, we have the surjectivity of $c$.

On the other hand, let $g\in \Hom_{\ModIp(\Lambda_X)}([\tcP],[\ker f'])$ be zero in $\Hom_{\ModIp(\Lambda_X)}([\tcP],[\ker T^af'])$. For a representative $g'$ of $g$, we have $T^bg'=0$ for some $b\in\bRz$ by Lemma \ref{1.10}. Hence $g=0\in \Hom_{\ModIp(\Lambda_X)}([P], [\ker f'])$. This gives the injectivity of $c$. 

Similar arguments prove the claims for $\im(f')$, $\coker(f')$, and $\coim(f')$.
\end{proof}

\begin{lemma}\label{1.15}
The objects defined in Lemma \ref{1.13} actually give kernel, image, cokernel, and coimage in $\ModIp(\Lambda_X)$.
\end{lemma}
\begin{proof}
Again, we will only prove for kernel and the others can be proved by similar arguments.

Let $\cP\xrightarrow{g} \cV\xrightarrow{f} \cW\in \ModIp(\Lambda_X)$ satisfy $f\circ g=0$. We have representatives
\begin{equation}
\tcP\xrightarrow{\tg}\tcV\xrightarrow{\tf}\tcW.
\end{equation}
in $\Mod^0(\Lambda_X)$. By replacing $\tcW$ with $\tcW\la c\ra$ with sufficiently large $c$ and $\tf$ with $T^c\tf$, we can take so that $\tf\circ \tg=0$ by Lemma \ref{1.10}. Then there exists a morhism $\tcP\rightarrow \ker \tf$ by the universality of the kernel. The commutative diagram
\begin{equation}
\xymatrix{
\ker\tf \ar[r]^{\tilde{\iota}} & \tcV \ar[r]^\tf&\tcW\\
& \tcP \ar[u]_{\tg} \ar[lu]_{\exists^!\tilde{h}} & 
}
\end{equation}
descends to the commutative diagram
\begin{equation}
\xymatrix{
[\ker\tf] \ar[r]^\iota & \cV \ar[r]^f&\cW\\
& \cP \ar[u]^g \ar[lu]_h & 
}
\end{equation}
 in $\ModIp(\Lambda_X)$, hence we only have to check the uniquness of the morphism $h$.
 
Let $h'\colon \cP\rightarrow [\ker f']$ be another morphism in $\ModIp(\Lambda_X)$ which fits into the diagram
\begin{equation}
\xymatrix{
[\ker\tf] \ar[r]^\iota & \cV \ar[r]^f&\cW\\
& \cP \ar[u]^g \ar[lu]_{h'} & 
}
\end{equation}
We can lift $h'$ to $\tilde{h}'\colon \tcP\rightarrow \ker\tilde{f}\la a\ra$ for some $a\in\bRz$. Take $b\in \bR_{\geq 0}$ so that $T^b\tilde{\iota}\la a\ra \circ \tilde{h}'=T^{a+b}\tg$ is satisfied.  Then we again get a commutative diagram.
\begin{equation}
\xymatrix{
\ker\tf\la a+b\ra \ar[r]^{\tilde{\iota}\la a+b\ra} & \tcV\la a+b\ra \ar[r]^{\tf\la a+b\ra} &\tcW\la a+b\ra\\
& \tcP \ar[u]_{T^{a+b}\tg} \ar[lu]_{T^{b}\tilde{h}'} & 
}
\end{equation}
On the other hand, we have the commutative diagram
\begin{equation}
\xymatrix{
\ker\tf\la a+b\ra \ar[r]^{\tilde{\iota}\la a+b\ra} & \tcV\la a+b\ra \ar[r]^{\tf\la a+b\ra}&\tcW\la a+b\ra\\
& \tcP \ar[u]_{T^{a+b}\tg} \ar[lu]_{T^{a+b}\tilde{h}} & 
}
\end{equation}
By the universality of $\ker\tf\la b\ra$, we have $T^{a+b}\tilde{h}=T^b\tilde{h}'$. Hence $h=h'$. This completes the proof.
\end{proof}

\begin{proof}[Proof of Proposition \ref{1.9}]
It remains to show that the isomorphism between $\im$ and $\coim$. Let $f$ be a morphism in $\ModIp(\Lambda_X)$ and $\tf$ be a lift of $f$. As shown in Lemma \ref{1.15}, $\im f$ is given by $[\im \tf]$ and $\coim f$ is given by $[\coim \tf]$. Since $\Mod^0(\Lambda_X)$ is abelian, there exists a canonical isomorphism $\im\tf\cong \coim\tf$. This also induces an isomorphism between $\im f$ and $\coim f$. This completes the proof.
\end{proof}

\begin{corollary}\label{kakkoexact}
The functor $[\cdot]\colon \Mod^0(\Lambda_X)\rightarrow \ModIp(\Lambda_X)$ is exact.
\end{corollary}
\begin{proof}
This is obvious from Lemma~\ref{1.15}.
\end{proof}

It is useful to state a kind of the converse of the above corollary.
\begin{lemma}\label{liftlemma}
Let
\begin{equation}
0\rightarrow \cV\xrightarrow{f} \cW\xrightarrow{g} \cX\rightarrow 0
\end{equation}
be an exact sequence of $\ModI_{pre}(\Lambda_X)$. Then there exists an exact sequence
\begin{equation}
0\rightarrow \tcV\xrightarrow{\tf} \tcW\xrightarrow{\tg} \tcX\rightarrow 0
\end{equation}
in $\Mod^0(\Lambda_X)$ which is a lift of the above sequence.
\end{lemma}
\begin{proof}
Take a lift $\tcV'\xrightarrow{\tf'}\tcW\xrightarrow{\tg'}\tcX'$ such that $\tg'\circ \tf'=0$. Set $\tcV:=\ker \tg'$ and $\tcX:=\image \tg'$. Then we have an exact sequence
\begin{equation}
0\rightarrow \tcV\xrightarrow{\tf} \tcW\xrightarrow{\tg} \tcX\rightarrow 0.
\end{equation}
Here $\tf$ and $\tg$ are caonical morphisms. We have an associated morphism $\tcV'\rightarrow \tcV$. In $\ModIp(\Lambda_X)$, this associates a morphism $\cV\rightarrow [\tcV]=[\ker \tg']=\ker g$. By the exactness of the given sequence, we have $\cV\xrightarrow{\cong} [\tcV]$. Hence $\tcV$ is a lift of $\cV$. In a similar way, one can see that $\tcX$ is a lift of $\cX$. This completes the proof.
\end{proof}

\section{The category $\ModI(\Lambda_{(\oX, D)})$}
In this section, we glue up $\ModIp(\Lambda_X)$ to obtain a modified category, especially for noncompact manifolds.
\subsection{Topological space with boundary}
In this paper, a {\em topological space with boundary} is a pair $(\overline{X},D_X)$ of a good topological space $\overline{X}$ with a closed subset $D_X$ of $\overline{X}$. We say $D_X$ is the boundary of $(\oX,D_X)$ and $\oX\bs D_X$ is the interior of $(\oX,D_X)$. A morphism between $(\overline{X}, D_X)$ and $(\overline{Y}, D_Y)$ is a continuous map $f$ between $\overline{X}$ and $\overline{Y}$ preserving the interiors. We denote the interiors by $X:=\oX\bs D_X$ and $Y:=\oY\bs D_Y$. We also denote the induced map between interiors by $f\colon X\rightarrow Y$ by the abuse of notation.

\begin{example}
\begin{enumerate}
\item Our primary examples of topological spaces with boundaries are of the following class: For a topological space $Z$, consider a locally closed subset $S$. Let $\overline{S}$ be the closure of $S$. Then $(\overline{S}, \overline{S}\bs S)$ is a topological space with boundary. We have a canonical map $(\oS, \oS\bs S)\rightarrow (Z,\varnothing)$ induced by the inclusion $\oS\hookrightarrow Z$.
\item By the definition of morphisms of topological spaces with boundary, we have canonical maps $(X,\varnothing)\rightarrow (\oX, D_X)$ and $\dX\rightarrow (\oX, \varnothing)$ induced by the identity $\id\colon \oX\rightarrow \oX$. On the other hand, such a canonical map does not exist from $\dX$ to $(X, \varnothing)$.
\end{enumerate}
\end{example}

Let $(\oX, D_X)$ be a topological space with boundary. The site $\Open_{(\oX, D_X)}$ is defined by the following data: the underlying category is the category of open subsets of $\oX\bs D_X$, a collection of open subsets $\{U_i\}_{i\in I}$ in $\oX\bs D_X$ is said to define a cover of $U$ if there exists a subset $J$ of $I$ such that the subcollection $\{U_i\}_{i\in J}$ still defines an open covering of $U$ and is locally finite over $\oX$. The following is clear:
\begin{lemma}
This cover gives a Grothendieck topology on $\Open_{(\oX, D_X)}$.
\end{lemma}
\begin{remark}
If $\oX$ is compact and $D_X=\varnothing$ then $\Open_{(\oX,D_X)}$ coincides with the usual site of $\oX$. 
\end{remark}

\begin{lemma}
Let $f\colon (\oX, D_X)\rightarrow (\oY, D_Y)$ is a morphism between topological spaces with boundary. Then there exists an induced morphism $\Open_{(\oY, D_Y)}\rightarrow \Open_{(\oX, D_X)}$.
\end{lemma}
\begin{proof}
Let $\{U_i\}_{i\in I}$ be a cover of $U$ in $\Open_{(\oY, D_Y)}$.  Let $J\subset I$ be as in the definition of the cover.  Then $\{f^{-1}(U_i)\}$ is an open covering of $f^{-1}(U)$ in $\oX$. Let $x\in X$ Take $x\in \oX$, then there exists a small neighborhood $V$ of $f(y)$ such that $V$ only intersects with a finite subset of $\{U_j\}_{j\in J}$. Then $f^{-1}(V)$ also only intersects with a finite subset of $\{f^{-1}(U_j)\}_{j\in J}$. Hence $\{f^{-1}(U_i)\}_{i\in I}$ is a cover of $f^{-1}(U)$ in $\Open_{(\oX, D_X)}$.
\end{proof}

\subsection{The category $\ModI(\Lambda_{(\oX, D_X)})$}
Let $(\oX, D_X)$ be a topological space with boundary. We set $X:=\oX\bs D_X$. Let $U\supset V$ be open subsets of $X$. Then we have a restriction functor
\begin{equation}
\Mod_{pre}^\frakI(\Lambda_U)\rightarrow \Mod_{pre}^\frakI(\Lambda_V).
\end{equation}

\begin{lemma}\label{restrictionexact}
This restriction functor is exact.
\end{lemma}
\begin{proof}
A short exact sequence in $\ModIp(\Lambda_U)$ can be lifted to an short exact sequence in $\Mod^0(\Lambda_U)$ by Corollary~\ref{kakkoexact}. Then we can restrict it to an exact sequence in $\Mod^0(\Lambda_V)$. By Lemma~\ref{1.15}, this also gives an exact sequence in $\ModIp(\Lambda_V)$.
\end{proof}

These maps form a presheaf of categories over the site $\Open_{(\oX,D_X)}$. This is not always a stack (even a prestack) because the tensor product $\otimes \bk$ on the hom-space breaks the sheaf property.

Take the stackfication (resp. prestackification) of this stack with respect to $\Open_{(\oX,D_X)}$. We denote it by $\ModI_{(\oX,D_X)}$ (resp. $\ModIps_{\dX}$). See Appendix for a short exposition of stackification.

\begin{definition}
The global section category of $\ModI_{(\oX,D_X)}$ is denoted by $\Mod^\frakI(\Lambda_\dX)$. For a manifold $X$, we set $\ModI(\Lambda_X):=\ModI(\Lambda_{(X, \varnothing)})$.
%We call an object of $\Mod^\frakI(\Lambda_X)$ $\bA$-graded locally liftable $\Lambda_X$-module.
\end{definition}

\begin{proposition}\label{ModIabel}
The category $\ModI_{(\oX, D_X)}(U)$ is an abelian category for any $U\in \Open_{\dX}$.
\end{proposition}
\begin{proof}
We will only consider about kernels. The similar argument holds for cokernels, images and coimages.

Let $f\colon \cV\rightarrow \cW$ be a morphism in $\ModI_\dX(U)$. Then there exists a covering $\{U_i\}$ of $U$ such that we have a descent data $f_i\colon \cV_i\rightarrow \cW_i$ in $\ModIps_\dX(U_i)$. If it is necessary, we can replace the covering with a finer covering so that each $f|_{U_i}$ is represented by a morphism $f_{i}\colon \cV_i\rightarrow \cW_i$ in $\ModIp(U_{i})$. On each intersection $U_{i}\cap U_{j}$, we have a further covering $\{U_{ijk}\}_k$ such that $(f_i-f_j)|_{U_{ijk}}=0$.

Then we have $\ker (f_{i})$ since $\ModIp(\Lambda_{U_{i}})$ is an abelian category. Since the restriction functors are exact (Lemma~\ref{restrictionexact}), we have $\ker (f_i|_{U_{ij}})|_{U_{ijk}}=\ker(f_i|_{U_{ijk}})=\ker(f_j|_{U_{ijk}})=\ker(f_j|_{U_{ij}})|_{U_{ijk}}$ in $\ModIp(\Lambda_{U_{ijk}})$. Hence we have $\ker(f_i)|_{U_{ij}}=\ker(f_i|_{U_{ij}})=\ker(f_j|_{U_{ij}})=\ker(f_j)|_{U_{ij}}$ in $\ModI_{\dX}(U_{ij})$. This further gives a descend data and glues up to an object $\cK\in \ModI_\dX(U)$.

For a morphism $g\colon \cX\rightarrow \cV$ with $f\circ g=0$, by taking a sufficiently fine cover $\{U_i\}$, we can represent them in $\ModIp(\Lambda_{U_i})$. Then one get a unique factorizing morphism $\cX|_{U_i}\rightarrow \ker f_i$. Again by taking a finer covering as in the previous part of the proof and the universality, the set of these factorizing morphisms gives a descent data and can be glued up into the unique factorizing $\cX\rightarrow \cK$. This shows $\cK$ is $\ker f$.
\end{proof}

Let $U$ be an open subset of $X$. Let $\alpha_U\colon \Mod^\frakI_{pre}(\Lambda_U)\rightarrow \Mod^\frakI_{(\oX,D)}(U)$ be the canonical functor.
\begin{lemma}\label{alphaexact}
The functor $\alpha_U$ is an exact functor.
\end{lemma}
\begin{proof}
Since kernels, cokernels, images, and coimages are defined locally, the assertion is obvious.
\end{proof}

\begin{lemma}\label{realizationofmorphism}If $\overline{U}$ is compact, the functor $\alpha_U$ is fully faithful.
\end{lemma}
\begin{proof}
We set $D_U:=D_X\cap \overline{U}$.
To show the claim, it is enough to prove $\Hom_{\Mod^\frakI_{pre}(\Lambda_U)}(\cV, \cW)$ is a sheaf over the site $\Open_{(\overline{U},D_U)}$. Since $\overline{U}$ is compact, any cover in $\Open_{(\overline{U}, D_U)}$ has a finite subcover.

We first assume that there exists a finite cover $\{U_i\}$ of $U$ such that the restriction of $f\in \Hom_{\Mod^\frakI_{pre}(\Lambda_X)}(\cV, \cW)$ to each open subset is zero. Let $\tf\in \Hom_{\Mod^0(\Lambda_X)}(\tcV, \tcW)$ be a representative. Then the restriction of $f$ to each open subset $U_i$ is represented by $\tf|_{U_i}$. Since $f|_{U_i}=0$, there exists big $T^a$ such that $T^a\tf|_{U_i}=0$ by Lemma~\ref{1.10}. Let $A$ be the maximum of those $a$'s. Then $T^A\tf=0$. Hence $f=0$.

Let $\{f_{i}\}\in \prod\Hom_{\Mod^\frakI_{pre}(\Lambda_{U_i})}(\cV|_{U_i}, \cW|_{U_i})$ satisfies the descent condition. Depending on $i$, we have a set of lifts $\tf_i\colon \tcV|_{U_i}\rightarrow \tcW|_{U_i}\la a_i\ra$ in $\Mod^0(\Lambda_{U_i})$. In our situation, we can take $a_i=a_j$ for any $i,j$, since the indexes are finite. Then we reset $\tcW$ with $\tcW\la a_i\ra$ by Lemma~\ref{1.15}. 
On $U_i\cap U_j$, $\tf_i$ and $\tf_j$ may not coincide, but $f_i=[\tf_i]$ coincides with $f_j=[\tf_j]$. Hence there exists $T^{a_{ij}}$ such that $T^{a_{ij}}\tf_i=T^{a_{ij}}\tf_j$. By taking the maximum among $a_{ij}$, we reset $\tf_j$ by $T^a \tf_j$ then the set $\{T^a \tf_i\}$ satisfies the descent condition in $\Mod^0(\Lambda_U)$. Hence we get a glued morphism.
\end{proof}

\begin{corollary}\label{globalobject}
If $X$ is compact, $\ModIp(\Lambda_X)$ is an abelian subcategory of $\ModI(\Lambda_{X})$.
\end{corollary}
\begin{proof}
The embedding is given by the proof of Lemma~\ref{realizationofmorphism}. 
\end{proof}

Let us denote the derived category by $D^{\bullet}(\ModI(\Lambda_{(\oX, D_X)}))$ ($\bullet=b, \pm$).

\subsection{Forgetting shifts}
Recall that we have the canonical functor $[\cdot]\colon\Mod^0(\Lambda_X)\rightarrow \ModIp(\Lambda_X)$.  There also exists a canonical functor $\alpha_X\circ [\cdot]\colon \Mod^0(\Lambda_X)\rightarrow \ModI(\Lambda_\dX)$.

\begin{lemma}\label{cdotexact}
The functor $\alpha_X\circ [\cdot]\colon \Mod^0(\Lambda)\rightarrow \ModI(\Lambda_\dX)$ is an exact functor.
\end{lemma}
\begin{proof}
This is a consequence of Lemma \ref{alphaexact} and Lemma~\ref{kakkoexact}.
\end{proof}
For simplicity, we will denote $\alpha_X\circ [\cdot]$ by $[\cdot]$. We denote the exact functor $D^\bullet(\Mod^0(\Lambda_X))\rightarrow D^\bullet(\ModI(\Lambda_\dX))$ induced by $[\cdot]$ by the same notation.

\subsection{Finite limits and finite colimits}
Since $\ModI(\Lambda_\dX)$ is an abelian category, it admits finite limits and finite colimits.
\begin{lemma}\label{limits}
Let $F\colon \frakU\rightarrow \Mod^0(\Lambda_X)$ be a finite diagram without loops in $\Mod^0(\Lambda_X)$. We have
\begin{equation}
\begin{split}
[\lim_{\substack{\longleftarrow \\ \frakU}}F]&\cong \lim_{\substack{\longleftarrow \\ \frakU}}[F]\\
[\lim_{\substack{\longrightarrow \\ \frakU}}F]&\cong \lim_{\substack{\longrightarrow \\ \frakU}}[F]
\end{split}
\end{equation}
\end{lemma}
\begin{proof}
We will only prove the first one. The second one can be proved in a similar manner.

It is enough to show that the left hand side satisfies the universality of the right hand side. Let $\cV$ be an object which is over $[F]$. Locally, we have a lift $\tcV\rightarrow F$. By the universality, we get a morphism $\tcV\rightarrow \displaystyle{\lim_{\substack{\longleftarrow \\ \frakU}}}F$, which induces a morphism $\cV\rightarrow [\displaystyle{\lim_{\substack{\longleftarrow \\ \frakU}}}F]$ locally. The uniqueness of this morphism can be shown by a method similar to the proof of Lemma \ref{1.15}. The uniqueness glue up these local morphisms to obtain the desired result.
\end{proof}

\begin{remark}
Contrary to the finite case, infinite (co)limits do not commute with $[\cdot]$ in general. We give one example in the following. Let us set $\oX=[0,\infty)$ and $D_X=\{0\}$. Consider $\cV_1:=\bigoplus_{a\in \bR}\RGamma_{[-a,\infty)}\bk_{\{t\geq 1/x\}}$ and $\cV_2:=\bigoplus_{a\in \bR}\RGamma_{[-a,\infty)}\bk_{\{t\geq 1/x^2\}}$. As we can see in the discussion of Section 5 below, we have $\Hom_{\ModI(\Lambda_\dX)}([\cV_2], 
[\cV_1])=0$. Let $i_b\colon (b,\infty)\hookrightarrow [0, \infty)$ be the open embedding for $b\in \bR_{>0}$.

If $[\cdot]$ and colimits commute, we have $[\cV_2]\cong \displaystyle{\lim_{\substack{\longrightarrow \\ b\rightarrow 0}}}[i_{b!}i_b^{-1}\cV_2]$. Again, from the discussion of Section 5, we can conclude
\begin{equation}
\Hom_{\ModI(\Lambda_\dX)}([\cV_2], [\cV_1])\cong \lim_{\substack{\longleftarrow \\ b\rightarrow 0}}\Hom_{\ModI(\Lambda_\dX)}(\cV_2|_{(b.\infty)}, \cV_1|_{(b,\infty)})\cong \bk.
\end{equation}
This is a contradiction.
\end{remark}

\subsection{Operations}
In this section, we will develop the six functors. As above, $(\oX, D_X), (\oY, D_Y)$ are topological spaces with boundaries. We set $X:=\oX\bs D_X$ and $Y:=\oY\bs D_Y$

\subsection*{Internal hom} 
Let $\tcV, \tcW$ be objects in $\Mod^0(\Lambda_X)$, the internal hom sheaf is defined by the assignment
\begin{equation}
\cHom(\tcV, \tcW)\colon U\mapsto \bigoplus_{a}\Hom_{\Mod^0(\Lambda_U)}(\cV|_U, \cW|_U\la a\ra).
\end{equation}
for an open subset $U\subset X$. This is canonically an $\bR$-graded $\Lambda_X$-module.

Let $\cV, \cW$ be objects of $\ModI(\Lambda_{(\oX,D_X)})$. Then there exists an open covering $\{U_i\}_{i\in I}$ of $X$ in the site $\Open_{(\oX,D_X)}$ such that there exists $\bR$-graded $\Lambda_{U_i}$-modules $\tcV_i, \tcW_i$ over each $U_i$ representing $\cV$ and $\cW$. Then one has an $\bR$-graded $\Lambda_{U_i}$-module $\cHom(\tcV_i, \tcW_i)$ over each $U_i$. 
\begin{lemma}\label{internalhom}
The set $\{[\cHom(\tcV_i, \tcW_i)]\}$ satisfies the descent and gives an object of $\ModI(\Lambda_{(\oX, D_X})$, which we will denote by $\cHom(\cV, \cW)$. This is independent of the choice of local lifts.
\end{lemma}
\begin{proof}
On $U_{ij}:=U_i\cap U_j$, we have the isomorphism $f\colon [\tcV_i|_{U_{ij}}]\rightarrow [\tcV_j|_{U_{ij}}]$ in $\ModIps_\dX(U_{ij})$. Then there exists an open covering $\{U_{ijk}\}$ of $U_{ij}$ where there exists a descent data $f_{ijk}\colon [\tcV_i|_{U_{ijk}}]\rightarrow [\tcV_j|_{U_{ijk}}]$  for the isomorphism $f\colon [\tcV_i|_{U_{ij}}] \rightarrow [\tcV_j|_{U_{ij}}]$.

We can take a lift $\tf_{ijk}\colon \tcV_i|_{U_{ijk}}\rightarrow \tcV_j|_{U_{ijk}}\la a\ra$ of this morphism and that of the inverse $\tg_{ijk}\colon \tcV_j|_{U_{ijk}}\rightarrow \tcV_i|_{U_{ijk}}\la b\ra$ for some $a, b$. The difference $\tg_{ijk}\la a\ra \circ \tf_{ijk}-T^{a+b}$ and $\tf_{ijk}\la b\ra\circ \tg_{ijk}-T^{a+b}$ becomes $0$ after multiplying $T^c$ for sufficiently big $c$ by Lemma \ref{1.10}. 

For simplicity, let us assume that $\cW$ has a global lift i.e., $\cW=[\tcW]$ although we can do the same for general $\cW$. We have the following induced morphisms:
\begin{equation}
\cHom(\tcV_i|_{U_{ijk}}, \tcW|_{U_{ijk}})\xrightarrow{p(\tg_{ijk})} \cHom(\tcV_j|_{U_{ijk}}, \tcW|_{U_{ijk}})\la -b\ra, \cHom(\tcV_j|_{U_{ijk}}, \tcW|_{U_{ijk}}), \xrightarrow{p(\tf_{ijk})} \cHom(\tcV_i|_{U_{ijk}}, \tcW|_{U_{ijk}})\la -a\ra
\end{equation} where $p(\tf_{ijk})$ and $p(\tg_{ijk})$ are precompositions of $f$ and $g$. Since $p(\tf_{ijk})\la a\ra\circ p(\tg_{ijk})-T^{a+b}\id$ and $p(\tg_{ijk})\circ p(\tf_{ijk})\la b\ra-T^{a+b}\id$ are also vanished by multiplying $T^c$ for big $c$, we can conclude $[\cHom(\tcV_i|_{U_{ijk}}, \tcW|_{U_{ijk}})]\cong [\cHom(\tcV_j|_{U_{ijk}}, \tcW|_{U_{ijk}})]$.  A similar argument as was done in Proposition \ref{ModIabel} gives a gluing of these isomorphisms to give a global object in $\ModI(\Lambda_X)$. The independence of the choice of local lifts is also clear.
\end{proof}

For $\cV \in \ModI(\Lambda_\dX)$, suppose $\cV|_U$ for $U\in \Open_\dX$ is represented by $\tcV_U\in \ModIp(\Lambda_X)$. Let us consider the assignment
\begin{equation}
V\mapsto \tcV_U(V)\otimes_\Lambda \bk.
\end{equation}
for $V\subset U$. Then this assignment does not depend on the choice of the lift $\tcV_U$. Hence one can associate a sheaf over $\Open_\dX$. We write it $\cV\otimes \bk$.

\begin{lemma}\label{globalsection}
For $\cV, \cW\in \ModI(\Lambda_\dX)$, the sheaf $\cHom(\cV,\cW)\otimes \bk$ over $\Open_\dX$ is canonically isomorphic to $\Hom_{\ModI_{\Lambda_\dX}}(\cV, \cW)$.
\end{lemma}
\begin{proof}
This is obvious from the construction.
\end{proof}

\subsection*{Tensor product}
First, for $\bR$-graded $\Lambda$-modules $V$ and $W$, their tensor product is defined as follows:
\begin{equation}
\Gr^a(V\otimes_\Lambda W):=\bigoplus_{b+c=a}\Gr^bV\otimes_k \Gr^cW/\sim
\end{equation}
where the equivalence relations are generated by 
\begin{equation}
v\otimes_k \alpha w\sim \alpha v\otimes_k w
\end{equation}
where $\alpha\in \Lambda, v\in V, w\in W$ are homogeneous and $\deg \alpha+\deg v+\deg w=a$. The tensor product $V\otimes_\Lambda W:=\bigoplus_a \Gr^a(V\otimes_\Lambda W)$ is canonically equipped with a $\Lambda$-module structure.

Let $\tcV, \tcW$ be objects of $\Mod^0(\Lambda_X)$. Then the tensor product $\tcV\otimes \tcW$ is defined by the sheafification of the assignment
\begin{equation}
\tcV\otimes_{\Lambda_X} \tcW\colon U\mapsto \tcV(U)\otimes_\Lambda\tcW(U).
\end{equation}

Let $\cV, \cW$ be objects of $\ModI(\Lambda_{(\oX, D_X)})$. Then there exists an open covering $\{U_i\}_{i\in I}$ of $X$ such that there exists $\bR$-graded $\Lambda_{U_i}$-modules $\tcV_i, \tcW_i$ over each $U_i$. Then one has another $\bR$-graded $\Lambda_{U_i}$-module $\tcV_i\otimes_{\Lambda_X} \tcW_i$ over each $U_i$.
\begin{lemma}
The set $\{[\tcV_i\otimes_{\Lambda_X}\tcW_i]\}$ satisfies the descent and gives an object of $\ModI(\Lambda_\dX)$, which we will denote by $\cV\otimes\cW$. This is independent of the choice of local lifts.
\end{lemma}
\begin{proof}
We can prove in a similar manner as in the proof of Lemma~\ref{internalhom}. 
\end{proof}

\subsection*{Tensor-Hom adjunction}
\begin{proposition}\label{tensorhomadjunction}
For $\cV, \cW, \cX\in \ModI(\Lambda_\dX)$, we have the following:
\begin{equation}\label{eqn3.5}
\Hom_{\ModI(\Lambda_\dX)}(\cV\otimes \cW, \cX)\cong \Hom_{\ModI(\Lambda_\dX)}(\cV, \cHom(\cW,\cX)).
\end{equation}
\end{proposition}
\begin{proof}
Let $\{U_i\}$ be a covering of $X$ such that we have lifts of $\cV, \cW, \cX$ over each open cover. Let us denote the liftings by $\tcV_i, \tcW_i, \tcX_i$. Then we have a canonical isomorphism
\begin{equation}
\cHom(\tcV_i\otimes \tcW_i, \tcX_i)\cong \cHom(\tcV_i, \cHom(\tcW_i,\tcX_i)).
\end{equation}
By tensoring $\bk$ and taking sheafification over $\Open_{(\overline{U_i}, \overline{U_i}\cap D_X)}$, we have an isomorphism
\begin{equation}
\Hom_{\ModI_\dX}(\cV\otimes \cW, \cX)\cong \Hom_{\ModI_\dX}(\cV,\cHom(\cW, \cX))
\end{equation}
as a sheaf over  $\Open_{(\overline{U_i}, \overline{U_i}\cap D_X)}$. The isomorphisms over $U_i$'s are glued up and give a desired result.
\end{proof}

\begin{corollary}
In the same setting as above, we have the following:
\begin{equation}
\cHom(\cV\otimes \cW, \cX)\cong \cHom(\cV, \cHom(\cW,\cX)).
\end{equation}
\end{corollary}
\begin{proof}
This is clear from the above proposition.
\end{proof}

\subsection*{Push-forward}
We will define push-forwards for a class of morphisms.

In the following, we only consider the following class of maps:
\begin{definition}
We say a morphism $f\colon \dX\rightarrow \dY$ is tame if the underlying map $f\colon \oX\rightarrow \oY$ is proper.
\end{definition}

\begin{remark}
For a locally closed subset $U\subset X$, a canonical morphism $(U, \varnothing)\rightarrow (X, \varnothing)$ is not tame in general. However $(\overline{U}, \overline{U}\bs U)\rightarrow (X, \varnothing)$ is tame. In this sense, we will consider the latter one as a standard inclusion morphism.
\end{remark}

Let $\cV$ be an object of $\ModI(\Lambda_{(\oX,D_X)})$ and $f$ be a tame map $(\oX, D_X)\rightarrow (\oY, D_Y)$. We first assume that $\cV$ has $\tcV$ with $[\tcV]\cong \cV$. In this case, we simply set
\begin{equation}
f_*\cV:=[f_*\tcV].
\end{equation}
where push-forward of $\bR$-graded $\Lambda_X$-module $\tcV$ is defined by $\Gr^af_*\tcV:=f_*\Gr^a\tcV$.
\begin{lemma}\label{well-defpush}
This is well-defined.
\end{lemma}
\begin{proof}
Let $\tcV'$ be another representative. Take a covering $\{U_i\}$ of $\Open_\dX$ such that we have lifts of the isomorphisms $\tg_i\colon \tcV|_{U_i}\rightarrow \tcV'\la a_i\ra|_{U_i}$ and $\tilh_i\colon \tcV'|_{U_i}\rightarrow \tcV|_{U_i}\la b\ra$ over each $U_i$. 
 Hence $\tg_i\la b\ra\circ \tilh_i-T^{a+b}$ and $\tilh_i\la a\ra\circ \tg_i-T^{a+b}$ are vanished by large $T^c$. By pushing forward these equations, we have
\begin{equation}
0=f_*(T^c(\tg\la b\ra\circ\tilh-T^{a+b}\id_{\tcV}))=T^c(f_*\tg\la b\ra\circ f_*\tilh-T^{a+b}\id_{f_*\tcV}).
\end{equation}
Hence we have $[f_*\iota_{i*}\tcV|_{U_i}]\cong [f_*\iota_{i*}\tcV'|_{U_i}]$ where $\iota_i\colon U_i\rightarrow X$ is the inclusion map.

Let $\frakU$ be the Cech nerve of $\{U_i\}$ and $\iota_U$ for $U\in \frakU$ is the inclusion map. Since $f$ is tame, $\{f(U_i)\}$ is locally finite in $\oY$ i.e., there exists a covering of $Y$ in $\Open_\dY$ such that there are only finite $U_i$'s in each open subset. Hence we have
\begin{equation}
\lim_{\substack{\longleftarrow \\ _{U\in \frakU}}}[f_*\iota_{U*}\tcV_U]\cong [\lim_{\substack{\longleftarrow \\ _{U\in \frakU}}}f_*\iota_{U*}\tcV_U]\cong [f_*\tcV]
\end{equation}
by Lemma \ref{limits}. Combining with the first part of the proof, we get an isomorphism $[f_*\tcV]\cong [f_*\tcV']$.
\end{proof}

Let $\cV$ be an object of $\ModI(\Lambda_\dX)$. If $f\colon \dX\rightarrow \dY$ is tame, there exists a covering $\{V_i\}$ of $Y$ and a finite cover $\{U_{ij}\}$ of each $f^{-1}(V_i)$ with lifts $\tcV_{ij}$ of $\cV$.

Let $\frakU_i$ be the Cech nerve of $\{U_{ij}\}$. We set
\begin{equation}
(f_*\cV)_i:=\lim_{\substack{\longleftarrow \\ _{U\in \frakU_i}}}[f_*\iota_{U*}\tcV_U].
\end{equation}

\begin{lemma}\label{welldefpush}
The collection $\{(f_*\cV)_i\}$ gives an object of $\ModI(\Lambda_Y)$, denoted by $f_*\cV$. Moreover it does not depend on the choice of coordinates and lifts.
\end{lemma}
\begin{proof}
It can be proved by a similar argument as in Lemma \ref{well-defpush}.
\end{proof}

\subsection*{Pull-back}
Let $\cV$ be an object of $\ModI(\Lambda_{(\oX,D_X)})$. There exists an open covering $\{U_i\}$ of $X$ with $\bR$-graded $\Lambda_{U_i}$-module $\tcV_i$ over each $U_i$. Over $f^{-1}(U_i)$, we assign a sheaf $[f^{-1}\tcV_i]$ and these can glue up together. We will denote the resulting object by $f^{-1}\cV$.

\subsection*{Push-Pull adjunction}
Let $\cV$ be an object of $\ModI(\Lambda_{(\oX,D_X)})$, $\cW$ be an object of $\ModI(\Lambda_{(\oY, D_Y)})$, and $f\colon \dX\rightarrow \dY$ be a tame morphism.

\begin{lemma}\label{pushpulladjunction}
We have the following natural isomorphism:
\begin{equation}
f_*\cHom(f^{-1}\cV, \cW)\simeq \cHom(\cV, f_*\cW).
\end{equation}
\end{lemma}
\begin{proof}
First let us take an open covering $\{U_i\}$ of $Y$ with local lifts $\{\tcV_i\}$. It is enough to prove the statement over each $U_i$. There exists a finite covering $\{V_j\}$ of $f^{-1}(U_i)$ with lifts $\{\tcW_j\}$. Then $\cHom(f^{-1}\cV, \cW)$ is represented by $\{\cHom(f^{-1}\tcV_i|_{V_j}, \tcW_j)\}$.

Let $\frakV$ be the Cech nerve of $\{V_j\}$. By the definition of the push forward, we have
\begin{equation}
f_*\cHom(f^{-1}\cV, \cW)|_{U_i}\simeq \lim_{\substack{\longleftarrow \\ V\in \frakV}}([f_*i_{V*}\cHom((f^{-1}\tcV_i)|_{V}, \tcW|_V)])
\end{equation}
Here $\tcW|_V$ means $\tcW_i|_V$ for some $V\subset V_i$.
We also have 
\begin{equation}
\begin{split}
f_*i_{V*}\cHom((f^{-1}\tcV_i)|_{V}, \tcW|_V)&\simeq f_*\cHom(f^{-1}\tcV_i, i_{V*}\tcW|_V)\\
&\simeq \cHom(\tcV_i, f_*\iota_{V*}\tcW|_V)
\end{split}
\end{equation}
for $V\in \frakV$. Hence
\begin{equation}
\begin{split}
\lim_{\substack{\longleftarrow \\ \frakV}}([f_*i_{V*}\cHom((f^{-1}\tcV_i)|_{V}, \tcW|_V)])&\simeq 
\lim_{\substack{\longleftarrow \\ \frakV}} \cHom([\tcV_i], [f_*\iota_{V*}\tcW|_V])\\
&\simeq \cHom([\tcV_i], \lim_{\substack{\longleftarrow \\ \frakV}} [f_*\iota_{V*}\tcW|_V])\\
&\simeq  \cHom([\tcV_i], [f_*(\tcW|_{f^{-1}(U_i)})])\\
&\simeq \cHom(\cV, f_*\cW)|_{U_i}.
\end{split}
\end{equation}
This completes the proof.
\end{proof}

\begin{lemma}\label{tensorhomadjunction2}
Assume the same setting as above. Then
\begin{equation}
\Hom_{\ModI(\Lambda_\dY)}(f^{-1}\cV,\cW)\cong \Hom_{\ModI(\Lambda_\dX)}(\cV, f_*\cW).
\end{equation}
\end{lemma}
\begin{proof}
Taking $\otimes \bk$ and the global sections (as in the paragraph above Lemma~\ref{globalsection}) of both sides of Lemma~\ref{pushpulladjunction}, the right hand side becomes $\Hom_{\ModI(\Lambda_\dX)}(\cV, f_*\cW)$ and the left hand side becomes
\begin{equation}
\begin{split}
(f_*\cHom(f^{-1}\cV, \cW) \otimes \bk) (Y)&\cong f_*(\cHom(f^{-1}\cV,\cW)\otimes \bk) (Y)\\
&\cong \Hom_{\ModI_\dX}(f^{-1}\cV, \cW).
\end{split}
\end{equation}
This completes the proof.
\end{proof}

\subsection*{Proper push-forwards}
Let $\cV$ be an object of $\ModI(\Lambda_{(\oX,D_X)})$ and $f$ be a map $(\oX, D_X)\rightarrow (\oY, D_Y)$. We first assume that $\cV$ has $\tcV$ with $[\tcV]\cong \cV$. In this case, we simply set
\begin{equation}
f_!\cV:=[f_!\tcV].
\end{equation}
\begin{lemma}
This is well-defined.
\end{lemma}
\begin{proof}
This can be proved in the same way as the proof of Lemma \ref{well-defpush}.
\end{proof}

Again by the same construction as in the case of push-forwards, we can define $f_!\cV$ in general under the assumption of the tameness.

\begin{assumption}
In the following, when we consider $f_*$ or $f_!$, we always assume the tameness of $f$.
\end{assumption}

\section{Derived category of $\ModI(\Lambda_\dX)$}
In this section, we develop fundamentals about derived operations for $\ModI(\Lambda_\dX)$.

\subsection{Injectives and flats}
\subsection*{Injectives}
\begin{lemma}\label{skyinj}
Let $\cF$ be an $\bR$-graded $\Lambda$-module. For $x\in X$, the skyscraper sheaf $[\cF_x]$ is an injective object. Moreover, the product $[\prod_{x\in V}\cF_x]$ for a subset $V\subset X$ is also an injective object.
\end{lemma}
\begin{proof}
The first part is almost trivial. Let us prove the second part.

Let $0\rightarrow \cV\xrightarrow{f} \cW$ be  an injection in the category $\ModI(\Lambda_\dX)$ with a map $\cV\xrightarrow{g} [\prod_{x\in X}\cF_x]$. Let us take a locally finite covering $\{U_i\}$ of $X$ with lifts $0\rightarrow \tcV_i\xrightarrow{\tf_i} \tcW_i$ and $\tcV_i\xrightarrow{\tg_i} \prod_{x\in U_i}\cF_x\la a_i\ra$. We also get a lift $\tcW_i\xrightarrow {\tilde{h}_i}\prod_{x\in U}\cF_x\la a_i\ra$.

For each $x\in V$, we choose $i_x$ from finite candidates of $i$'s satisfying $x\in U_i$. We set $\tcW_x:=(\tcW_{i_x})_x$. Over each $U_i$, the morphism $\tilde{h}_i$ gives an element of $(\bigoplus_{a\in \bR}\prod_{x\in U_i}\Hom^{a}(\tcW_{x}, \cF_x))\otimes_\Lambda k\cong \Hom(W|_{U_i}, [\prod_{x\in V}\cF_x])$ which is zero on $i_x\neq i$. Then they are trivially glued up to give a desired lift of $g$.
\end{proof}

\begin{comment}
\begin{lemma}\label{injlift}
Let $\tcI$ be an injective object in $\Mod^0(\Lambda_X)$. Then $[\tcI]$ is an injective object in $\ModI(\Lambda_X)$.
\end{lemma}
\begin{proof}
Take a diagram
\begin{equation}
\xymatrix{
\cW \ar[r]^f\ar[d]_{g}&\cX\\
[\tcI].&
}
\end{equation}
Take a lift of the above diagram
\begin{equation}
\xymatrix{
\tcW \ar[d]_{\tg}\ar[r]^{\tf}&\tcX\\
\tcI. &
}
\end{equation}
Since $\iota_{i!}\tcI_i$ is an injective object, we get a desired lift of $\tg$:
\begin{equation}
\xymatrix{
\tcW \ar[d]_{\tg}\ar[r]^{\tf}&\tcX\ar@{-->}[dl]^{\exists}\\
\tcI. &
}
\end{equation}
This induces a lift of $g$. which completes the proof.
\end{proof}
\end{comment}

\begin{proposition}\label{injectives}
The category $\ModI(\Lambda_{(\oX,D_X)})$ has enough injectives.
\end{proposition}
\begin{proof}
Take $\cV\in \ModI(\Lambda_{(\oX,D_X)})$. Then there exists a locally finite covering $\{U_i\}$ of $X$ with lifting $\tcV_i$. As usual, one can embed $\tcV_i$ to an injective object $\tcI_i$  which is a product of skyscraper sheaves.

Hence we have the inclusion $[\tcV_i]\hookrightarrow [\tcI_i]$. This induces the inclusion $\cV\hookrightarrow \bigoplus[\iota_{i*}\tcI_i]$, where the latter is a locally finite direct sum hence it exits. By Lemma \ref{skyinj}, $\bigoplus[\iota_{i*}\tcI_i]$ is also an injective object. This completes the proof.
\end{proof}

The above proof also shows the following:
\begin{corollary}\label{liftresol}
For $\tcV\in \Mod^0(\Lambda_X)$, there exists an injective resolution $\tcI^\bullet:=\tcI^0\rightarrow \tcI^1\rightarrow \cdots$ of $\tcV$ giving an injective resolution $[\tcI^\bullet]$ of $[\tcV]$.
\end{corollary}
\begin{proof}
This follows from that $[\cdot]$ is an exact functor (Lemma \ref{cdotexact}) and Lemma \ref{skyinj}.
\end{proof}

\newcommand{\tcF}{\tilde{\cF}}

\subsection*{Flats}
\begin{lemma}\label{liftflat}
Let $\tcF$ be a flat $\bR$-graded $\Lambda$-module. Then $[\tcF]$ is a flat object.
\end{lemma}
\begin{proof}
Let $\cV\rightarrow \cW\in \ModI(\Lambda_\dX)$ be an injection. Let us take an open covering $\{U_i\}$ of $X$ with representatives $\{\tcV_i\}$, $\{\tcW_i\}$ and $\tf_i\colon \tcV_i\rightarrow \tcW_i$. Here one can take $\tf_i$ as an injection by Lemma \ref{liftlemma}. Then $\cV\otimes [\tcF]$ (resp. $\cW\otimes [\tcF]$) is represented by $\tcV_i\otimes \tcF|_{U_i}\rightarrow \tcW_i\otimes \tcF|_{U_i}$, which is an injection. Then by Lemma \ref{1.15}, the morphism $f\otimes [\tcF]$ is also an injection. This completes the proof.
\end{proof}

\begin{proposition}
The category $\ModI(\Lambda_{\dX})$ has enough flats.
\end{proposition}
\begin{proof}
Let $\cV$ be an object of $\ModI(\Lambda_X)$. Let us take a locally finite covering $\{U_i\}$ of $X$ with lifting $\{\tcV_i\}$. Then by the same construction as in \cite{KS}, there exists a flat object $\tcF_i$ with a surjection $\tcF_i\rightarrow \tcV_i$. Let $\iota_i\colon U_i\hookrightarrow X$ be the open imbedding. Hence $[\iota_{i!}\tcF_i]$ is also a flat object by Lemma \ref{liftflat} and we have a surjection $\bigoplus_i[\iota_{i!}\tcF_i]\rightarrow \cV$. This completes the proof.
\end{proof}

By a similar argument as in Corollary \ref{liftresol}, we get the following:
\begin{corollary}\label{liftresol2}
For $\tcV\in \Mod^0(\Lambda_X)$, a flat resolution $\tcF^\bullet:=\tcF^0\leftarrow \tcF^{-1}\leftarrow \cdots$ of $\tcV$ gives a flat resolution $[\tcF^\bullet]$ of $[\tcV]$.
\end{corollary}

\subsection{Derived functors}
Note that right and left exactness of various functors $f_*, f_!, f^{-1}, \cHom, \otimes$ are the same as in the case of $\bk$-modules, according to Lemma \ref{1.10}. 
\subsection*{Derived functors}
\begin{lemma}\label{rightderived}
Take $\tcV \in D^b(\Mod^0(\Lambda_X))$ and $\tcW\in D^b(\Mod^0(\Lambda_X))$. Then we have $[\bR f_* \tcV]\simeq \bR f_*[\tcV]$, $[\bR f_!\tcV]\simeq \bR f_![\tcV]$ and $[\RcHom(\tcW, \tcV)]\simeq \RcHom([\tcW], [\tcV])$.
\end{lemma}
\begin{proof}
Since $[\cdot]$ is an exact functor, it suffices to show for an object $\Mod^0(\Lambda_X)$ by a standard argument in homological algebra. Then $\tcV$ has an injective resolution $\tcI^\bullet$ such that $[\tcI^\bullet]$ is an injective resolution of $[\tcV]$ by Lemma \ref{liftresol}. For $F\in \{f_*, f_!, \cHom([\tcW], -)\}$, we have
\begin{equation}
[\bR F(\tcV)]\simeq [F(\tcI^\bullet)]\simeq F[\tcI^\bullet]\simeq \bR F([\tcV]).
\end{equation}
This completes the proof.
\end{proof}

\begin{lemma}\label{derivedpullback}
Take $\tcV\in D^\bullet(\Mod^0(\Lambda_X))$. Then we have $[f^{-1}\tcV]\simeq f^{-1}[\tcV]$.
\end{lemma}
\begin{proof}
This is clear from the definition of $f^{-1}$ and its exactness on $\Mod^0(\Lambda_X)$. 
\end{proof}

\begin{lemma}\label{leftderived}
Take $\tcV \in D^\bullet(\Mod^0(\Lambda_X))$ and $\tcW\in D^b(\Mod^0(\Lambda_X))$. Then we have $[\tcV\dotimes\tcW]\simeq [\tcV]\dotimes [\tcW]$.
\end{lemma}
\begin{proof}
One can prove by the same argument as in Lemma \ref{rightderived} by using Corollary \ref{liftresol2}.
\end{proof}

\subsection*{Derived adjuntions}
\begin{lemma}\label{derivedadjunction}
There exists the following isomorphism
\begin{equation}
\RcHom(\cV\dotimes \cW, \cX)\simeq \RcHom(\cV, \RcHom(\cW, \cX)).
\end{equation}
for $\cV, \cW, \cX\in D^b(\ModI(\Lambda_\dX))$.
\end{lemma}
\begin{proof}
This can be proved by a standard argument. Let us take a flat resolution $\cF$ of $\cW$ and an injective resolution $\cI$ of $\cX$. Then $\RcHom(\cF, \cI)\simeq \cHom(\cF, \cI)$ is again an injective object. Actually, we have 
\begin{equation}
\cHom(-, \cHom(\cF, \cI)))\cong \cHom((-)\otimes \cF, \cI)
\end{equation}
by Lemma~\ref{tensorhomadjunction}. Then both sides of the equality in the statement is quasi-isomorphic to $\cHom(\cV\otimes \cF, \cI)$. This completes the proof.
\end{proof}

\begin{lemma}\label{derivedadjunction2}
There exists the following isomorphism
\begin{equation}
\bR f_*\RcHom(f^{-1}\cV, \cW)\simeq \RcHom(\cV, \bR f_*\cW)
\end{equation}
for $\cV\in D^b(\ModI(\Lambda_\dY))$ and $\cW\in D^b(\ModI(\Lambda_\dX))$.
\end{lemma}
\begin{proof}
By Lemma~\ref{tensorhomadjunction2} and the exactness of $f^{-1}$ imply that push-forward of an injective is again injective. Also, pull-back of a flat object is again flat. Let $\cF$ be a flat resolution of $\cV$ and $\cI$ be an injective resolution of $\cW$. By replacing with these resolutions, we can work with underived functors, then Lemma~\ref{pushpulladjunction} completes the proof.
\end{proof}

\subsection*{Upper shriek} 
To construct upper shriek, we follow the argument in \cite{KS}.

Let $f\colon Y\rightarrow X$ be a map. Assume that $f_!\colon \Mod(\bZ_X)\rightarrow \Mod(\bZ_Y)$ has finite cohomological dimension. Let $\tcV$ be an object of $\Mod^0(\Lambda_X)$ and $K$ be a flat $f$-soft $\bZ_Y$-module. We define a presheaf by
\begin{equation}
(\Gr^a(f^!_K\tcV))(U):=\Gamma(U, \cHom^a_{\Lambda_X}(f_!(\Lambda_Y\otimes_{\bZ_Y}K_U), \tcV)).
\end{equation}
This is actually a sheaf by \cite[Lemma 3.1.3]{KS}. We set $f^!_K\tcV:=\bigoplus_{a\in \bR}\Gr^a(f^!_K\tcV)$ which is an object of $\Mod^0(\Lambda_Y)$.
Let us moreover suppose $\tcV$ be an injective object.
\begin{lemma}
Under the above assumption, we have the following:
\begin{enumerate}
\item The object $f_K^!\tcV$ is an injective object of $\Mod^0(\Lambda_{Y})$.
\item For any $\tcW\in \Mod^0(\Lambda_{Y})$. we have a canonical isomorphism
\begin{equation}
\Hom_{\Mod^0(\Lambda_X)}(f_!(\tcW\otimes_\bZ K), \tcV)\xrightarrow{\simeq}\Hom_{\Mod^0(\Lambda_Y)}(\tcW, f^!_K\tcV).
\end{equation}
\end{enumerate}
\end{lemma}
\begin{proof}
This is done by the same argument as in the proof of \cite[Lemma 3.1.3]{KS}.
\begin{comment}
For $\phi\in \Hom(f_!(\tcW\otimes K), \tcV)$ and open subset $V\subset Y$, we have morphisms
\begin{equation}
\tcW(V)\otimes_{\Lambda}f_!(\Lambda_Y\otimes K_V)\rightarrow f_!(\tcW\otimes K_V)\rightarrow f_!(\tcW\otimes K)\rightarrow \tcV.
\end{equation}
The composition of these morphisms give $\tcW(V)\rightarrow f^!_K\tcV(V)$. Hence we get $\alpha(\tcW)\colon \Hom(f_!(\tcW\otimes \tcV), F)\rightarrow \Hom(\tcW, f^!_K\tcV)$. We will prove $\alpha(\tcW)$ is an isomorphism. One can prove the same argument as in the proof of \cite[Lemma 3.1.3]{KS}. 
\end{comment}
\end{proof}

Let us take $K$ by the following.
\begin{lemma}[{\cite[Proposition 3.1.4]{KS}}]\label{Zresol}
The sheaf $\bZ_Y$ admits a finite flat $f$-soft resolution $K$.
\end{lemma}
Let $K^+(\Mod^0(\Lambda_X))$ be the homotopy category of injective complexes bounded below of objects in $\Mod^0(\Lambda_X)$. Then we have an equivalence $D^+(\Mod^0(\Lambda_X))\cong K^+(\Mod^0(\Lambda_X))$. We set the composition
\begin{equation}
f^!\colon D^b(\Mod^0(\Lambda_X))\hookrightarrow K^+(\Mod^0(\Lambda_X))\xrightarrow{f^!_K}K^+(\Mod^0(\Lambda_Y))\xrightarrow{\cong} D^+(\Mod^0(\Lambda_Y)).
\end{equation}

\begin{lemma}
The functor $f^!$ is the right adjoint of $\bR f_!$. We moreover have
\begin{equation}
\RcHom(\bR f_!\tcV, \tcW)\cong \bR f_*\RcHom(\tcV, f^!\tcW).
\end{equation}
\end{lemma}
\begin{proof}
For $\tcW\in K^+(\Mod^0(\Lambda_X))\cong D^+(\Mod^0(\Lambda_X))$, we have
\begin{equation}
\Hom_{K^+(\Mod^0(\Lambda_X))}(f_!(\tcW\otimes_{\bZ_Y}K), \tcV)\cong \Hom_{K^+(\Mod^0(\Lambda_X))}(\tcW, f^!\tcV)
\end{equation}
by the above lemma. Since $f_!\tcW\otimes K\simeq \bR f_!\tcW$, we complete the proof of the first assertion. The second assertion can also be proved by the argument of the proof of \cite[Proposition 3.1.10]{KS}.
\end{proof}

Let us now discuss the upper shriek in $D^b(\ModI(\Lambda_\dX))$.
Let $\cV$ be an object of $\ModI(\Lambda_\dX)$ and $K$ be a $\bZ_Y$-module.

Take a locally finite covering $\{U_i\}$ of $X$ with lifts $\{\tcV_i\}$. Hence we get $f_K^!\tcV_i$.
\begin{lemma}
The data $\{[f_K^!\tcV_i]\}$ gives an object of $\ModI(\Lambda_\dX)$. We denote the resulting object by $f^!_K\cV$.
\end{lemma}
\begin{proof}
Over $U_{ij}:=U_i\cap U_j$, we have $f_{ij}\colon [\tcV_i]|_{U_{ij}}\xrightarrow{\cong} [\tcV_j]|_{U_{ij}}$. We can lift this map to $\tf_{ij}\colon \tcV_i|_{U_{ij}}\rightarrow \tcV_j\la a_{ij}\ra|_{U_{ij}}$ (by taking a refined covering if necessary). We also have a lift of the inverse map $\tf_{ji}$. Then $\tf_{ij}\circ \tf_{ji}-T^{a_{ij}+a_{ji}}\id$ is killed by some $T^a$. The map $\tf_{ij}$ induces a map $f_K^!\tf_{ij}\colon f^!_K\tcV_i|_{U_{ij}}\rightarrow f^!_K\tcV_j|_{U_{ij}}$. Then we also have $f_K^!\tf_{ji}$. Then $f_K^!\tf_{ij}\circ f_K^!\tf_{ji}-T^{a_{ij}+{a_{ji}}}\id=f_K^!(\tf_{ij}\tf_{ji}-T^{a_{ij}+a_{ji}}\id)$ is also killed by $T^a$. This completes the proof.
\end{proof}

Since $D^b(\ModI(\Lambda_X))$ has injective resolutions, we have an equivalence $K^{+}(\Lambda_\dX)\cong D^+(\ModI(\Lambda_\dX))$ where the left hand side is the homotopy category of complexes bounded below of injective objects. We denote the composition $D^b(\ModI(\Lambda_X))\hookrightarrow K^{+}(\Lambda_X)\xrightarrow{f^!_K}D^+(\ModI(\Lambda_X))$ by the notation $f^!$.

\begin{proposition}
The functor $f^!$ is an exact functor.
\end{proposition}
\begin{proof}
The exactness easily follows from the exactness of $f^!$ on $D^b(\Mod^0(\Lambda_X))$.
\end{proof}

\begin{proposition}\label{localshriek}
For $\tcV\in D^b(\Mod^0(\Lambda_X))$, we have $f^![\tcV]\cong [f^!\tcV]$.
\end{proposition}
\begin{proof}
Replace $\tcV$ be an injective complex given in Lemma \ref{skyinj}. Then $[\tcV]$ is also an injective complex. By the definition of $f^!$ for $\Mod^0(\Lambda_X)$ and $\ModI(\Lambda_\dX)$, we have $[f^!\tcV]\cong [f^!_K\tcV]\cong f^!_K[\tcV]\cong f^![\tcV]$. This completes the proof.
\end{proof}

\begin{remark}
By carefully seeing homotopy coherence, one may also construct $f^!$ by taking Proposition \ref{localshriek} as the local definition.
\end{remark}

\subsection*{Shriek adjunction}
\begin{proposition}\label{shriekadjunction2}
There exists a functorial isomorphism:
\begin{equation}
\Hom_{D^b(\ModI(\Lambda_\dX))}(\bR f_!\cW,\cV)\cong \Hom_{D^b(\ModI(\Lambda_\dY))}(\cW, f^!\cV)
\end{equation}
for $\cW\in D^b(\ModI(\Lambda_{\dX})$ and $\cV\in D^b(\ModI(\Lambda_\dY))$.
\end{proposition}
\begin{proof}
First, note that $\bR f_!\cW\simeq f_!(\cW\otimes K)$ which is deduced from the local consideration. Let $\cI$ be an injective resolution of $\cV$ and $C(\ModI(\Lambda_\dX))$ be the category of bounded complexes of $\ModI(\Lambda_\dX)$. Then the left hand side of the desired equality is 
\begin{equation}
\Hom_{C(\ModI(\Lambda_\dY)}(f_!(\cW\otimes K), \cI)\cong \Hom_{D^b(\ModI(\Lambda_\dX))}(\cW, f^!\cI).
\end{equation}
We also have a morphism
\begin{equation}
\Hom_{D^b(\ModI(\Lambda_\dX))}(\cW, f^!\cV)\rightarrow \Hom_{D^b(\ModI(\Lambda_\dX))}(\cW, f^!\cI)
\end{equation}
coming from the morphism $\cV\rightarrow \cI$.
We would like to prove this is an isomorphism. Let us see locally on $Y$. From the construction in Lemma~\ref{injectives}, the complex $\cI$ is coming from an injective object $\tcI$ locally. Hence we have an isomorphism
\begin{equation}
\cHom([\tcW], f^![\tcI])\cong [\cHom(\tcW, f^!\tcI)]\cong [\RcHom(\tcW, f^!\tcI)]\cong \RcHom(\cW, f^!\cV)
\end{equation}
Here we used the fact that $[\cdot]$ is exact and $f^!\tcI$ is injective. Then Lemma~\ref{globalsection2} completes the proof.
\end{proof}

\begin{proposition}\label{shriekadjunction}
There exists a functorial isomorphism:
\begin{equation}
\RcHom(\bR f_!\cW, \cV)\simeq \bR f_*\RcHom(\cW, f^!\cV).
\end{equation}
for $\cW\in D^b(\ModI(\Lambda_{\dX}))$ and $\cV\in D^b(\ModI(\Lambda_\dY))$.
\end{proposition}
\begin{proof}
As usual sheaves, we have a canonical morphism $\bR f_*\RcHom(\cW, f^!\cV)\rightarrow \RcHom(\bR f_!\cW, \bR f_!f^!\cV)$. By the adjunction (Proposition~\ref{shriekadjunction2}), we have a morphism $\RcHom(\bR f_!\cW, \bR f_!f^!\cV)\rightarrow \RcHom(\bR f_!\cW, \cV)$. We would like to see the composition is an isomorphism. By a local consideration, this can be deduced from the usual case.
\end{proof}

\subsection*{Formulas}
\begin{lemma}\label{tensordiagonal}
Let $\delta\colon X\rightarrow X\times X$ be the diagonal embedding. For $\cV, \cW\in \ModI(\Lambda_\dX)$, we have $\delta^{-1}(\cV\boxtimes^\bL \cW)\simeq \cV\dotimes \cW$.
\end{lemma}
\begin{proof}
This is clear from the same formula for usual sheaves.
\end{proof}

As in Lemma~\ref{globalsection}, we can relate $\bR\cHom$ and usual $\Hom$ as follows. Let $\cV, \cW\in D^b(\ModI(\Lambda_\dX))$, then from $\bR\cHom(\cV,\cW)\in \ModI(\Lambda_\dX)$, we can construct a complex of sheaves $\bR\cHom(\cV, \cW)\otimes \bk$ as in the paragraph before Lemma~\ref{globalsection}.
\begin{lemma}\label{globalsection2}
The space of global sections of $\RcHom(\cV, \cW)\otimes \bk$ is canonically isomorphic to $\Hom_{D^b(\ModI(\Lambda_\dX))}(\cV, \cW)$.
\end{lemma}
\begin{proof}
Let $\cI$ be an injective resolution of $\cW$ and $\cF$ be a flat resolution of $\cV$. Then we have
\begin{equation}
\begin{split}
\Hom_{D^b(\ModI(\Lambda_\dX))}(\cV, \cW)&\cong \Hom_{C(\ModI(\Lambda_\dX))}(\cV, \cI)\\
&\cong H^0(\Hom_{\ModI_{\Lambda_\dX}}(\cV,\cI)(X))\\
&\cong H^0(\cHom(\cF, \cI)\otimes \bk(X))
\end{split}
\end{equation}
Actually $\otimes \bk$ is exact, as we will see in the proof of Lemma~\ref{exactnessofF}. This completes the proof.
\end{proof}

\begin{lemma}\label{formulas}
For $\cV, \cW,\cX\in D^b(\ModI(\Lambda_\dX))$, $\cV'\in D^b(\ModI(\Lambda_\dY))$, and a tame morphism $f\colon \dX\rightarrow \dY$, the followings hold:
\begin{enumerate}
\item $\Hom_{D^b(\ModI(\Lambda_\dX))}(\cV\otimes^\bL\cW,\cX)\cong \Hom_{D^b(\ModI(\Lambda_\dX))}(\cV, \bR\cHom(\cW,\cX))$.
\item $\Hom_{D^b(\ModI(\Lambda_\dX))}(f^{-1}\cV', \cW)\cong \Hom_{D^b(\ModI(\Lambda_\dY))}(\cV', \bR f_*\cW)$.
\item $\Hom_{D^b(\ModI(\Lambda_\dX))}(f_!\cW, \cV')\cong \Hom_{D^b(\ModI(\Lambda_\dX))}(\cW, f^!\cV')$.
\end{enumerate}
\end{lemma}
\begin{proof}
This follows from Lemma~\ref{globalsection2}, Lemma~\ref{derivedadjunction}, Lemma~\ref{derivedadjunction2}, Proposition~\ref{shriekadjunction}.
\end{proof}

\begin{lemma}[Projection formula]\label{projectionformula}
We have the following:
\begin{equation}
\bR f_!(\cV\dotimes f^{-1}\cW)\simeq \bR f_!\cV\dotimes \cW.
\end{equation}
for $\cV, \cW\in D^b(\ModI(\Lambda_\dX))$.
\end{lemma}
\begin{proof}
We use Yoneda, Lemma~\ref{formulas}, ad Proposition~\ref{shriekadjunction}:
\begin{equation}
\begin{split}
\Hom(\bR f_!(\cV\dotimes f^{-1}\cW), \cX)&\simeq \Hom(\cV \dotimes f^{-1}\cW, f^!\cX)\\
&\simeq \Hom(f^{-1}\cW, \bR\cHom(\cV, f^!\cX))\\
&\simeq \Hom(\cW, f_*\RcHom(\cV, f^!\cX))\\
&\simeq \Hom(\cW, \RcHom(f_*\cV, \cX))\\
&\simeq \Hom(f_*\cV\dotimes\cW, \cX).
\end{split}
\end{equation}
We omitted the subscripts to Hom-spaces to shorten the notation. This completes the proof.
\end{proof}

\begin{lemma}\label{shriekhom}
We have the following formula
\begin{equation}
f^{!}\RcHom(\cV, \cW)\simeq \RcHom(f^{-1}\cV, f^!\cW)
\end{equation}
for $\cV, \cW\in D^b(\ModI(\Lambda_\dX))$.
\end{lemma}
\begin{proof}
We use Yoneda, Lemma~\ref{formulas}, Lemma~\ref{projectionformula}:
\begin{equation}
\begin{split}
\Hom(\cX, f^!\RcHom(\cV, \cW))&\simeq \Hom(f_!\cX, \RcHom(\cV, \cW))\\
&\simeq \Hom(f_!\cX\dotimes \cV, \cW)\\
&\simeq \Hom(f_!(\cX\dotimes f^{-1}\cV), \cW)\\
&\simeq \Hom(\cX\dotimes f^{-1}\cV,f^!\cW)\\
&\simeq \Hom(\cX, \RcHom(f^{-1}\cV,f^!\cW)).
\end{split}
\end{equation}
We omitted the subscripts to Hom-spaces to shorten the notation. This completes the proof.
\end{proof}

\begin{lemma}\label{tensorpullback}
We have the following formula
\begin{equation}
f^{-1}(\cV\dotimes\cW)\simeq f^{-1}\cV\dotimes f^{-1}\cW
\end{equation}
for $\cV, \cW\in D^b(\ModI(\Lambda_\dX))$.
\end{lemma}
\begin{proof}
We use Yoneda, Lemma~\ref{pushpulladjunction}, Lemma~\ref{formulas}:
\begin{equation}
\begin{split}
\Hom(f^{-1}(\cV\dotimes \cW), \cX)&\simeq \Hom(\cV\dotimes \cW, \bR f_*\cX)\\
&\simeq \Hom(\cV, \RcHom(\cW, \bR f_*\cX))\\
&\simeq \Hom(\cV, \bR f_*\RcHom(f^{-1}\cW, \cX))\\
&\simeq \Hom(f^{-1}\cV,\RcHom(f^{-1}\cW, \cX)\\
&\simeq \Hom(f^{-1}\cV\dotimes f^{-1}\cW, \cX).
\end{split}
\end{equation}
We omitted the subscripts to Hom-spaces to shorten the notation. This completes the proof.
\end{proof}

\section{Irregular constructibility}
In this section, we introduce the notion of $\bC$-constructibility for objects in $\ModI(\Lambda_\dX)$. It is defined in the same way for stratification as usual constructible sheaves but with a strong assumption on gradings coming from Sabbah--Mochizuki--Kedlaya's Hukuhara--Levelt--Turritten theorem. In this section, we consider $\dX=(X, \varnothing)$ with $X$ is a complex manifold. We denote $\ModI(\Lambda_\dX)$ by $\ModI(\Lambda_X)$.

\subsection{Formal structure}
In this subsection, we recall as a motivation the theory of formal structures of meromorphic connections initiated by Sabbah \cite{SabbahHLT} and developed by Mochizuki (algebraic case) \cite{Mochizukiformal} and Kedlaya (analytic case) \cite{Kedlaya}.

Let $Z$ be a divisor in a complex manifold $X$ and $\hO_{x}$ be the formal completion of $\cO_X$ at $x\in X$. Let $\cM$ be a meromorphic connection over $X$ with poles along $Z$. We set $\widehat{\cM}_x=\cM_x\otimes_{\cO_x}\hO_x$ and $\hO(*Z)_x:=\cO(*Z)_x\otimes_{\cO_x}\hO_x$.

\begin{definition}
\begin{enumerate}
\item For $\phi\in \hO(*Z)_{x}$, we set $\widehat{\cE}(\phi)$ to be $\hO(*Z)_{x}$ as a $\hO_x$-module with a connection $\nabla$ over $\hO_x$ such that 
\begin{equation}
\nabla s:=\partial(\phi)\cdot s
\end{equation}
for the generator $s$. 
\item We assume that $Z$ is a normal crossing divisor and take a local coordinate $\{x_i\}_{i=1}^n$ such that $Z$ is defined by $\prod_{i=1}^mx_i=0$. An $\hO(*Z)$-module $\widehat{\cR}$ with a connection $\nabla$ is {\em regular} if there exists an $\hO_x$-submodule $\cL$ such that $\cL\otimes_{\hO_x}\hO(*Z)_x\cong \widehat{\cR}$ and $\nabla(\cL)\subset \bigoplus_{i=1}^m x^{-1}_i\cL$.
\end{enumerate}
\end{definition}

\begin{definition} We continue the notations in Definition 5.1.2.
\begin{enumerate}
\item A good decomposition of $\widehat{\cM}_x$ is an isomorphism
\begin{equation}
\widehat{\cM}_x\cong \bigoplus_{\alpha\in I}\widehat{\cE}(\phi_\alpha)\otimes_{\hO(*Z)_x}\widehat\cR_\alpha
\end{equation}
where $\phi_\alpha\in \hO(*Z)_{x}$ and each $\widehat\cR_\alpha$ is regular with the conditions
\begin{enumerate}
\item Each $\phi_\alpha$ has the form $u\prod_{j=1}^mx_i^{-i_j}$ for some unit $u\in \hO_x$ and nonnegative integers $i_1,..., i_m$.
\item For $\alpha, \beta\in I$, if $\phi_a-\phi_\beta\not\in \hO_x$, then $\phi_\alpha-\phi_\beta$ has the form $u\prod_{j=1}^mx_i^{-i_j}$ for some unit $u\in \hO_x$ and nonnegative integers $i_1,..., i_m$.
\end{enumerate}
\item We say $\cM$ admits a good decomposition at $x\in Z$ if $\widehat{\cM}_x$ admits a good decomposition.
\end{enumerate}
\end{definition}

In general, meromorphic connections do not have good decompositions as explained in \cite{SabbahHLT}. Sabbah's conjecture says that they do after modifications, which is proved by Mochizuki and Kedlaya.

\begin{theorem}[{\cite[Theorem 8.2.2]{Kedlaya}}]
For a point $x\in Z$, there exists an open neighborhood $U$ of $z$ and a map $f\colon Y\rightarrow U$ which is proper surjective and unramified covering over $f^{-1}(U\bs (U\cap Z))$ such that $f^*\cE$ admits a good decomposition at each point of $y\in f^{-1}(Z)$.
\end{theorem}
As explained in \cite{Sabbahgoodformal}, using Mochizuki's result, we have additional results. For $\widehat{\cM}_x$ which admits a good decomposition, let $\Phi_{x}$ be the subset of $\hO(*Z)_x/ \hO_x$ consisting of the classes of $\phi_\alpha$'s.
\begin{theorem}[{\cite[Theorem 2.2.1]{Sabbahgoodformal}}]
The subset $\Phi_x$ is actually a subset of $\cO(*Z)_z/\cO_x$. Moreover there exists a neighborhood $U$ of $x$ such that for any $x'\in U$, $\widehat{\cM}_{x'}$ has a good decomposition and $\Phi_{x'}$ is given by the restriction of representatives of $\Phi_x$.
\end{theorem}

Let $\varpi\colon \widetilde{X}(Z)\rightarrow X$ be the real blow-up of $X$ along $Z$ (with real analytic structure specified in \cite{D'Agnolo--Kashiwara}). Let $C^{\infty,\mathrm{temp}}_{\widetilde{X}}(Z)$ be the subsheaf of the sheaf of $C^\infty$-functions consisting of functions which are tempered at the exceptional divisor. Let further $\cA_{\widetilde{X}(Z)}$ be the subsheaf of $C^{\infty,\mathrm{temp}}_{\widetilde{X}}(Z)$ consisting of functions whose restrictions on $X\bs Z$ are holomorphic. We set $\cD^\cA_{\widetilde{X}(Z)}:=\cA_{\widetilde{X}(Z)}\otimes_{\varpi^{-1}\cO_X} \cD_X$. For a $\cD$-module $\cN$ on $X$, we set $\varpi^*\cN:=\cD^\cA_{\widetilde{X}(Z)}\otimes_{\varpi^{-1}\cD_X}\varpi^{-1}\cN$.

Suppose that $\cM$ has a good decomposition $\bigoplus_{\alpha\in I}\widehat\cE(\phi_\alpha)\otimes \widehat\cR_\alpha$ at $x$. For each $\phi_\alpha$, by taking a representative locally around $x$, we set $\cE(\phi_\alpha)$ to be a meromorphic connection $(\cO(*Z), \nabla)$ defined by $\nabla s:=\partial(\phi) s$ for the generator $s$. We also set $\cR_\alpha$ to be a regular meromorphic connection defined locally around $x$ corresponding to $\widehat{\cR}_\alpha$.

The following thoerem is proved in \cite{Mochizukiformal} and explained in \cite{SabbahStokes}.
\begin{theorem}[{\cite[Theorem 12.5]{SabbahStokes}}]
There exists an open covering $\{U_i\}$ of a neighborhood of $\varpi^{-1}(x)$ such that each restriction $(\varpi^*\cM)|_{U_i}$ is isomorphic to $(\varpi^*(\bigoplus_{\alpha \in I}\cE(\phi_\alpha)\otimes \cR_\alpha))|_{U_i}$.
\end{theorem}

\subsection{Irregular constant sheaf $\Lambda^\phi$}
In this subsection, we prepare some preliminary lemmas concerning a class of modules.

Let $\dS$ be a topological space with boundary. Let $\phi$ be a $\bC$-valued continuous function over $S:=\oS\bs D_S$. We set 
\begin{equation}
\begin{split}
\Gr^a\Lambda^\phi_S&:=p_*\Gamma_{S\times [-a,\infty)}\bk_{t\geq\Re\phi}\\
\Lambda^\phi_{S}&:=\bigoplus_{a\in \bR}p_*\Gamma_{S\times [-a,\infty)}\bk_{t\geq\Re\phi}
\end{split}
\end{equation}
where $\bk_{t\geq \Re\phi}$ is the constant sheaf supported on the set $\lc(s, t)\in S\times \bR\relmid t\geq \Re\phi(s)\rc$ and $p\colon S\times \bR\rightarrow S$ is the projection.

\begin{lemma}
The sheaf $\Lambda^\phi_{S}$ defines an object of $\Mod^\bR(\Lambda_S)$. In particular, an object of $\Lambda^\phi_\dS:=[\Lambda^\phi_S]\in \ModI(\Lambda_{(\overline{S}, D_S)})$.
\end{lemma}
\begin{proof}
Since the sheaf is globally presented as a direct sum, the restriction morphism preserves grading. The $\Lambda$-action is given as follows: For $b\in \bRz$, we have a canonical morphism
\begin{equation}
\Gamma_{S\times [-a,\infty)}\bk_{t\geq \Re\phi}\rightarrow \Gamma_{S\times [-a-b,\infty)}\bk_{t\geq \Re\phi}.
\end{equation}
This action gives the action of $T^b$.
\end{proof}

\begin{comment}
\begin{remark}\label{irregularconstant}
We would like to consider those sheaves as  ``irregular constant sheaves''. Let us take $\Lambda^\phi\in \Mod^\bR(\Lambda_S)$. Let $U$ be a relatively compact open subset $U\subset S$ and set $b:=\inf_U \Re \phi$. Then for $-a\leq b$, we have $p_*\Gamma_{S\times [-a, \infty)}k_{t\geq\Re\phi}|_U\cong k$. Hence we have 
\begin{equation}
\bigoplus_{-a\leq b} p_*\RGamma_{[-a, \infty)}k_{t\geq\Re\phi}(V)\cong \Lambda
\end{equation}
as a $\Lambda$-module without grading for any $V\subset U$. Since $\bigoplus_{-a\leq b} p_*\RGamma_{[-a, \infty)}k_{t\geq \Re\phi}(V)\cong \Lambda^\phi$ in $\ModI(\Lambda_\dS)$, we can conclude that $\Lambda^\phi$ is isomorphic to the constant sheaf over $U$.
\end{remark}
\end{comment}

We would like to see the structure of $\Lambda^\phi_S$ a little bit closer.
\begin{lemma}
Let $U$ be a connected open subset of $S$ such that $\phi|_U$ is bounded. Set $b:=\inf_U\Re \phi$. Then $\Lambda^\phi_S(U)\cong \Lambda\cdot T^b$. 
\end{lemma}
\begin{proof}
Note that $\Gr^a\Lambda^\phi_S(U)\cong \Gamma_{U\times [-a, \infty)}(U\times\bR, \bk_{t\geq \Re\phi})$. This is the kernel of the restriction morphism $\Gamma(U\times \bR, \bk_{t\geq \Re\phi})\rightarrow \Gamma(U\times (-\infty, -a), \bk_{t\geq \Re\phi})$. Since $U$ is connected, the set defined by $t\leq \Re\phi$ is also connected. Hence we have $\Gamma(U\times \bR, \bk_{t\geq \Re\phi})\cong \bk$. On the other hand, $\Gamma(U\times (-\infty, -a), \bk_{t\geq \Re\phi})\cong 0$ if and only if $U\times (-\infty, -a)\cap \{t\geq \Re\phi\}=\varnothing$. This is equivalent to $a<\inf_U\Re\phi$. This completes the proof.
\end{proof}

For given $x\in S$, let us set as follows:
\begin{equation}
\Lambda^{\phi(x)}:=\begin{cases}
\bigoplus_{-a\leq \Re\phi(x)}\bk &\text{if $x$ is a local minimum}\\
\bigoplus_{-a< \Re\phi(x)}\bk &\text{otherwise}.
\end{cases}
\end{equation}
These are $\bR$-graded $\Lambda$-modules with obvious gradings. Note that these are torsion-free $\Lambda$-modules and the ring $\Lambda$ has a valuation. Hence these modules are flat.

From this lemma, the following is clear.
\begin{corollary}\label{irrconststalk}
For $x\in S$, the stalk $(\Lambda^\phi_S)_x\cong \Lambda^{\phi(x)}$. 
\end{corollary}

\begin{corollary}\label{irrconstcomponent}
We have $\Gr^d\Lambda^\phi_S\cong \bk_{\mathrm{Int}\overline{\lc x\relmid -d<\Re\phi(x)\rc}}$
\end{corollary}
\begin{proof}
Let $x$ be a point with $\Re\phi(x)=-d$. The point $x$ is a local minimum if and only if $x$ is in the interior of the closure of $\lc-x\relmid -d<\Re\phi(x)\rc$. This completes the proof.
\end{proof}

Also, the module $\Lambda^\phi_\dS$ plays the role similar to the constant sheaves in the usual theory of sheaves. The following lemma is an example of this motto.
\begin{lemma}\label{flatness}
The module $\Lambda^\phi_\dS$ is a flat object in $\ModI(\Lambda_{\dS})$.
\end{lemma}
\begin{proof}
Let $\cV\rightarrow \cW$ be an injective morphism in $\ModI(\Lambda_\dS)$. We would like to show the induced morphism $\cV\otimes \Lambda^\phi_\dS\rightarrow \cW\otimes \Lambda^\phi_\dS$ is again injective. There exists a covering $\{U_i\}$ of $S$ which is locally finite in $\oS$ such that there exists representatives $\cV_i, \cW_i$ of $\cV$ and $\cW$ over each $U_i$. It is enough to prove the injectivity over each $U_i$.

By Lemma \ref{liftlemma}, one can assume the restriction $\cV_i\rightarrow \cW_i$ is still injective. Since the tensor product commutes with taking stalks, it reduces to show that $\cV_x\otimes (\Lambda^\phi_S)_x\rightarrow \cW_x\otimes (\Lambda^\phi_S)_x$ is injective. Since $(\Lambda^\phi_S)_x\cong \Lambda^{\phi(x)}$ (Corollary~\ref{irrconststalk}) is a torsion-free $\Lambda$-module, this completes the proof.
\end{proof}

\begin{lemma}\label{boundeds}
Let $\phi_1$ and $\phi_2$ be $\bC$-valued continuous functions over connected $S$ such that $\max\{0, \Re\phi_1-\Re\phi_2\}$ is bounded. Then there exists a canonical idenitification 
\begin{equation}
\Hom_{\ModI(\Lambda_\dS)}(\Lambda^{\phi_1}_\dS, \Lambda^{\phi_2}_\dS)\cong \bk.
\end{equation}
If moreover $\Re\phi_1-\Re\phi_2$ is bounded, two objects are isomorphic.
\end{lemma}
\begin{proof}
Since $\max\{0,\Re\phi_1-\Re\phi_2\}$ is bounded, there exists a large $c\in \bR$ such that $\Re\phi_2+c\geq \Re\phi_1$. The restriction map $\bk_{\Re\phi_1\geq t}\rightarrow \bk_{\Re\phi_{2}+c\geq t}$ induces a morphism $\Lambda^{\phi_1}_S\rightarrow \Lambda^{\phi_2}_S\la c\ra$ of $\bR$-graded $\Lambda_S$-modules.
If $\max\{0, \Re\phi_2-\Re\phi_1\}$ is also bounded, in the same way, we also have a morphism $\Lambda^{\phi_2}_S\rightarrow \Lambda^{\phi_1}_S\la d\ra$ for some $d\geq 0$. The composition $\Lambda^{\phi_1}_S\rightarrow \Lambda^{\phi_1}_S\la c+d\ra$ is given by $T^{c+d}$. This is the identity of $\Lambda^\phi_\dS$ in $\ModI(\Lambda_X)$. The same for the other direction. This completes the proof of the second part of the statement. We call the morphism $\Lambda^{\phi_1}_S\rightarrow \Lambda^{\phi_2}_S$ and its scalar multiples {\em standard morphisms}. In the below, we will see there are only standard morphisms.

Let $f$ be a nonzero morphism in $\Hom_{\ModI(\Lambda_\dS)}(\Lambda^{\phi_1}_\dS, \Lambda^{\phi_2}_\dS)$. Let us take a representative $\tf\colon \Lambda^{\phi_1}_S\rightarrow \Lambda^{\phi_2+c}_S$ as a morphism of $\bR$-graded $\Lambda$-modules locally on $U\subset S$. We can take so that $c+\Re\phi_2>\Re\phi_1$ and replace $\phi_2$ with $\phi_2+c$ We consider $d\in \bR$ such that the grading $d$-part of $\tf$ is nonzero. To see this part more explicitly, let us prepare some notations.

Let us set $\mathrm{Int}\overline{\lc x\in U\relmid -d< \Re\phi_i(x)\rc}=\sqcup_a S_{d,i}^a$ be the decomposition into connected components. Since $\Gr^d\Lambda_U^\phi=\bk_{\mathrm{Int}\overline{\lc x\relmid -d< \Re\phi(x)\rc}}$, we have $\Gr^d\Lambda_U^{\phi_i}=\bigoplus_{a}\bk_{S^a_{d,i}}$. We have
\begin{equation}
\tf_d\colon \bigoplus_{a}\bk_{S^a_{d,1}} \rightarrow \bigoplus_{a}\bk_{S^a_{d,2}}.
\end{equation}
There exists $d'\in \bRz$ such that there exists a connected component $S_i$ of $\mathrm{Int}\overline{\lc x\relmid -d'<\Re\phi_i(x)\rc}$ for each $i$ such that $S_i\supset \mathrm{Int}\overline{\lc x\relmid -d<\Re\phi_i(x)\rc}$. Then we have a commutative diagram
\begin{equation}
\xymatrix{
\bk_{S_1}\ar[r]^{\tf_{d'}} & \bk_{S_2} \\
\bigoplus_{a}\bk_{S^a_{d,1}} \ar[u]_{T^{d'-d}} \ar[r]_{\tf_d}&\ar[u]_{T^{d'-d}} \bigoplus_{a}\bk_{S^a_{d,2}}
}
\end{equation}
Since $S_1$ and $S_2$ are connected, the hom-space between them is 1-dimensional. Hence $\tf_d$ is induced by a standard morphism.  This completes the proof.  
\end{proof}

We prepare the following crucial lemma. The corresponding observation in the theory of enhanced ind-sheaves is a key to the formulation of irregular Riemann--Hilbert correspondence \cite{D'Agnolo--Kashiwara}.
\begin{lemma}\label{orthogonality}
Let $\dS$ be a topological space with boundary with $S$ connected. Let $\phi_1, \phi_2$ be $\bC$-valued continuous functions on $S$. Assume that there exists an open subset $V$ of $S$ such that $\overline{V}\cap D_S$ is nonempty and $\Re\phi_2-\Re\phi_1$ is divergent to $-\infty$ on $\overline{V}\cap D_S$. Then there exists no nonzero morphisms from $\Lambda_\dS^{\phi_1}$ to $\Lambda_\dS^{\phi_2}$. 
\end{lemma}
\begin{proof}
For $f\in \Hom_{\ModI(\Lambda_\dS)}(\Lambda_\dS^{\phi_1}, \Lambda_\dS^{\phi_2})$, let us take a representative $\tf\colon \Lambda_\dS^{\phi_1}\rightarrow \Lambda_\dS^{\phi_2+c}$ as a morphism between $\bR$-graded $\Lambda_S$-modules. Since $\Re\phi_2-\Re\phi_1$ is negatively divergent, there exists a neighborhood $U$ of $D_S$ such that $\Re\phi_2+c-\Re\phi_1$ is negative on $U\bs D_S$. Hence over $U\bs D_S$, the restriction of $\tf$ is zero there. By Lemma \ref{boundeds} and the connectedness of $S$, $f$ is zero everywhere.
\end{proof}

We also give the following.
\begin{lemma}\label{tensorstandard}
For $\Lambda^{\phi_i}_\dS\in \ModI(\Lambda_\dS)$ $(i=1,2)$, we have $\Lambda^{\phi_1}_\dS\otimes \Lambda^{\phi_2}_\dS\cong \Lambda^{\phi_1+\phi_2}_\dS$. In particular, $\Lambda^\phi_\dS\otimes \Lambda^{-\phi}_\dS\cong \Lambda_\dS$.
\end{lemma}
\begin{proof}
We have $\Gr^a\Lambda_S^{\phi_i}=\bk_{\lc x\relmid \Re\phi(x)>-a\rc}$ for $i=1,2$. Hence we have a map $\Gr^a\Lambda^{\phi_1+\phi_2}_S\rightarrow \Gr^b\Lambda^{\phi_1}_S\otimes_k\Gr^c\Lambda^{\phi_2}_S$ for $a=b+c$. Hence we get a map $m\colon \Lambda^{\phi_1+\phi_2}_S\rightarrow \Lambda^{\phi_1}_S\otimes \Lambda^{\phi_2}_S$. By Corollary \ref{irrconststalk}, the stalks of both sides at $x\in X$ are $\bigoplus_{-a\leq \Re\phi_1(x)+\Re\phi_2(x)}\Lambda^{\phi_1(x)+\phi_2(x)}_S$ or  $\bigoplus_{-a< \Re\phi_1(x)+\Re\phi_2(x)}\Lambda^{\phi_1(x)+\phi_2(x)}_S$. Hence the kernel and cokernel of $m$ is killed by $T^a$ for any $a\in \bR$. Therefore the kernel and cokernel are zero in $\ModI(\Lambda_\dS)$. This completes the proof. 
\end{proof}

Similarly, we have
\begin{lemma}
For $\Lambda^{\phi_i}_\dS\in \ModI(\Lambda_\dS)$ $(i=1,2)$, we have $\cHom(\Lambda^{\phi_1}_\dS, \Lambda^{\phi_2}_\dS)\cong \Lambda^{\phi_2-\phi_1}_\dS$.
\end{lemma}
\begin{proof}
One can prove in a similar way as in the proof of Lemma~\ref{tensorstandard}.
\end{proof}

The following will be repeatedly used later.
\begin{corollary}\label{Rhomstandard}
We have $\RcHom(\Lambda^{\phi_1}_\dS, \Lambda^{\phi_2}_\dS)\simeq \Lambda^{\phi_1-\phi_2}_\dS$.
\end{corollary}
\begin{proof}
Let $\cI$ be an injective resolution of $\Lambda^{\phi_2}_\dS$. We have the following:
\begin{equation}
\begin{split}
\Hom_{D^b(\ModI(\Lambda_\dS))}(\cV, \RcHom(\Lambda^{\phi_1}_\dS, \Lambda^{\phi_2}_\dS))&\cong \Hom_{D^b(\ModI(\Lambda_\dS))}(\cV\otimes \Lambda^{\phi_1}_\dS, \Lambda^{\phi_2}_\dS)\\
&\cong \Hom_{C(\ModI(\Lambda_\dS))}(\cV\otimes \Lambda^{\phi_1}_\dS, \cI)\\
&\cong \Hom_{C(\ModI(\Lambda_\dS))}(\cV,\cHom(\Lambda^{\phi_1}_\dS, \cI))
\end{split}
\end{equation}
Here we used flatness of $\Lambda^{\phi_1}_\dS$.

First, note that $\cHom(\Lambda^{\phi_1}_\dS, \cI)\cong \Lambda^{-\phi_1}_\dS\otimes \cI$ in $C(\ModI(\Lambda_\dS))$.
Second, $\cI$ is locally given by $[\prod_x \cF_x]$ where $\cF_x$ is a skyscraper sheaf. Since $\cHom(\Lambda^{\phi_1}_\dS, \prod_x\cF_x)\cong \prod_x\cHom(\Lambda^{\phi_1}_\dS, \cF_x)$, the object $\cHom(\Lambda^{\phi_1}_\dS, \cI)$ is also injective. Hence we have
\begin{equation}
\begin{split}
\Hom_{D^b(\ModI(\Lambda_\dS))}(\cV, \RcHom(\Lambda^{\phi_1}_\dS, \Lambda^{\phi_2}_\dS))&\cong
\Hom_{D^b(\ModI(\Lambda_\dS))}(\cV,\Lambda^{-\phi_1}_\dS\otimes\cI)\\
&\cong \Hom_{D^b(\ModI(\Lambda_\dS))}(\cV,\Lambda^{-\phi_1}_\dS\dotimes\cI)\\
&\cong \Hom_{D^b(\ModI(\Lambda_\dS))}(\cV,\Lambda^{-\phi_1}_\dS\dotimes\Lambda^{\phi_2}_\dS)\\
&\cong \Hom_{D^b(\ModI(\Lambda_\dS))}(\cV,\Lambda^{\phi_2-\phi_1}_\dS)
\end{split}
\end{equation}
This completes the proof.
\end{proof}

\subsection{Definition}
Let $V$ be a neighborhood of $0\in \bC^n$ and consider a simple normal crossing $D_I=\bigcup_{i\in I}\{z_i=0\}\cap V$. For $A:=\{a_i\}\in \bZ^{I}$, $\Phi_A\colon \bC^n\rightarrow \bC^n$ is defined by $z_i^{a_i}$ where $a_i=0$ for $i\not \in I$.
\begin{definition}\label{multi}
\begin{enumerate}
\item A correspondence $f\colon V\bs D_I\rightarrow \bC$ is a {\em multi-valued meromorphic function} if there exists $A:=\{a_i\}\in \bZ^{I}$ and a meromorphic function $f'$ on $\Phi_A^{-1}(V)$ with poles in $\Phi_A^{-1}(D_I)$ such that $f$ is equal to $z\mapsto \lc f'(z') \relmid z'\in (\Phi_A)^{-1}(z)\rc$. 
\item A finite set of multi-valued meromorphic function is said to be {\em good}, if it satisfies the conditions in Definition 5.2 after taking the pull-backs along $\Phi_A$. 
\end{enumerate}
\end{definition}
For a multi-valued meromorphic function $\phi$ and an open subset $U$ on which $\phi$ is represented by a set of  single-valued holomorphic functions $\{\phi_k\}_{k\in K}$, we set $\Lambda^\phi:=\bigoplus_{k\in K}\Lambda^{\phi_k}$.

For $S$ a locally closed complex submanifold $X$, consider $(\overline{S}, D_S:=\overline{S}\bs S)$ as a topological space with boundary.

\begin{definition}
Let $\cV$ be an object of $\Mod^\frakI(\Lambda_\dS)$. We call $\cV$ is a {\em good irregular local system} if the followings hold: 
\begin{enumerate}
\item $D_S$ is normal crossing.
\item For any point $x\in D_S$, there exists a neighborhood $U$ of $x$ such that the restriction $\cV|_U\in \ModI(\Lambda_{(U, \varnothing)})$ is isomorphic to a finite direct sum of the constant sheaf $\Lambda_U$.
\item  For any point $x\in \overline{S}\bs S$, there exists
\begin{enumerate}
\item a neighborhood $U$ of $x$
\item a finite good set of multi-valued meromorphic functions $\{\phi_j\}_{j\in J}$ over $U$ with poles in $D_S$, and 
\item a finite cover $\{U_k\}_{k\in K}$ of $U\bs U\cap D_S$
\end{enumerate}
such that 
\begin{enumerate}
\item there exists an open covering $\{U_k'\}_{k\in K}$ of the real blow-up of $U$ along $D_S$ with $U_k=U_k'\cap (U\bs D_S)$, and 
\item each restriction of $\cV|_{U_k}:=\cV|_{(\overline{U_k}, \overline{U_k}\cap D_S)}:=\iota_{(\overline{U_k}, \overline{U_k}\cap D_S)}^{-1}\cV\in \ModI(\Lambda_{(\overline{U_k}, \overline{U_k}\cap D_S)})$
 is isomorphic to the finite direct sum $\bigoplus_{j\in J}\Lambda^{\phi_j}_{(\overline{U_k}, \overline{U_k}\cap D_S)}$. Here $\iota_{(\overline{U_k}, \overline{U_k}\cap D_S)}$ is the canonical map induced by the inclusion $\overline{U_k}\hookrightarrow \overline{S}$.
\end{enumerate}
\end{enumerate}
If the set of multi-valued functions is actually the set of meromorphic functions, we call it a {\em unramified good irregular local system}.
\end{definition}

\begin{remark}
We believe the goodness assumption in 3.b in the above can be removed by a similar consideration done in \cite{Mochizukicurvetest}.
\end{remark}

\begin{lemma}\label{coveringdef}
The above condition 3 is equivalent to the following:
{\em 3'.  For any point $x\in \overline{S}\bs S$, there exists 
\begin{enumerate}
\item a neighborhood $U$ of $x=:0$ (with the notation used in Definition \ref{multi}), 
\item $A:=\{a_i\}\in \bZ^{I}$, 
\item a finite set of meromorphic functions $\{\phi_j\}_{j\in J}$ over $U':=\Phi_A^{-1}(U)$ with poles in $D':=\Phi_A^{-1}(D_I)$, and
\item a finite cover $\{U_k\}_{k\in K}$ of $U'\bs U'\cap D'$
\end{enumerate} 
such that 
\begin{enumerate}[(a)]
\item there exists an open covering $\{U_k'\}_{k\in K}$ of the real blow-up of $U$ along $D'$ with $U_k=U_k'\cap (U\bs D')$, and 
\item each restriction of $(\Phi_A^*\cV)|_{(\overline{U_k}, \overline{U_k}\cap D')}:=\iota_{(\overline{U_k}, \overline{U_k}\cap D')}^{-1}(\Phi_A^*\cV)\in \ModI(\Lambda_{(\overline{U_k}, \overline{U_k}\cap D')})$
 is isomorphic to the finite direct sum $\bigoplus_{j\in J}\Lambda^{\phi_j}_{(\overline{U_k}, \overline{U_k}\cap D')}$. Here $\iota_{(\overline{U_k}, \overline{U_k}\cap D')}$ is the canonical map induced by the inclusion $\overline{U_k}\hookrightarrow \overline{S}$.
\end{enumerate}}
\end{lemma}
\begin{proof}
This is just from the definition of multi-valued meromorphic functions.
\end{proof}

\begin{comment}
\begin{remark}
In this remark, we would like to explain the presence of the proper push-forward in the definition and the reason why we cannot compare simply on each starta with standard sheaves. One way to explain is coming from $\cD$-module side. The Riemann--Hilbert correspondence is compatible with various operations. If we take a $\cD$-module $M$ over $X$ with singularity only at $D$. Then $M$ is an integrable connection outside the singularity. Hence the pull-back of enhanced $\bR$-constructible ind-sheaves on the open subset should be a usual local system. How can it be realized? It is accomplished by the power of ind-sheaf (subanalytic sheaf). As a sheaf $\bigoplus_{a\in \bR}\RGamma_{S\times [-a,\infty)}k_{\Re\phi\geq 0}$ is not isomorphic to the constant sheaf $\Lambda_S$. However, they are isomorphic in the category $\Mod^\frakI(\Lambda_S)$.
\end{remark}
\end{comment}

\begin{definition}
For a complex manifold $U$ with a divisor $D$, a {\em modification} of $(U, D)$ is a morphism $f\colon ({U'}, D')\rightarrow (U, D)$ where $({U'}, D')$ is another complex manifold with a divisor and $f$ is a projective map between ${U'}$ and $U$ preserving divisors and induces the identity map between $U\bs D$ and ${U'}\bs D'$.
\end{definition}

\begin{remark}
Viewing $(U, D)$ and $(U' D')$ as topological spaces with boundaries, a modification is a morphism of topological spaces with boundaries.
\end{remark}

%Let $p\colon (\overline{S}', D_{S'})\rightarrow (\overline{S}, D)$ be a modification of $\dS$. 

\begin{definition}\label{irreloc}
An object $\cV\in \Mod^\frakI(\Lambda_\dS)$ is said to be an {\em irregular local system} if the followings hold: 
\begin{enumerate}
\item For any point $x\in S$, there exists a neighborhood $U$ of $x$ such that the restriction $\cV|_U\in \ModI(\Lambda_{(U, \varnothing)})$ is isomorphic to a finite direct sum of the constant sheaf $\Lambda_U$.
\item For any point $x\in D_S$, there exists a neighborhood $U$ of $x$ and a modification $p\colon (U',D') \rightarrow ({U}, D_S\cap U)$ such that $p^{-1}(\cV|_{(U,D_S\cap U)})$ is a good irregular local system.
\end{enumerate}
We say an irregular local system is single-valued type (resp. multi-valued type) if the good irregular local system appeared in 2 is unramified (ramified).
\end{definition}

Let $\cV$ be an irregular local system on $\dS$. Take a point $x\in D_S$. Then by the definition of irregular local systems, there exists a relatively compact open neighborhood $U$ of $x$ with a modification $p\colon U'\rightarrow U$. Then for any $y\in p^{-1}(D_S)=:D'$, there exists a finite cover $\{U_k\}_k$ of $U'\bs D'$ given in the definition of good irregular local systems. We have $\cV|_{U_k}\cong \bigoplus_i\Lambda^{\phi_i}_{(\overline{U_k}, \overline{U_k}\cap D')}$.

Since $U'\bs D'\cong U\bs D_S$. we get a finite covering $\cU$ of $U\bs D_S$ such that $\cV|_{\oU, D_S\cap \oU}$ is isomorphic to a direct sum of irregular constant sheaves for each $U\in \cU$.
\begin{definition}
We call a finite covering $\cU$ of $U\bs D_S$ given above a {\em sectorial covering} of $\cV$ around $x$.
\end{definition}

\begin{lemma}\label{commonmodification}
For $\cV, \cW\in \ModI(\Lambda_\dS)$ and $x\in D_S$, there exists a neighborhood $U$ of $x$ with a modification $(U',D')\rightarrow (U, D)$ such that $p^{-1}(\cV|_{(U, U\cap D_S)})$ and $p^{-1}(\cW|_{(U, U\cap D_S)})$ are irregular local systems. In particular, $\cV$ and $\cW$ have a common sectorial covering.
\end{lemma}
\begin{proof}
This is standard.
\end{proof}

Then we would like to define one of the fundamental objects in this paper.

\begin{definition}
Let $\cV$ be an object of $\Mod^\frakI(\Lambda_X)$. We say $\cV$ is {\em irregular constructible} if the followings hold:
There exists a $\bC$-analytic stratification $\cS$ of $X$ such that the restriction $\cV|_{(\oS, D_S:=\oS\bs S)}$ to each stratum $S\in \cS$ is an irregular local system as an object of $\ModI(\Lambda_\dS)$.
\end{definition}
Let us denote the full subcategory of $\ModI(\Lambda_X)$ spanned by irregular constructible sheaves by $\Mod_{ic}(\Lambda_X)$.

\begin{proposition}\label{icabelian}
The category $\Mod_{ic}({\Lambda_X})$ is abelian.
\end{proposition}

\begin{proof}
Since $\ModI(\Lambda_X)$ is abelian, it suffices to show kernels, cokernels, images, and coimages of morphisms between irregular constructible sheaves are also irregular constructible sheaves. Let $f\colon \cV\rightarrow \cW$ be a morphism between irregular constructible sheaves. One can take a common $\bC$-Whitney stratification for $\cV$ and $\cW$. Then it suffices to show Lemma \ref{ickernel} below.
\end{proof}

\begin{lemma}\label{ickernel}
Kernels, cokernels, images, coimages of morphisms between irregular local systems are irregular local systems.
\end{lemma}
To prove Lemma \ref{ickernel}, we prepare some notions and lemmas.
\begin{definition}
Let $\phi_i$ ($i=1,2$) be meromorphic functions over $U$ with poles in $D$. We say $\phi_1$ and $\phi_2$ are equivalent if there exists a bounded holomorphic function $\phi$ over $U$ such that $\phi_1=\phi_2+\phi$. We denote the set of meromorphic functions over $(U, D)$ modulo this equivalence relation by $\mathrm{M}(U,D)$.
\end{definition}

Recall that $\Lambda^{\phi_1}_{(U,D)}$ and $\Lambda^{\phi_2}_{(U, D)}$ are canonically for $\phi_1=\phi_2\in M(U, D)$
by Lemma~\ref{boundeds}.

\begin{comment}
\begin{lemma}\label{orthogonal2}
Let $\phi_i$ ($i=1,2$) be meromorphic functions over a disk $\Delta:=\lc z\in \bC\relmid |z|<1\rc$ with poles at $0$ Let $\gamma$ be a ray in $\Delta$ emanating from $0$. If there exists a nonzero morphism between $\Lambda^{\phi_1}_{(\Delta,0)}$ and $\Lambda^{\phi_2}_{(\Delta, 0)}$ on $\Delta\bs \gamma$. Then $\phi_1$ and $\phi_2$ coincide in $\mathrm{M}(\Delta,0)$.
\end{lemma}
\begin{proof}
We assume that $\phi_1\neq \phi_2$ in $\mathrm{M}(\Delta,0)$. By the Riemann extension theorem, the difference $\phi_1-\phi_2$ is unbounded. 

%Take $D'$ an irreducible component of $D$ on which $\phi_1-\phi_2$ is unbounded. Let us take a point $x\in D'$ and a small 1-dimensional disk $\Delta$ centered at $x$ which intersects with $D'$ only at $x$. Suppose there exists a nonzero morphism $f\colon \Lambda_{(U, D)}^{\phi_1}\rightarrow \Lambda_{(U, D)}^{\phi_2}$. Then $f$ induce a  nonzero morphism on $\Delta$.

Since $\phi_1-\phi_2$ is unbounded, there exists a region in $\Delta\bs \gamma$ where $\Re(\phi_2-\phi_1)$ is negatively divergent. Hence there are no nonzero morphisms by Lemma \ref{orthogonality}. This completes the proof. 
\end{proof}
\end{comment}

\begin{proof}[Proof of Lemma \ref{ickernel}]
Let $\cV$ and $\cW$ be irregular local systems over $(U,D)$. Since the definition of irregular local systems is local, we can consider locally on a open subset $U$. There exists a modification $p\colon (U', D')\rightarrow (U, D)$ such that $p^{-1}\cV$ and $p^{-1}\cW$ are both good irregular local systems by Lemma \ref{commonmodification}.

A morphism $f\colon\cV\rightarrow \cW$ induces a morphism over $U$ and we pull-back $f$ by $p$. Then by the exactness of the pull-back, kernel cokernel, image, coimage (we denote those by $A$) of $p^{-1}f$ are pull-backs of those for $f$ i.e., $p^{-1}A(f)\cong A(p^{-1}f)$. 

Furthermore, we can pull-back more by a covering map $\Phi_A$ to make $p^{-1}\cV$ and $p^{-1}\cW$ unramified irregular local systems. Then again, $\Phi_A^{-1}p^{-1}A(f)\cong A(\Phi_A^{-1}\circ p^{-1}f)$. It suffices to show that this is an irregular local system.

So we reset the notations. Let $\cV$ and $\cW$ be unramified good irregular local systems and $f\colon \cV\rightarrow \cW$ be a morphism. Then there exist sets of meromorphic functions $\Phi_\cV$ and $\Phi_\cW$ over $(U, D)$ which are appeared in the definition of irregular local system. 

Take a point $x\in D$, a neighborhood $U$ of $x$, and a sectorial covering $\cU$ of $U\bs D$ for $\cV$ and $\cW$. On each $U\in \cU$, we have isomorphisms $\cV|_U\cong \bigoplus_{\phi\in \Phi_\cV}\Lambda_{(\overline{U}, \overline{U}\cap D)}^\phi$ and $\cW|_U\cong \bigoplus_{\psi\in\Phi_\cW}\Lambda_{(\overline{U}, \overline{U}\cap D)}^\psi$.

Suppose the following; there exists a sector $U\in\cU$ such that the restriction of $f$ to the component $\Lambda^{\phi}_{(\overline{U}, \overline{U}\cap D)}\rightarrow \Lambda^\psi_{(\overline{U}, \overline{U}\cap D)}$ is nonzero where $\phi\in \Phi_\cV, \psi\in \Phi_\cW$ with $\phi\neq\psi$. 

Let $U'$ be the adjacent sector of $U$. Then the restriction of $f$ to $\Lambda^{\phi}_{(\overline{U'}, \overline{U'}\cap D)}\rightarrow \Lambda^\psi_{(\overline{U'}, \overline{U'}\cap D)}$ is nonzero again. This implies $\max\{\Re\phi-\Re\psi\}$ is bounded by Lemma~\ref{boundeds}. We can continue this procedure and we eventually will arrive a sector on which $\phi-\psi$ is negatively divergent since $\phi\neq \psi$. This is a contradiction.

Hence we cannot have such a morphism. This means $f|_{U}$ is diagonal with respect to indices $M(U, D)\times M(U,D)$. Hence the morphism $f|_{U}$ is represented by a sum of $c\cdot T^a\colon \Lambda_{(\overline{U}, \overline{U}\cap D)}^\phi\rightarrow \Lambda^\phi_{(\overline{U}, \overline{U}\cap D)}$ where $c\in k$ by Lemma~\ref{boundeds}. The $A(c\cdot T^a)$ ia again of the form of a sum of $\Lambda^\phi_{(\overline{U}, \overline{U}\cap D)}$. This completes the proof.
\end{proof}

We prepare the following lemma for the next subsection.
\begin{lemma}\label{thickness}
The category $\Mod_{ic}(\Lambda_X)$ is a thick subcategory of $\ModI(\Lambda_X)$.
\end{lemma}
\begin{proof}
Let
\begin{equation}
0\rightarrow \cV\rightarrow \cX\rightarrow \cW\rightarrow 0
\end{equation}
be an exact sequence in $\ModI(\Lambda_X)$ with $\cV, \cW\in \Modic(\Lambda_X)$. Let $\cS$ be a common stratification of $\cV$ and $\cW$. Since pull-backs are exact, we can reduce to the case that $\cV, \cW$ are irregular local systems on $(\oS, D_S)$. For any point $x\in D_S$, there exists a neighborhood $U$ of $x$ such that $U\bs D_S$ has a finite sectorial covering $\{U_i\}$ and $\cV$ (resp. $\cW$) is isomorphic to $\bigoplus_j\Lambda^{\phi_j}_{(\overline{U_i}, \overline{U_i}\cap D_S)}$ (resp. $\bigoplus_{k}\Lambda^{\psi_k}_{(\overline{U_i}, \overline{U_i}\cap D_S)}$). So we have an exact sequence
\begin{equation}
0\rightarrow \bigoplus_{j}\Lambda^{\phi_j}_{(\overline{U_i}, \overline{U_i}\cap D_S)}\rightarrow \cX|_{U_i}\rightarrow \bigoplus_{k}\Lambda^{\psi_k}_{(\overline{U_i}, \overline{U_i}\cap D_S)}\rightarrow 0
\end{equation}
on each $U_i$. 

We have already seen that $\RcHom(\Lambda^\psi_{(\overline{U_i}, \overline{U_i}\cap D_S)}, \Lambda^\phi_{(\overline{U_i}, \overline{U_i}\cap D_S)})\simeq \Lambda_{(\overline{U_i}, \overline{U_i}\cap D_S)}^{\phi-\psi}$ in Corollary~\ref{Rhomstandard}. Then
\begin{equation}
\begin{split}
\Ext^1_{D^b(\ModI(\Lambda_\dS))}(\Lambda_{(\overline{U_i}, \overline{U_i}\cap D_S)}^\psi, \Lambda_{(\overline{U_i}, \overline{U_i}\cap D_S)}^\phi)&\cong \Hom_{D^b(\ModI(\Lambda_\dS))}(\Lambda_{(\overline{U_i}, \overline{U_i}\cap D_S)}^\psi, \Lambda_{(\overline{U_i}, \overline{U_i}\cap D_S)}^\phi[1])\\
&\cong \Hom_{D^b(\ModI(\Lambda_\dS))}(\Lambda_{(\overline{U_i}, \overline{U_i}\cap D_S)}, \Lambda_{(\overline{U_i}, \overline{U_i}\cap D_S)}^{\phi-\psi}[1])\\
&\cong 0,
\end{split}
\end{equation}
since $\Lambda_{(\overline{U_i}, \overline{U_i}\cap D_S)}$ is free. This completes the proof.
\end{proof}

\subsection{Derived category and six operations}
\begin{definition}
Cohomologically irregular constructible $\Lambda_\dX$-module is an object of $D^b(\Mod^\frakI(\Lambda_\dX))$ such that all the cohomologies are irregular constructible sheaves. We denote the full subcategory spanned by those objects by $D^b_{ic}(\Lambda_X)$
\end{definition}

\begin{proposition}
The category $D^b_{ic}(\Lambda_X)$ is a triangulated category.
\end{proposition}
\begin{proof}
This is clear from Lemma \ref{thickness}.
\end{proof}

We will now see Grothendieck six operations on this category.

\subsection*{Tensor}
\begin{proposition}\label{ictensor}
Let $\cV, \cW\in D^b_{ic}(\Lambda_X)$, we have $\cV\dotimes \cW\in D^b_{ic}(\Lambda_X)$.
\end{proposition} 
\begin{proof}
This is obvious from Lemma~\ref{tensorpullback} and Lemma~\ref{tensorstandard}.
\end{proof}

\subsection*{Verdier duality}
First, we prepare the following useful lemma: Let $\dS$ be a topological space with boundary. Let $U$ be an open subset of $S$ and $\overline{U}$ be the closure inside $\overline{S}$.  Consider the map $i\colon (\overline{U}, D_U:=\overline{U}\bs U)\rightarrow \dS$. We denote the closed complement of $U$ in $\overline{S}$ by $V$. We denote the map $j\colon (V, V\cap D_S)\rightarrow \dX$.
\begin{lemma}\label{icrecollement}
There exists an exact triangle:
\begin{equation}
i_!i^{-1}F\rightarrow F\rightarrow j_*j^*F\xrightarrow{[1]}.
\end{equation}
\end{lemma}
\begin{proof}
Note that $i$ and $j$ are tame maps. Then this is clear from the corresponding statement for usual sheaves and the commutativity results for $[\cdot]$ proved in 4.2.
\end{proof}

For the constant map $a_X\colon X\rightarrow *$, we set $\omega_X^\Lambda:=a_X^!\Lambda\cong \Lambda\otimes_\bk \omega_X\in D^b_{ic}(\Lambda_X)$ as usual. We also set
\begin{equation}
\bD\cV:=\RcHom(\cV, \omega^\Lambda_X)\in D^b(\ModI(\Lambda_X)).
\end{equation}

First note tha following:

\begin{lemma}\label{dualiclocal}
For $\Lambda^\phi_{X}\in D^b_{ic}(\Lambda_X)$, we have $\bD\Lambda^\phi_X\cong \Lambda^{-\phi}_X\otimes_\bk\omega_X$.
\end{lemma}
\begin{proof}
This is clear from Lemma~\ref{Rhomstandard}.
\end{proof}

Then we have:
\begin{lemma}
We have $\bD\cV\in D^b_{ic}(\Lambda_{X})$.
\end{lemma}
\begin{proof}
Let $\cS$ be a stratification of $\cV$. Let $U$ be the union of open subsets of $\cS$.
By applying Lemma~\ref{icrecollement}, we have an exact triangle
\begin{equation}
\RcHom(i_!i^!\cV, \omega_X^\Lambda)\leftarrow \RcHom(\cV, \omega_X^\Lambda)\leftarrow \RcHom(j_!j^{-1}\cV, \omega_X^\Lambda)\leftarrow .
\end{equation}
Then we have 
\begin{equation}
\begin{split}
\RcHom(i_!i^!\cV, \omega^\Lambda_X)&\simeq i_!\RcHom(i^{-1}\cV, i^{-1}\omega^\Lambda_X)\\
\RcHom(j_!j^{-1}\cV, \omega^\Lambda_X)&\simeq j_*\RcHom(j^{-1}\cV, j^!\omega^\Lambda_X).
\end{split}
\end{equation}
Since it is clear that irregular constructibility are preserved under $i_!$ and $j_*$, we can prove the desired result by the induction of the dimension of the strata and Lemma~\ref{dualiclocal}.
\end{proof}

\begin{lemma}\label{unipotent}
Let $\cV\in D^b_{ic}(\Lambda_X)$. Then a natural morphism $\cV\rightarrow\bD\bD\cV$ is an isomorphism.
\end{lemma}
\begin{proof}
It is also enough to show the statement for irregular local systems. Then the statement is clear from $\bD\bD\Lambda_X^\phi=\Lambda_X^{\phi}$.
\end{proof}

\begin{lemma}\label{contrafullyfaithful}
Let $\cV, \cW\in D^b_{ic}(\Lambda_X)$. We have
\begin{equation}
\Hom_{D^b_{ic}(\Lambda_X)}(\cV, \cW)\cong \Hom_{D^b_{ic}(\Lambda_X)}(\bD\cW, \bD\cV).
\end{equation}
\end{lemma}
\begin{proof}
We have
\begin{equation}
\begin{split}
 \Hom_{D^b_{ic}(\Lambda_X)}(\bD\cW, \bD\cV)&\cong  \Hom_{D^b_{ic}(\Lambda_X)}(\bD\cW, \RcHom(\cV,\omega^\Lambda_X))\\
 &\cong  \Hom_{D^b_{ic}(\Lambda_X)}(\bD\cW\dotimes \cV, \omega^\Lambda_X)\\
 &\cong \Hom_{D^b_{ic}(\Lambda_X)}(\cV, \RcHom(\bD\cW, \omega^\Lambda_X))\\
 &\cong \Hom_{D^b_{ic}(\Lambda_X)}(\cV, \bD\bD\cW)\\
 &\cong \Hom_{D^b_{ic}(\Lambda_X)}(\cV, \cW).
\end{split}
\end{equation}
This completes the proof.
\end{proof}

\begin{corollary}\label{icVerdier}
The contravariant functor $\bD\colon D^b_{ic}(\Lambda_X)\rightarrow (D^b_{ic}(\Lambda_X))^{op}$ is a contravariant equivalence.
\end{corollary}
\begin{proof}
This is clear from Lemma \ref{contrafullyfaithful} and Lemma \ref{unipotent}.
\end{proof}

\begin{proposition}\label{invshriek}
We have a natural isomorphism
\begin{equation}
f^!\circ \bD\cong \bD\circ f^{-1}
\end{equation}
\end{proposition}
\begin{proof}
For $\cV\in D^b_{ic}(\Lambda_X)$, we have
\begin{equation}
\begin{split}
f^!\circ \bD(\cV)&\simeq f^!\RcHom(\cV, \omega^\Lambda_X)\\
&\simeq \RcHom(f^{-1}\cV, f^!\omega^\Lambda_X)\\
&\simeq \RcHom(f^{-1}\cV, \omega^\Lambda_X)\\
&=:\bD\circ f^{-1}(\cV).
\end{split}
\end{equation}
Here we used Lemma \ref{shriekhom} on the first line. This completes the proof.
\end{proof}

\subsection*{Hom}
\begin{proposition}
Let $\cV, \cW\in D^b_{ic}(\Lambda_X)$, we have $\RcHom(\cV, \cW)\in D^b_{ic}(\Lambda_X)$. 
\end{proposition}
\begin{proof}
As usual, we can see that $\RcHom(\cV, \cW)\simeq \bD(\bD\cW\dotimes \cV).$ Then this is a corollary of the preceding results.
\end{proof}

\subsection*{Pull-backs}
\begin{proposition}
Let $\cV\in D^b_{ic}(\Lambda_X)$, then $f^{-1}\cV, f^!\cV\in D^b_{ic}(\Lambda_X)$.
\end{proposition}
\begin{proof}
For $f^{-1}$, this is clear from the definition.

By using Lemma \ref{unipotent} below, we have
\begin{equation}
\begin{split}
f^!\cV&\cong f^!\bD\bD\cV\\
&\cong f^!\cHom(\bD\cV, \omega_Y^\Lambda)\\
&\cong \cHom(f^{-1}\bD\cV, \omega^\Lambda_X).
\end{split}
\end{equation}
Here the final form is in $D^b_{ic}(\Lambda_X)$.
\end{proof}

\subsection*{Proper push-forwards}
Push-forwards are more difficult and we use irregular Riemann--Hilbert correspondence proved below.
\begin{proposition}\label{pushic}
Let $f\colon X\rightarrow Y$ be a proper morphism. Then $f_*\cV\in D^b_{ic}(\Lambda_Y)$ for $\cV\in D^b_{ic}(\Lambda_X)$.
\end{proposition}
\begin{proof}
For $\cV\in D^b_{ic}(\Lambda_X)$, take $\cV':=\cV\otimes_\bk\bC$. Then we get a holonomic $\cD$-module $\cM:=(\Sol^{\Lambda})^{-1}(\cV')$. Due to Malgrange~\cite{Malgrangeirregular}, there exists a good lattice on $\cM$. Hence the push-forward of $\cM$ along $f$ is again holonomic and we have $f_*\cV'\in D^b_{ic}((\Lambda\otimes_\bk\bC)_Y)$ by \cite{D'Agnolo--Kashiwara} and Lemma \ref{pushsolcommute}. Since the irregular constructibility is preserved under $\otimes_{\bk}\bC$. This completes the proof.
\end{proof}

\subsection{Global $\bR$-graded realization}
The following proposition says ``an object of $D^b_{ic}(\Lambda_X)$ is actually a sheaf over $X$''. It is logically not important, but conceptually makes us feel easy to irregular constructible sheaves. We use some results from the later sections to prove the following. 
\begin{proposition}\label{realization}
The essential image of $[\cdot]\colon D^b(\Mod^0(\Lambda_X))\rightarrow D^b(\ModI(\Lambda_X))$ contains $D^b_{ic}(\Lambda_X)$.
\end{proposition}
\begin{proof}
Let $\cV$ be an irregular constructible sheaf. Then we have $[\bR\tM'(M^{-1}(\cV))]\simeq \cV\in D^b(\ModI(\Lambda_X))$ by Section 7.2 and Theorem \ref{RHmain}. Since $[\cE]\in D^b(\ModI(\Lambda_X))$ for $\cE\in D^b(\Ind(\Mod^0(\Lambda_X)))$ if and only if $\cE\in D^b(\Mod^0(\Lambda_X))$, we have $\bR\tM'(M^{-1}(\cV))\in D^b(\Mod^0(\Lambda_X))$. This completes the proof.
\end{proof}

\begin{comment}
Let us assume $X$ is compact. Let us denote the essential image of $D^b(\ModIp(\Lambda_X))\rightarrow D^b(\ModI(\Lambda_X))$ given by Corollary \ref{globalobject} by the same notation $D^b(\ModIp(\Lambda_X))$.
\begin{corollary}
If $X$ is compact, the category $D^b_{ic}(\Lambda_X)$ is a full subcategory of $D^b(\ModIp(\Lambda_X))$.
\end{corollary}
\begin{proof}
This is clear form Proposition \ref{realization}.
\end{proof}
\end{comment}

\begin{comment}
\begin{remark}
One can ask the reason why we are working with the category $\ModI(\Lambda_X)$ but not with the essential image of $[\cdot]\colon \Mod^0(\Lambda_X)\rightarrow \ModI(\Lambda_X)$. This question is reasonable since irregular constructible sheaves are living in this latter category consisting of ``globally realized objects''. The reason is that it is not clear that the category is abelian or not.
\end{remark}
\end{comment}

\section{Forgetting grading}
In this section, we discuss the relationship between irregular constructible sheaves and constructible sheaves. For a topological space with boundary $\dX$, we set $X:=\oX\bs D_X$.
\subsection{Forgetting grading}
\begin{comment}
We also assume that the following finiteness condition:
\begin{assumption}\label{3.20}
\begin{enumerate}
\item For any open subset $U$, $\cV(U)$ has a finite rank. (Any finitely generated module over $\Lambda$ is principally generated.)
\item There exists a finite collection of open subsets $\{U_i\}$ of $X$ such that (i) $\cV|_{U_i}$ is flabby (ii) $\cV(U)\rightarrow \bigoplus_{i}\cV(U_i\cap U)$ is injective.
\end{enumerate}
\end{assumption}

\begin{lemma}
Irregular constructible sheaves satisfy Assumption \ref{3.20}.
\end{lemma}
\end{comment}
\begin{lemma}\label{exactnessofF}
There exists an exact functor
\begin{equation}
\mathfrak{F}\colon\Mod^\frakI(\Lambda_\dX)\rightarrow \Mod(\bk_X).
\end{equation}
\end{lemma}
\begin{proof}
For an object $\cV$, let us take a locally finite covering $\{U_i\}$ of $X$ with representatives $\{\tcV_i\}\subset \ModIp(\Lambda_X)$. There exists an isomorphism $f_{ij}\colon [\tcV_i]|_{U_i\cap U_j}\xrightarrow{\cong} [\tcV_j]|_{U_i\cap U_j}$ in $\ModIps_\dX(U_{ij})$. We can take a covering $\{U_{ijk}\}$ on which we have a descent data $f_{ijk}\colon [\tcV_i]|_{U_{ijk}}\rightarrow [\tcV_j]|_{U_{ijk}}$ for $f_{ij}$. Let $\tf_{ijk}\colon \tcV_i|_{U_{ijk}}\rightarrow \tcV_j|_{U_{ijk}}\la a\ra$ be a lift of $f_{ijk}$. 

Then $f_{ijk}|_{U_{ijk}\cap U_{ijl}}=f_{ijl}|_{U_{ijk}\cap U_{ijl}}$ means there exists $b\in \bR_{>0}$ such that $T^b\cdot ((\tf_{ijk}-\tf_{ijl})|_{U_{ijk}\cap U_{ijl}})$. This means $\tf_{ijk}\otimes_\Lambda \bk =\tf_{ijl}\otimes_\Lambda \bk$. Hence the set $\{\tf_{ijk}\}$ gives an isomorphism $f_{ij}\otimes_\Lambda \bk\colon \tcV_i|_{U_{ij}}\otimes_\Lambda \bk\rightarrow  \tcV_j|_{U_{ij}}\otimes_\Lambda \bk$. Again, these morphisms can be glued up and give a $\bk$-module sheaf $\cV\otimes_\Lambda \bk$. By a similar argument, one can actually see this does not depend on the choice of lifts.

For $f\in\Hom_{\Mod^\frakI(\Lambda_X)}(\cV, \cW)$, there exists a covering $\{U_i\}$ of $X$ with lifts $\{\tf_i\} \subset \Mod^\bR(\Lambda_{U_i})$. Then we get a set of morphisms $\{\tf_i\otimes_{\Lambda_X} \bk_X\}$. One can see these are glued up to a morphism in $\Mod(\bk_X)$ depending only on $f$ by a similar argument as above. The resulting morphism is denoted by $\frakF(f)$. It is clear that this correspondence preserves the compositions. Hence $\frakF$ gives a functor.

We would like to see the functor $\frakF$ is exact.
Let
\begin{equation}
0\rightarrow \cV\xrightarrow{f} \cW\xrightarrow{g} \cX\rightarrow 0
\end{equation}
be an exact sequence in $\ModI(\Lambda_\dX)$. It is equivalent to that there exists a locally finite open covering $\{U_i\}$ of $X$ such that we have an exact sequence
\begin{equation}
0\rightarrow \cV_i\xrightarrow{f_i} \cW_i\xrightarrow{g_i} \cX_i\rightarrow 0
\end{equation}
over each $U_i$. By Lemma \ref{liftlemma}, it can be lifted to an exact sequence of $\bR$-graded $\Lambda_X$-modules
\begin{equation}
0\rightarrow \tcV_i\xrightarrow{\tf_i} \tcW_i\xrightarrow{\tg_i} \tcX_i\rightarrow 0.
\end{equation}
Since tensor product is left exact, we get an exact sequence
\begin{equation}
\tcV_i\otimes_{\Lambda_X}\bk_X\xrightarrow{\tf_i\otimes \id} \tcW_i\otimes_{\Lambda_X}\bk_X\xrightarrow{\tg_i\otimes \id} \tcX_i\otimes_{\Lambda_X}\bk_X\rightarrow 0.
\end{equation}

It remains to show $\tf_i\otimes \id$ is injective. Let us take a homogeneous section of the kernel of $\tf_i\otimes \bk$. Since it is a subsheaf of $\tcV_i\otimes_{\Lambda_X}\bk_X$, it is locally represented by the form $s\otimes 1$. If $s\otimes 1$ is nonzero, it means that $T^a\cdot s\neq 0$ in $\tcV_i$. Hence we have $\Lambda_U\cdot s\hookrightarrow \tcV_i|_U$ where $U$ is the open set on which $s$ is defined. If $\tf_i(s)\otimes 1=0$, we have some $T^a$ such that $T^a\tf_i(s)=0$ by Lemma \ref{1.105}. Hence we have a sequence of morphisms  over $U$ of $\bR$-graded $\Lambda$-modules
\begin{equation}
\Lambda_U\cdot s\rightarrow \tcV_i\xrightarrow{T^a\tf_i} \tcW_i\la a\ra
\end{equation}
whose composition is zero. Since $\Lambda_U\cdot s$ is nonzero in $\ModI(U_i)$, the morphism $[T^a\tf_i]=[\tf_i]=f_i$ has a nontrivial kernel. This contradicts to the injectivity of $f_i$. Hence $\tf_i\otimes \id$ is injective.
\end{proof}

\begin{lemma}\label{Finv}
Let $f\colon \dX\rightarrow \dY$ be a map between topological spaces with boundaries. Then we have
\begin{equation}
\frakF\circ f^{-1}\cong f^{-1}\circ \frakF.
\end{equation}
\end{lemma}
\begin{proof}
For an $\bR$-graded $\Lambda_X$-module $\cV$, let us consider $f^{-1}\cV$. The sheaf $\frakF\circ f^{-1}\cV(U)$ is a sheaf associated with the presheaf
\begin{equation}
U\mapsto f^{-1}\cV(U)\otimes_{\Lambda}\bk.
\end{equation}
On the other hand, the sheaf $f^{-1}\circ \frakF(\cV)$ is a sheaf associated with the presheaf
\begin{equation}
U\mapsto f^{-1}(\cV\otimes_{\Lambda_X} \bk_X)(U).
\end{equation}
By the definition,
\begin{equation}
\begin{split}
f^{-1}\cV(U)\otimes_{\Lambda}\bk&\cong \lb \lim_{\substack{\longrightarrow \\ V\supset f(U)}}\cV(V)\rb \otimes_\Lambda \bk\\
&\cong  \lim_{\substack{\longrightarrow \\ V\supset f(U)}}(\cV(V)\otimes_{\Lambda}\bk)\\
&\cong  f^{-1}(\cV\otimes_{\Lambda_X} \bk_X)(U).
\end{split}
\end{equation}
Hence they are the same.
\end{proof}

\begin{lemma}\label{iltol}
Let $\cV\in \ModI(\Lambda_\dX)$ be an irregular local system. Then $\frakF(\cV)$ is a local system.
\end{lemma}
\begin{proof}
There exists an open covering of $U$ such that $\cV$ is represented by a direct sum of irregular constant sheaves $\Lambda^\phi$. By the definition of $\frakF$, it is enough to see $\Lambda^\phi\otimes_\Lambda \bk$ is a constant sheaf over any enough small open subset. This is clear.
\end{proof}

\begin{lemma}\label{derivedfunctorF}
Let $G\colon \ModI(\Lambda_\dX)\rightarrow \ModI(\Lambda_\dY)$ and $\tilde{G}\colon \Mod(\bk_X)\rightarrow \Mod(\bk_Y)$ be right (resp. left) exact functors such that $\frakF\circ G\cong G\circ \frakF$. Then we have $\frakF\circ \bR G\simeq \bR \tilde{G}\circ \frakF$ $($resp. $\frakF\circ \bL G\simeq \bL \tilde{G}\circ \frakF$ $)$.
\end{lemma}
\begin{proof}
Let $\cV\in \Mod^\frakI(\Lambda_\dX)$ and take an injective resolution $\cI^\bullet$ by using Proposition \ref{injectives}. Note that skyscraper sheaves $\Lambda_x$ used in this injective resolution are mapped to skyscraper sheaves $k_x$. Combining with the exactness of $\frakF$ (Lemma \ref{exactnessofF}), we can conclude that $\frakF(\cI^\bullet)$ is an inejctive resolution of $\frakF(\cV)$. Hence we have
\begin{equation}
\frakF\circ \bR G(\cV)\simeq \frakF\circ G(\cI^\bullet)\simeq  \tilde{G}\circ \frakF(\cI^\bullet)\simeq \bR \tilde{G}\circ \frakF(\cV).
\end{equation}

Similarly, for a free $\bR$-graded $\Lambda$-module $\cF$, the module $\frakF(\cF_U)$ is a direct sum of $\bk_U$, hence is flat. By Lemma \ref{liftflat}, we can do a similar argument as above. This completes the proof.
\end{proof}

\begin{lemma}\label{Fpush}
Let $f$ be a proper map $X\rightarrow Y$. We have an equality  
\begin{equation}
\frakF\circ \bR f_!\simeq \bR f_!\circ \frakF
\end{equation}
of functors $D^b(\ModI(\Lambda_\dX))\rightarrow D^b(\bk_Y)$.
\end{lemma}
\begin{proof}
By Lemma \ref{derivedfunctorF}, it is enough to show the underived version. For $\cV\in \Mod^\bR(\Lambda_X)$ and an open subset $U$, both $f_!\circ \frakF(\cV)$ and $\frakF\circ f_!$ have $\cV(f^{-1}(U))\otimes \bk$ over $U$. This completes the proof. 
\end{proof}

\begin{lemma}\label{Fpush2}
Let $i_\dX\colon \dX\rightarrow (\oX,\varnothing)$ be the canonical map and $i_X\colon X\hookrightarrow \oX$ be the inclusion. We have an equality
\begin{equation}
\frakF\circ \bR i_{\dX !}\simeq \bR i_{X!}\circ \frakF.
\end{equation}
\end{lemma}
\begin{proof}
Again, we only prove the underived version. One can prove in a similar way to Lemma \ref{Fpush}.
\end{proof}

\subsection{The case of irregular constructible sheaves}

\begin{proposition}\label{ictoc}
The functor $\frakF$ is restricted to $\Mod_{ic}(\Lambda_X)\rightarrow \Mod_{c}(\bk_X)$.
\end{proposition}
\begin{proof}
For $\cV\in \Mod_{ic}(\Lambda_X)$, let us take a stratification $\cS$ of $X$. For each $S\in \cS$, let us denote the inclusions by $i_{\dS}\colon \dS\hookrightarrow (X, \varnothing)$ and $i_S\colon S\hookrightarrow X$. Then we have $i_S^{-1}\frakF(\cV)\cong \frakF(i_\dS^{-1}(\cV))$ by Lemma \ref{Finv}. By Lemma \ref{iltol}, this is a local system. Hence $\frakF(\cV)$ is a constructible sheaf with respect to $\cS$.
\end{proof}

We also denote the induced functor $D^b(\ModI(\Lambda_X))\rightarrow D^b(\Mod(\bk_X))$ by $\frakF$.

\begin{corollary}
The functor $\frakF$ is restricted to $D^b_{ic}(\Lambda_X)\rightarrow D^b_c(\bk_X)$.
\end{corollary}
\begin{proof}
For $\cV^\bullet\in D^b(\Mod_{ic}(\Lambda_X))$, since $\frakF$ is exact on the abelian categories (Lemma \ref{exactnessofF}), we have $H^i(\frakF(\cV^\bullet))\cong \frakF(H^i(\cV^\bullet))$. By Proposition \ref{ictoc}, we have $\frakF(H^i(\cV^\bullet))\in\Mod_c(\bk_X)$.
\end{proof}

\begin{lemma}\label{vanishingF}
If we have $\frakF(E)\simeq 0$ for an irregular constructible sheaf $E$, we have $E\simeq 0$.
\end{lemma}
\begin{proof}
An irregular constructible sheaf is locally isomorphic to $\bigoplus_{i\in I}\Lambda^{\phi_i}$ for some $\phi_i$'s. Since $\frakF(\bigoplus_i\Lambda^{\phi_i})\cong \bk^{|I|}$, $\frakF(E)\cong 0$ is equivalent to $|I|=0$. This means $E\cong 0$. This completes the proof.
\end{proof}

\section{Enhanced sheaves and $\Lambda$-modules}
\subsection{$\bR$-constructible enhanced ind-sheaves}
In this section, we recall the definition of $\bR$-constructible enhanced ind-sheaves. For more detailed accounts, we refer to the original \cite{D'Agnolo--Kashiwara} and the survey \cite{KSD-module}. Let $X$ be a real analytic manifold. Let $\oR$ be the two point compactification of $\bR$ i.e. $\bR\cong (0,1)\hookrightarrow [0,1]=\oR$. We define the category of enhanced ind-sheaves by two-steps: First, we set 
\begin{equation}
D^b(\mathrm{I}\bk_{X\times(\oR,\bR)}):=D^b(\mathrm{I}\bk_{X\times\oR})/D^b(\mathrm{I}\bk_{X\times{\oR\bs \bR}})
\end{equation}
where $D^b(\mathrm{I}\bk_M)$ is the bounded derived category of ind-sheaves over $M$ \cite{MR1827714}. We set $\bk_{t\lesseqgtr 0}:=\bk_{\lc (x,t)\in M\times \oR\relmid t\in \bR, t\lesseqgtr 0 \rc}$. The definition of the convolution product $\potimes$ can be extended to the objects in $D^b(\mathrm{I}\bk_{M\times(\oR,\bR)})$. We set 
\begin{equation}
\mathrm{IC}_{t^*=0}:=\lc K\relmid K\potimes \bk_{\leq 0}\simeq 0, K\potimes \bk_{\geq 0}\simeq 0\rc.
\end{equation}
The category of enhanced ind-sheaves over $X$ is defined by 
\begin{equation}
E^b({\mathrm{I}\bk_M}):=D^b(\mathrm{I}\bk_{M\times \bR_\infty})/\mathrm{IC}_{t^*=0}.
\end{equation}

The triangulated category $E^b(\Ik_M)$ has monoidal operations $\potimes$ and $\cIhom^+$. For a morphism $M\rightarrow N$ of real analytic manifolds, there are associated functors
\begin{align}
Ef_{!!}, Ef_*&\colon E^b(\Ik_M)\rightarrow E^b(\Ik_N),\\
Ef^{-1}, Ef^!&\colon E^b(\Ik_N)\rightarrow E^b(\Ik_M).
\end{align}
They form adjoint pairs $Ef_{!!}\dashv Ef^!$ and $Ef^{-1}\dashv Ef_*$.

We further set
\begin{equation}
\bk^E_X:=\underset{a \rightarrow\infty}{\forlim}\bk_{t\geq a}
\end{equation}
as an object of $E^b({\mathrm{I}\bk_M})$. As usual, $\forlim$ means Ind-colimit.
\begin{definition}
\begin{enumerate}
\item An object $\cE$ of $E^b({\Ik_M})$ is said to be {\em $\bR$-constructible} if there exists an open covering $\{U_i\}$ of $X$ such that there exists an $\bR$-constructible sheaf $\cE_U$ over each $U\times\bR$ such that $\cE|_{U\times \oR}\simeq \cE_U\potimes \bk^E_{U}$.
\item An enhanced $\bR$-constructible ind-sheaf $\cE$ of $E^b({\Ik_M})$ is said to be {\em $\bC$-constructible} if the following holds: There exists an open covering $\{U\}$ of $X$ and a $\bC$-stratification $\cS_U$ for each $U$ such that (i) there exists an $\bR$-constructible sheaf $\cE_{U_i}$ such that $\cE|_{U \times \bR}\simeq \cE_U\potimes\bk^E_U$, (ii) each cohomology sheaf $\cH^i(\cE_U|_S)$ for each $S\in \cS_U$ are isomorphic to a direct sum of sheaves of the form $\bk_{t\geq \phi(x)}$ for some continuous function $\phi$.
\end{enumerate}
\end{definition}
We denote the full subcategory spanned by $\bR$-constructible (resp. $\bC$-constructible) enhanced ind-sheaves by $E^b_{\bR\hi c}(\Ik_M)$ (resp. $E^b_{\bC\hi c}(\Ik_M)$). The category $E^b_{\bR\hi c}(\Ik_M)$ has a contravariant autoequivalence $\bD$, analogous to the Verdier dual.

\subsection{From enhanced sheaves to $\Lambda$-modules}
For a sheaf $\cE$ on $X\times \oR$, let us consider an object $\bigoplus_{-a\in \bR}p_*\Gamma_{[-a,\infty)}\cE$ where $p\colon X\times \bR\rightarrow X$ is the projection. This is equipped with the action of $\Lambda$ as follows: The action of $T^b$ on $M(\cE)$ is the product of
\begin{equation}
p_*\Gamma_{[-a,\infty)}\cE\rightarrow p_*\Gamma_{[-b-a,\infty)}\cE
\end{equation}
induced by the canonical map $\bk_{[-b-a,\infty)}\rightarrow \bk_{[-a,\infty)}$.

For an object $\cE=\underset{\substack{\longrightarrow\\ i}}{\forlim}\cE_i\in \Mod(\Ik_{M\times \oR})$, we set
\begin{equation}
\begin{split}
\tM'(\cE)&:=\forlimit{i}\lb\bigoplus_{-a\in \bR}p_*\Gamma_{[-a,\infty)}\cE_i\rb \in \Ind(\Mod^0(\Lambda_X)),\\
\tM(\cE)&:=[\cdot]\circ \tM':=\forlimit{i}\left[ \bigoplus_{-a\in \bR}p_*\Gamma_{[-a,\infty)}\cE_i\right] \in \Ind(\ModI(\Lambda_X)).
\end{split}
\end{equation}

\begin{lemma}
The correspondence $\tM$ is a left exact functor $\Mod(\Ik_{X\times \oR})\rightarrow \Ind(\ModI(\Lambda_X))$.
\end{lemma}
\begin{proof}
Since $\tM$ is a functor obtained as the inditization of a left exact functor, which is again left exact \cite{categoriesandsheaves}.
\end{proof}

We denote the right derived functor of $\tM$ by $\bR\tM\colon D^b(\Ik_{X\times \oR})\rightarrow D^b(\Ind(\ModI(\Lambda_X)))$. Recall that there exist embeddings 
\begin{equation}
(-)\potimes \bk_{t\geq 0}\colon E^b(\Ik_X)\rightarrow D^b(\Ik_{(X\times\oR, X\times \bR)})
\end{equation}
and
\begin{equation}
(-)\otimes \bk_{X\times \bR}\colon D^b(\Mod(\Ik_{(X\times \oR, X\times \bR)}))\rightarrow D^b(\Ik_{X\times \oR})
\end{equation}
Composing these with $\bR\tM$, we get 
\begin{equation}
M:=\bR\tM(((-)\potimes \bk_{\geq 0})\otimes \bk_{X\times \bR})\colon E^b(\Ik_X)\rightarrow D^b(\Ind(\ModI(\Lambda_X))).
\end{equation}

\begin{lemma}\label{stabM}
Let $\cE$ be an $\bR$-constructible sheaf over $X\times \bR$. Then we have $M(\cE\potimes \bk^E_X)\in D^b(\ModI(\Lambda_X))$.
\end{lemma}
\begin{proof}
By the definition of $M$ and $\bk^E_X$, it is enough to show that the natural morphisms
\begin{equation}
\left[ \bigoplus_{-a\in \bR}\bR p_*\RGamma_{[-a,\infty)}\cE\right] \rightarrow \left[ \bigoplus_{-a\in \bR} \bR p_*\RGamma_{[-a-c,\infty)}\cE\right] 
\end{equation}
are isomorphisms for any $c\in \bRz$. The cone is given by
\begin{equation}
\left[ \bigoplus_{-a\in \bR}\bR p_*\RGamma_{[-a-c,-a)}\cE\right].
\end{equation}
Since $T^c\in \Lambda$ is vanished on this object, this is zero. Hence the morphisms are isomorphisms.
\end{proof}

\begin{lemma}
The functor $M$ is restricted to a functor $E^b_{\bR\hi c}(\Ik_X)\rightarrow D^b(\ModI(\Lambda_X))$, which is also denoted by $M$.
\end{lemma}
\begin{proof}
For an $\bR$-constructible enhanced ind-sheaf $\cE$, there exists a locally finite covering $\cU$ of $X$ such that we have $\cE|_{U\times \bR}\simeq \cE_U\potimes \bk^E_U$ and $(n+2)$-fold covers are empty. By the Cech construction, $\cE$ is represented by a result of mapping cones of $i_!(\cE_U\potimes \bk^E_U)$. This implies $\tM(\cE)$ is obtained as a finite mapping cones of $\tM(i_!(\cE_U\potimes \bk_U^E))$. By Lemma \ref{stabM}, this means that $\tM(\cE)$ is in $D^b(\ModI(\Lambda_X))$. This completes the proof.
\end{proof}

Let $S$ be a locally closed subset in $X$ and $\overline{S}$ be the closure of $S$ in $X$ and set $D_S:=\overline{S}\bs S$. Let $\phi$ be a $\bC$-valued function on $S$ and take $\Lambda^\phi_\dS$. Since there exists a tame map $i_\dS\colon \dS\rightarrow X$, we get $i_{\dS !}\Lambda^\phi_\dS\in \ModI(\Lambda_X)$.
We also  set $\cE^\phi:=\bk_{\Re \phi\leq t}\potimes \bk^E_X\in E^b(\Ik_X)$.

\begin{lemma}
We have $M(\cE^\phi)\cong i_{\dS !}\Lambda^\phi$.
\end{lemma}
\begin{proof}
This is clear from the definitions and Lemma \ref{stabM}.
\end{proof}

\begin{lemma}\label{irrconstantff}
There exists a canonical isomorphism
\begin{equation}
\Hom_{\ModI(\Lambda_X)}(i_{\dS !}\Lambda^{\phi}, i_{\dS !}\Lambda^{\phi'})\cong \Hom_{E^b(\Ik_X)}(\cE^\phi, \cE^{\phi'}).
\end{equation}
\end{lemma}
\begin{proof}
By Lemma \ref{boundeds} and Lemma \ref{orthogonality}, we have
\begin{equation}
\Hom_{\ModI(\Lambda_X)}(i_{\dS !}\Lambda^{\phi}, i_{\dS !}\Lambda^{\phi'})\cong 
\begin{cases}
\bk &\text{ $\max\{0, \Re\phi-\Re\phi'\}$ is bounded }\\
0 & \text{ otherwise.}
\end{cases}
\end{equation}
It is standard to see that the RHS also has the same formula. In the case that $\max\{0, \Re\phi-\Re\phi'\}$ is bounded, there exists $c\in \bRz$ such that $\Re\phi< \Re\phi'+c$ everywhere. For a map $f\in \Hom_{E^b(\Ik_X)}(\cE^\phi, \cE^{\phi'})$, we have a representative $\tf\colon \bk_{\lc t\geq \Re\phi(x)\rc}\rightarrow \bk_{\lc t\geq \Re\phi'(x)+c\rc}$ of usual $\bR$-constructible sheaves. Then $\tf$ induces a morphism $i_{\dS !}\Lambda^\phi\rightarrow i_{\dS !}\Lambda^{\phi'}$. It is easy to see that the induced morphism only depends on the choice of $f$. By the proof of Lemma \ref{boundeds}, this gives an isomorphism. 
\end{proof}

\section{Irregular Riemann--Hilbert correspondence}
In this section, we will prove our version of the irregular Rimann--Hilbert correspondence as a corollary of D'Agnolo--Kashiwara's one. In this section, we will work over $\bC$.

\subsection{Notations for analytic $\cD$-modules}
We refer the theory of analytic $\cD$-modules to \cite{KashiwaraD}. In this subsection, we simply recall the notations. For a complex manifold,  $\cD_X$ is the sheaf of differential operators, $\Mod(\cD_X)$ is the category of left $\cD$-modules, and $D^b(\cD_X)$ is the bounded derived category of $\cD$-modules. We denote the full subcategory of $D^b(\cD_X)$ spanned by cohomologically holonomic $\cD$-modules by $D^b_{\hol}(\cD_X)$.

The Verdier dual $\bD$ is a contravariant autoequivalence of $D^b_{\hol}(\cD_X)$. For a morphism of complex manifolds $f\colon X\rightarrow Y$, we can define the following functors:
\begin{align}
\int_{f}&\colon D^b(\cD_X)\rightarrow D^b(\cD_Y); \cM\mapsto \bR f_*(\cD_{X\leftarrow Y}\dotimes_{\cD_X}\cM)\\
f^\dagger&\colon D^b(\cD_Y)\rightarrow D^b(\cD_X);\cN\mapsto \cD_{Y\rightarrow X}\dotimes_{f^{-1}\cD_Y}f^{-1}\cN[\dim X-\dim Y].
\end{align}
by using transfer $\cD$-modules $\cD_{X\leftarrow Y}$ and $\cD_{X\rightarrow Y}$.
The functor $f^\dagger$ always preserves cohomologically holonomic modules. If $f$ is proper, $\int_f$ also preserves cohomologically holonomic modules. For a proper $f$, the pair of functors form an adjoint pair $\int_f \dashv f^\dagger$. We also set $f^\star:=\bD\circ f^\dagger \circ \bD\colon D^b_{\hol}(\cD_Y)\rightarrow D^b_{\hol}(\cD_X)$ and $f_\star:=\bD\circ \int_f\circ \bD$. Then $f^\star\dashv f_\star$

\subsection{Irregular Riemann--Hilbert correspondence using enhanced sheaves}
We recall the irregular Riemann--Hilbert correspondence by D'Agnolo--Kashiwara:
\begin{theorem}[\cite{D'Agnolo--Kashiwara}]
There exists a contravariant embedding
\begin{equation}
\Sol^E\colon D^b_{\mathrm{hol}}(\cD_X)\hookrightarrow E^b_{\bR\hi c}(\mathrm{I}\bC_X).
\end{equation}
\end{theorem}
Our convention is slightly different from the original one in \cite{D'Agnolo--Kashiwara}: Let $\Sol^{\bE}$ be the original one. We set $\Sol^E:=\Sol^\bE[\dim X]$ We have $\Sol^E:=\bD\circ \DR^E$ where $\DR^E$ is the same as the original one. We collect some properties of the irregular Riemann--Hilbert correspondence as follows:
\begin{proposition}[{\cite[Theorem 9.4.8, Proposition 9.4.10]{D'Agnolo--Kashiwara}}]

\begin{enumerate}
\item There exists a canonical isomorphism $\bD\circ \DR^E\simeq \DR^E\circ \bD$.
\item For a morphism $f\colon X\rightarrow Y$ of complex manifolds, there exists an isomorphism $\DR^E\circ f^\dagger \simeq Ef^!\circ \DR^E $.
\item For a proper map $f\colon X\rightarrow Y$, we have $\DR^E\circ \int_f\simeq Ef_*\circ \DR^E$. 
\item There exists an isomorphism $\Sol^E(\cM\boxtimes\cN)\simeq \Sol^E(\cM)\pboxtimes\Sol^E(\cN)$ for $\cM\in D^b_{\hol}(\cD_X)$ and $\cN\in D^b_{\hol}(\cD_Y)$.
\end{enumerate}
\end{proposition}

We will also use the following fundamental result. Let $Y$ be an analytic hypersurface of the complex manifold $X$. Take a meromorphic function $\phi$ with poles in $Y$; $\phi\in \cO_X(*Y)$. We set $\cE^\phi:=(\cD_X\cdot e^\phi)(*Y)$.

Our convention for $\Sol^E$ is shifted from D'Agnolo--Kashiwara's one to hold the following:
\begin{proposition}[{\cite[Lemma 9.3.1]{D'Agnolo--Kashiwara}}]\label{expimage}
There exists an isomorphism 
\begin{equation}
\Sol^E(\cE^\phi)\simeq k^E_X\potimes k_{t\geq \Re\phi(x)}[\dim X].
\end{equation}
\end{proposition}

\subsection{Irregular Riemann--Hilbert correspondence}

Let us denote the essential image of $\Sol^E$ by $E^b_{\cD}(\IC_X)$.

\begin{lemma}\label{8.4}
The object $M(\cE)$ is irregular constructible for $\cE\in E^b_\cD(\IC_X)$.
\end{lemma}
\begin{proof} Let us take a holonomic $\cD$-module $\cM$ and consider $\cE:=\Sol^E(\cM)$. Let $Y$ be a divisor of $X$ containing the singularities of $\cE$. We set $\cM(*Y):=\cM\otimes \cO(*Y)$. We have an exact triangle
\begin{equation}
\cM\rightarrow \cM(*Y)\rightarrow \cM_1\xrightarrow{[1]}
\end{equation}
where $\cM_1$ is a $\cD$-module supported on $Y$. By Theorem 5.5 and Proposition~\ref{expimage}, the image $M(\Sol^E(\cM(*Y)))$ is an irregular constructible. Hence it suffices to show that $M(\Sol^E(\cM_1))$ is irregular constructible sheaf.

Let $\pi\colon Y'\rightarrow Y$ be a result of normalization and a resolution of singularities of $Y$. Let $E$ be the inverse image of the union of singularities of $Y$ and $\cM_1$. Then there exists a canonical morphism
\begin{equation}
\cM_1\rightarrow \cM_1':=\pi_\star (\pi^\star\cM)(*E)
\end{equation}
Since $(\pi^\star\cM)(*E)$ is a meromophic connection, $M(\Sol^E(\cM_1'))$ is irregular constructible.

Since the cone of this morphism is living on a divisor of $Y$, we can prove the desired statement by iterating these arguments.
\end{proof}

Our version of irregular Riemann--Hilbert correspondence is the following:
\begin{theorem}\label{RHmain}
The functor $M$ is a contravariant exact equivalence:
\begin{equation}
M\colon E^b_{\cD}(\IC_X)\xrightarrow{\simeq} D^b_{ic}(\Lambda_X).
\end{equation}
In particular, there exists a contravariant equivalence
\begin{equation}
\Sol^\Lambda:=M\circ \Sol^E\colon D^b_{\hol}(\cD_X)\xrightarrow{\simeq}D^b_{ic}(\Lambda_X).
\end{equation}
\end{theorem}
\begin{proof}
First, we will prove the fully faithfulness. Let $\cM, \cN$ be a holonomic $\cD_X$-modules. Set $\cE:=\Sol^E(\cM)$ and $\cF:=\Sol^E(\cN)$. Then we have
\begin{equation}
\Hom(\cM, \cN)\cong \Hom(\cF, \cE).
\end{equation}
Let $Y$ be a divisor containing the singularities of $\cM$ and $\cN$. Then we have
\begin{equation}
\Hom(\cM, \cN(*Y))\cong \Hom(\cM, \cHom(\cO(*Y), \cN(*Y)))\cong \Hom(\cM(*Y), \cN(*Y)).
\end{equation}
By Lemma~\ref{irrconstantff}, we have
\begin{equation}
\Hom(\cM(*Y), \cN(*Y))\cong \Hom(M(\Sol^E(\cN(*Y))), M(\Sol^E(\cM(*Y))).
\end{equation}
Let $\cN_1$ be the cone of a canonical morphism $\cN\rightarrow \cN(*Y)$. Then $\cN_1$ is supported on $Y$. By a similar procedure done in the proof in Lemma~\ref{8.4}. We complete the proof of the fully faithfulness.

Now we will prove the essential surjectivity. We only have to see that the functor hit each irregular local system by Lemma \ref{icrecollement}.

Let $\cV$ be an irregular local system on $\dX$. After a modification, we get a good irregular local system. This associates a good Stokes local system in the sense of \cite{SabbahStokes}. By the Riemann-Hilbert correspondence in \cite{SabbahStokes}, this gives a meromorphic connection over the modification. By pushing forward to the original base, we get a desired object.
\end{proof}

\begin{remark}
One can also prove Theorem \ref{RHmain} by using curve test criterion by Mochizuki \cite{Mochizukicurvetest}.
\end{remark}

\begin{corollary}
There exists an exact equivalence
\begin{equation}
\DR^\Lambda:=\bD\circ \Sol^\Lambda\colon D^b_{\hol}(\cD_X)\xrightarrow{\simeq} D^b_{ic}(\Lambda_X).
\end{equation}
\end{corollary}
\begin{proof}
This is clear from Theorem \ref{RHmain} and Corollary \ref{icVerdier}.
\end{proof}

\subsection{Functors}
In this subsection, we prove the commutativity between $\Sol^\Lambda$ and various functors. We assume that all the spaces are without boundary in this subsection.

\begin{proposition}\label{commpull}
Let $f\colon X\rightarrow Y$. We have
\begin{equation}
M\circ Ef^*\simeq f^{-1}\circ M.
\end{equation}
In particular,
\begin{eqnarray}
\Sol^\Lambda\circ f^\dagger&\simeq f^{-1}\circ \Sol^\Lambda.
\end{eqnarray}
\end{proposition}
\begin{proof}
Since we know the functors are commutative wth $\Sol^E$, it is enought to see the commutativity with the functor $M$. Let us take an $\bR$-constructible sheaf $\cE$ on $X\times \oR$. Let $\overline{f}$ be the direct product of $f\colon X\rightarrow Y$ and $\id\colon \oR\rightarrow \oR$. Then we have
\begin{equation}
p_*\RGamma_{X\times [a, \infty)}(\overline{f}^{-1}\cE)\simeq f^{-1}p_*\RGamma_{Y\times [a,\infty)}\cE.
\end{equation}
Hence we have $M(\overline{f}^{-1}\cE)\simeq f^{-1}M(\cE)$. This proves the first line.
\end{proof}

\begin{lemma}\label{boxdecomposition}
We have
\begin{equation}
M(-\boxtimes -)\simeq M(-)\boxtimes M(-).
\end{equation}
For $\cM\in D^b_{\hol}(\cD_X)$ and $\cN\in D^b_{\hol}(\cD_Y)$, we have
\begin{equation}
\Sol^\Lambda(\cM\boxtimes \cN)\simeq \Sol^\Lambda(\cM)\boxtimes \Sol^\Lambda(\cN).
\end{equation}
\end{lemma}
\begin{proof}
By \cite{D'Agnolo--Kashiwara}, we have $\Sol^E(\cM\boxtimes \cN)\simeq \Sol^E(\cM)\pboxtimes\Sol^E(\cN)$. Hence it suffices to prove $M(\Sol^E(\cM)\pboxtimes\Sol^E(\cN))\simeq M(\Sol^E(\cM))\boxtimes M(\Sol^E(\cN))$.

First, note that we have $p_*\RGamma_{[a,\infty)}(\cE\pboxtimes \cF)\simeq p_*\RGamma_{t_1+t_2\geq a}(\cE\boxtimes \cF)$. We also have a map
\begin{equation}
p_*\RGamma_{[b, \infty)\times [c, \infty)}(\cE\boxtimes \cF)\rightarrow p_*\RGamma_{t_1+t_2\geq b+c}(\cE\boxtimes \cF).
\end{equation}
By combining these, we get a map $M(\cE)\boxtimes M(\cF)\rightarrow M(\cE\pboxtimes \cF)$. It suffices to show that this map is locally an isomorphism. This is clear from Lemma \ref{tensorstandard} and an easy observation $\cE^{\phi_1}\pboxtimes \cE^{\phi_2}\simeq \cE^{\phi_1\boxtimes \phi_2}$.
\end{proof}

\begin{proposition}\label{tensorhomMcomm}
We have 
\begin{equation}
M(-\otimes -)\simeq M(-)\otimes M(-).
\end{equation}
For $\cM, \cN\in D^b_{\hol}(\cD_X)$, we have
\begin{eqnarray}
\Sol^\Lambda(\cM\otimes \cN)\simeq&\Sol^\Lambda(\cM)\otimes \Sol^\Lambda(\cN)[-\dim X].
\end{eqnarray}
We also have
\begin{equation}
M\circ \cHom^E(-,-)\simeq \cHom(M(-), M(-)).
\end{equation}
\end{proposition}
\begin{proof}
Let $\delta\colon X\rightarrow X\times X$ be the diagonal embedding. Then $\cM\otimes \cN\cong \delta^{\dagger}(\cM\boxtimes \cN)$. Then we have,
\begin{equation}
\begin{split}
\Sol^\Lambda(\delta^{\dagger}(\cM\boxtimes \cN))&\simeq \delta^{-1}(\Sol^\Lambda(\cM\boxtimes \cN))[-\dim X]\\
&\simeq \delta^{-1}(\Sol^\Lambda(\cM)\boxtimes \Sol^\Lambda(\cN))[-\dim X]\\
&\simeq \Sol^\Lambda(\cM)\otimes \Sol^\Lambda(\cN)[-\dim X]
\end{split}
\end{equation}
where we used Lemma \ref{tensordiagonal} and Lemma \ref{boxdecomposition}. The second claim follows form the adjunction. This completes the proof.
\end{proof}

\begin{proposition}
Let $f\colon X\rightarrow Y$. We have
\begin{eqnarray}
\Sol^\Lambda\circ f^\star&\simeq f^!\circ \Sol^\Lambda,\\
\Sol^\Lambda\circ \bD&\simeq  \bD\circ \Sol^\Lambda.
\end{eqnarray}
\end{proposition}
\begin{proof}
The first one is followed by the second one and Proposition \ref{commpull}. 

We have
\begin{equation}
\begin{split}
\bD\circ \Sol^\Lambda(\cE)&\simeq \cHom(\Sol^\Lambda(\cE), \omega^\Lambda_X)\\
&\simeq M\circ \cHom^E(\Sol^E(\cE), \omega^E_X)\\
&\simeq M\circ \Sol^E(\bD(\cE))\\
&\simeq \Sol^\Lambda\circ \bD(\cE).
\end{split}
\end{equation}
where we used Proposition \ref{tensorhomMcomm} and the commutativity of $\Sol^E$ with $\bD$ (\cite{D'Agnolo--Kashiwara}). This completes the proof of the third line.
\end{proof}

\begin{proposition}\label{pushsolcommute}
Let $f\colon X\rightarrow Y$ be a proper map. Then we have
\begin{equation}
\DR^\Lambda\circ \int_{f!}\simeq f_!\circ \DR^\Lambda
\end{equation}
\end{proposition}
\begin{proof}
By various adjunctions, we have
\begin{equation}
\begin{split}
\Hom\lb \DR^\Lambda\lb\int_{f!}\cM\rb, \DR^\Lambda(\cN)\rb&\cong \Hom\lb \int_{f!}\cM,\cN\rb\\
&\cong \Hom(\cM, f^\dagger\cN)\\
&\cong \Hom(\DR^\Lambda(\cM), \DR^\Lambda(f^\dagger\cN))\\
&\cong \Hom(\DR^\Lambda(\cM), f^!\DR^\Lambda (\cN))\\
&\cong \Hom(f_!\DR^\Lambda(\cE), \DR^\Lambda \cF).
\end{split}
\end{equation}
\end{proof}

\section{Irregular perverse sheaves}
In this section, we define the irregular perverse t-structure on the category of irregular constructible complexes. Over $\bC$, the heart is equivalent to the category of holonomic $\cD$-modules. We also prove t-exactness of various functors.
\subsection{Definition}
For an object $\cV$ of $D^b_{ic}(\Lambda_X)$, we define the support by
\begin{equation}
\supp(\cV):=\bigcup_{j}\supp(\frakF(H^j(\cV))) \subset X
\end{equation}

Let us define the irregular perverse t-structure.
\begin{definition}\label{irregularperversetstr}
Let $\pD_{ic}^{\leq 0}(\Lambda_X)$ (resp. $\pD_{ic}^{\geq 0}(\Lambda_X)$) be the full subcategory of $D^b_{ic}(\Lambda_X)$ spanned by objects satisfying 
\begin{equation}
\begin{split}
&\dim\lc\supp H^j(\cV)\rc\leq -j \\
&(\text{resp.}\dim\lc\supp H^j(\bD\cV)\rc\leq -j ) \text{ for each }j\in \bZ.
\end{split}
\end{equation}
\end{definition}

\begin{comment}
\begin{lemma}\label{Dperverse}
The condition
\begin{equation}
\dim\lc\supp H^j(\bD\cV)\rc\leq -j \text{ for each }j\in \bZ
\end{equation}
is equivalent to the condition
\begin{equation}
\dim\lc\supp H^j(\bD\frakF(\cV))\rc\leq -j \text{ for each }j\in \bZ.
\end{equation}
\end{lemma}
\begin{proof}
\end{proof}
\end{comment}

Let $\pD^{\leq 0}(\bk_X)$ be the perverse t-structure of $D^b_c(\bk_X)$.
\begin{lemma}\label{Fpexact}
We have $\frakF(\pD_{ic}^{\leq 0}(\Lambda_X))\subset \pD^{\leq 0}(\bk_X)$. Conversely, if $\frakF(\cV)\in  \pD^{\leq 0}(\bk_X)$ for $\cV\in D^b_{ic}(\Lambda_X)$, we have $\cV\in \pD_{ic}^{\leq 0}(\Lambda_X)$.
\end{lemma}
\begin{proof}
Since $\frakF$ is t-exact with respect to the standard t-structure (Lemma \ref{exactnessofF}), we have $H^i(\frakF(\cV))\cong \frakF(H^i(\cV))$. 
\end{proof}

\begin{proposition}\label{tstr}
The pair $(\pD_{ic}^{\leq 0}(\Lambda_X), \pD_{ic}^{\geq 0}(\Lambda_X))$ forms a t-structure.
\end{proposition}
To prove this proposition, we first prepare the following lemma.
\begin{lemma}\label{vanishinglemma}
We have
\begin{equation}
H^ji_\dS^{-1}\cF\simeq 0 \text{ $(j>-\dim S)$ }, H^ji^!_\dS\cG\simeq 0 \text{ $(j<1-\dim S)$}. 
\end{equation}
for $\cF\in \pD_{ic}^{\leq 0}(\Lambda_X)$ and $\cG\in \pD_{ic}^{\geq 1}(\Lambda_X)$.
\end{lemma}
\begin{proof}
Note that the same statement for usual perverse t-structure is known (e.g \cite[Proposition 8.1.22]{HTTD}). By the commutativity proved in Section 6, the first statement follows from Lemma \ref{vanishingF}. The second statement is followed by the Verdier dual.
\end{proof}

\begin{proof}[Proof of Proposition \ref{tstr}]
First, we will prove that for $\cF\in \pD_{ic}^{\leq 0}(\Lambda_X)$ and $\cG\in \pD_{ic}^{\geq 1}(\Lambda_X)$, the vanishing 
\begin{equation}\label{vanishing}
\Hom(\cF, \cG)=0.
\end{equation}
Take a Whitney stratification $X=\sqcup_{S\in \cS}S$. We have
\begin{equation}
i^!_\dS\bR\cHom(\cF, \cG)\simeq \bR\cHom(i_\dS^{-1}\cF, i^!_\dS\cG).
\end{equation}
by Lemma \ref{shriekhom}.

By Lemma \ref{vanishinglemma}, for any $S$, we have
\begin{equation}
H^j\RGamma_{S}\bR\cHom(\cF, \cG)\simeq 0 \text{ ($j<1$)}.
\end{equation}
Then (\ref{vanishing}) can be proved by the induction in \cite{self-dual} (see also \cite{HTTD}).

It remains to show that the decomposition of objects in $D^b_{ic}(\Lambda_X)$. One can prove by a usual argument for perverse sheaves as in \cite[Theorem 8.1.27]{HTTD}. 
\end{proof}

\begin{definition}
The heart of t-structre is called the category of irregular perverse sheaves and denoted by $\Ierv(k_X)$.
\end{definition}

\begin{theorem}\label{main}
The functor $\Sol^\Lambda$ restrict to a contravariant equivalence
\begin{equation}
\Mod_{\mathrm{hol}}(\cD_X)\xrightarrow{\simeq} \Ierv(\bC_X).
\end{equation}
\end{theorem}

\begin{lemma}\label{6.11}
Let $\cD_i$ $(i=1,2)$ be triangulated categories with t-structures $(D^{\leq 0}_i, D_i^{\geq 0})$. Let $F\colon \cD_1\rightarrow \cD_2$ be a t-exact equivalence. Then $F$ gives an equivalence between t-structures.
\end{lemma}
\begin{proof}
We have to show that $F\colon D^{\leq 0}_1\rightarrow D^{\leq 0}_2$ is essentially surjective. Let $\cE$ be an object of $D^{\leq 0}_2$. Then we have a standard triangle
\begin{equation}
\tau_{\leq 0}F^{-1}(\cE)\rightarrow F^{-1}(\cE)\rightarrow \tau_{\geq 1}F^{-1}(\cE)\xrightarrow{[1]}
\end{equation}
By applying $F$ again, we have $F(\tau_{\geq 1}F^{-1}(\cE))\cong 0$ since $\cE\in D^{\leq 0}_2$. Since $F$ is an equivalence, we have $\tau_{\geq 1}F^{-1}(\cE)\cong 0$. Hence $F^{-1}(\cE)\in D^{\leq 0}_1$. We can prove for the positive part in a similar manner. This completes the proof.
\end{proof}

\begin{proof}[Proof of Theorem \ref{main}]
By Lemma \ref{6.11}, it is enough to show that $\Sol^\Lambda$ is t-exact. We only show the condition
\begin{equation}\label{claim}
\dim \{\supp H^j(\cV)\}\leq -j.
\end{equation}
The other case follows from the Verdier duality.

Let $\cM$ be a holonomic $\cD$-module.

We will prove by the dimensional induction: Let us assume the assertion is true for any complex manifold with $\dim<\dim X$.

Let $D$ be a divisor containing the singularities of $\cM$. Let $p\colon X'\rightarrow X$ be a resolution making $D$ normal crossing. We denote the normal crossing divisor by $D'$. Then $p^\dagger\cM$ is again a holonomic $\cD_{X'}$-module (not a complex), since $p$ is a submersion. If $p^\dagger\cM$ satisfies the claim, we have a sequence of inequalities
\begin{equation}
\begin{split}
-j&\geq \dim \supp \cH^j(\Sol^\Lambda(p^\dagger \cM))=\dim \supp \cH^j(p^{-1}\Sol^\Lambda(\cM))=\dim \supp p^{-1}\cH^j(\Sol^\Lambda(\cM))\\
&\geq \dim \supp \cH^j(\Sol^\Lambda(\cM))
\end{split}
\end{equation}
This means the claim also holds for $\cM$. Hence we can assume that $D$ is normal crossing.

Since the claim is local, we can further assume that $D$ is a simple normal crossing: there is a set of coordinate hyperplanes $\{D_i\}$ such that $D=\bigcup_{i\in I}D_i$. Consider the following exact triangle:
\begin{equation}
C\rightarrow \cM\rightarrow \cM(*D)\rightarrow,
\end{equation}
Hence we also have
\begin{equation}
\Sol^\Lambda(\cM(*D))\rightarrow \Sol^\Lambda(\cM)\rightarrow \Sol^\Lambda(C)\rightarrow 
\end{equation}
By Lemma \ref{expimage}, the complex $\Sol^\Lambda(\cM(*D))$ is concentrated in degree $-\dim X$, it is enough to show the claim (\ref{claim}) for $\Sol^\Lambda(C)$. 

The holonomic complex $C$ is supported on $D$ and $\cH^j(C)=0$ for $j\neq 0,1$. Hence the degree of $\bD C$ is concentrated on $-1,0$. Let $\iota_i\colon D_i\hookrightarrow X$ be the closed embedding. Then by \cite[Proposition 1.5.16]{HTTD}, the functor $\iota_i^\dagger$ is a right derived functor. Hence $\iota_i^\dagger \bD C$ is concentrated in $-1,0,1$. Hence $\iota_i^\star C:=\bD\iota_i^\dagger \bD C$ is also concentrated in degree $-1,0,1$. Let us write as $\iota_i^\star C=F_{-1}\xrightarrow{d} F_0\rightarrow F_1$. Then the truncation
\begin{equation}
C_i:=\coker d\rightarrow F_1
\end{equation}
has a natural map $\iota_i^\star C\rightarrow C_i$. By summing up the morphisms, we get a sequence of morphisms $C\rightarrow \bigoplus_{i\in I}\int_{\iota_i}\iota_i^\star C\rightarrow \bigoplus_{i\in I}\int_{\iota_i}C_i$. Then we define the exact triangle
\begin{equation}
C^1\rightarrow C\rightarrow \bigoplus_{i\in I}\int_{\iota_i}C_i\rightarrow.
\end{equation}
Then $C^1$ has degree $0,1,2$. Since $\Sol^\Lambda(\int_{\iota_i}{C}_i)=\iota_{i*}\Sol^\Lambda({C}_i)$ and the assumption of the induction, the complex $\Sol^\Lambda(\int_{\iota_i}C_i)$ satisfies the claim (\ref{claim}). Hence it is enough to show for $C^1$. Since the morphism $C\rightarrow \bigoplus_{i\in I}\int_{\iota_i}C_i$ is an isomorphism on the complement of intersections of $D_i$'s, the complex $C^1$ is supported on intersections of $D_i$'s.

Set $\iota_{ij}\colon D_{ij}:=D_i\cap D_j\hookrightarrow X$ for $i<j$. Again by a similar argument to $C$, we can see that $\iota_{ij}^\star C^1$ is concentrated on $-2,...,2$. By a similar truncation as above, we can define $C_{ij}$ with degree $0,1,2$ and a map $\iota_{ij}^\star C^1\rightarrow C_{ij}$. Again, we can define $C^2$ by the same argument and one can see that it is enough to show (\ref{claim}) for $C^2$.

By proceeding this induction, we finally have a holonomic complex supported on $\bigcap_{i\in I} D_i$ with degree $0,...., \dim X-I$. We can see that this holonomic complex again satisfies (\ref{claim}). This completes the proof.
\end{proof}

\subsection{t-exactness of various operations}
By using the functor $\frakF$, we can prove various $t$-exactness in parallel with the theory of usual perverse t-structure. We only discuss some of them for illustration.

\begin{proposition}
The Verdier dual $\bD$ induces a contravariant equivalence $\Ierv(\bk_X)\xrightarrow{\simeq} \Ierv(\bk_X)^{op}$.
\end{proposition}
\begin{proof}
Since $\bD$ is a contravariant equivalence of $D^b_{ic}(\Lambda_X)$, it suffices to show $\bD(\Ierv(\bk_X))\subset \Ierv(\bk_X)$. For $\cV\in \Ierv(\bk_X)$, we have $\frakF(\bD(\cV))\cong \bD(\frakF(\cV))\in \Perv(\bk_X)$. By Lemma \ref{Fpexact}, we have $\bD(\cV)\in \Ierv(\bk_X)$. This completes the proof.
\end{proof}

\begin{proposition}\label{texact}
Let $f\colon X\rightarrow Y$ be a morphism of complex manifolds. We assume that $f$ is proper for $3$ and $4$. The following holds:
\begin{enumerate}
\item $f^{-1}(\pD^{\geq 0}_{ic}(\Lambda_X))\subset \pD_{ic}^{\leq \dim X-\dim Y}(\Lambda_Y)$
\item $f^{!}(\pD^{\leq 0}_{ic}(\Lambda_X))\subset \pD_{ic}^{\geq -\dim X+\dim Y}(\Lambda_Y)$
\item For any $\cV\in \pD^{\geq 0}_{ic}(\Lambda_X)$, we have $\bR f_*\cV\in \pD_{ic}^{\geq -(\dim X-\dim Y)}(\Lambda_Y)$.
\item  For any $\cV\in \pD^{\leq 0}_{ic}(\Lambda_X)$, we have $\bR f_!\cV\in \pD_{ic}^{\leq (\dim X-\dim Y)}(\Lambda_Y)$.
\end{enumerate}
\end{proposition}
\begin{proof}
The statements 1, 3 and 4 can also be proved easily by using the commutativities of $\frakF$ with $f^{-1}$ and  $f_!$ (Section 6) and Lemma \ref{Fpexact}. The statement 2 is the Verdier dual of 1.
\end{proof}

\begin{remark}
Other right/left t-exactness for various functors known in the theory of perverse sheaves can be also proved by using the argument used in Proposition \ref{texact}.
\end{remark}

\begin{remark}
Here we assumed the properness for 3 and 4 for simplicity. One can remove the assumption by working with ind/pro objects to define push-forwards for non-tame morphisms.
\end{remark}

\section{Algebraic case}
In this section, we deduce the algebraic version of the results.

\subsection{Notations for algebraic $\cD$-modules}
For the theory of algebraic $\cD$-modules, we refer to \cite{HTTD}. For a smooth quasi-projective variety $X$, we denote the sheaf of algebraic differential operators by $\cD_X$. We denote the category of left $\cD_X$-modules by $\Mod(\cD_X)$, the bounded derived category by $D^b(\cD_X)$, and the full subcategory of cohomologically holonomic modules by $D^b_{\hol}(\cD_X)$.

We denote the Verdier dual by $\bD$. For a morphism of algebraic varieties $f\colon X\rightarrow Y$, we define $\int_f$ and $f^\dagger$ by the same formula as in the analytic case. In algebraic case, both functors preserve holonomic objects without properness assumption. We set $f^\star:=\bD\circ f^\dagger\circ \bD$ and $\int_{f!}:=\bD\circ \int_f\circ \bD$. Then we have two adjoint pairs $f^\star \dashv \int_f$ and $\int_{f!}\dashv f^\dagger$.

Let $X^{\an}$ be the complex manifold associated with $X$. The analytification functor is an exact functor $(\cdot)^{\an} \colon \Mod(\cD_X)\rightarrow \Mod(\cD_{X^{\an}})$. We also denote the induced functor on the derived categories by the same notation $(\cdot)^\an$. It preserves the holonomicity. We note the following:
\begin{lemma}[{\cite[Proposition 4.7.2]{HTTD}}]
Let $f\colon X\rightarrow Y$ be a morphism between algebraic varieties and $f^{\an}\colon X^{\an}\rightarrow Y^{\an}$ be the associated morphisms between complex manifolds. Then the following hold:
\begin{enumerate}
\item For $\cM\in D^b_{\hol}(\cD_Y)$, we have a canonical isomorphism $(f^\dagger \cM)^{\an}\simeq (f^{\an})^\dagger(\cM)^{\an}$.
\item  If $f$ is proper, we have a canonical isomorphism $\lb \int_f\cN\rb^\an \simeq \int_{f^\an}(\cN)^\an$ for $\cN\in D^b_{\hol}(\cD_X)$.
\end{enumerate}
\end{lemma}

\subsection{Algebraic irregular constructible sheaves}
Let $X$ be a smooth quasi-projective variety. Let $\oX$ be a smooth projective variety with a Zariski open embedding $i_X\colon X\rightarrow \oX$. We set $j_{X}\colon D_X:=\oX\bs X\hookrightarrow X$. 
\begin{definition}
An object $\cV\in \ModI(\Lambda_\dX)$ is {\em algebraic irregular constructible} if the following holds: there exists an algebraic stratification $\cS$ of $\oX$ refining $\oX=X\sqcup D_X$ such that each restriction of $\cV$ to $S\in \cS$ is an irregular local system.
\end{definition}
We denote the full subcategory of irregular constructible sheaves by $\Mod_{ic}(\Lambda_\dX)$. Let $i_\dX\colon \dX\rightarrow (\oX,\varnothing)$ be the canonical morphism, which is tame. We also denote the inclusion by $i_{D_X}\colon (D_X,\varnothing)\rightarrow (\oX, \varnothing)$. 
\begin{lemma}\label{characterization}
The functors $i_{\dX!}\colon \Mod_{ic}(\Lambda_\dX)\rightarrow \ModI(\Lambda_\oX)$ is fully faithful embedding onto the full subcategory spanned by objects satisfying $i_{D_X}^{-1}\cV\simeq 0$. The functor $i_{\dX*}$ is also fully faithful. In both cases, the quasi-inverses are given by $i_\dX^{-1}$.
\end{lemma}
\begin{proof}
This simply follows from Lemma \ref{icrecollement}
\end{proof}

\begin{lemma}
The category $\Mod_{ic}(\Lambda_\dX)$ does not depend on the choice of $\oX$.
\end{lemma}
\begin{proof}
We will prove the assertion in two steps, let us first assume that $p\colon \oX'\rightarrow \oX$ be a map between two projective compactifications of $X$ extending $\id\colon X\rightarrow X$. Then it is clear that $p_*$ induces an desired equivalence of categories.

Now let $\oX'$ be an arbitrary projective compactification of $X$. Then there exists $\oX''$ with maps $\oX''\rightarrow \oX$ and $\oX''\rightarrow \oX'$ extending $\id \colon X\rightarrow X$. This can be done by taking a smooth blow-up replacement of the closure of the diagonal embedding $X\rightarrow \oX\times \oX'$. From the first part of this proof, we have done. 
\end{proof}
We will denote the category of algebraic irregular constructible sheaves by $\Mod_{ic}(\Lambda_X)$. 
\begin{lemma}
The abelian subcategory $\Mod_{ic}(\Lambda_X)$ is thick in $\ModI(\Lambda_\dX)$.
\end{lemma}
\begin{proof}
One can prove by mimicking the proof of Lemma \ref{thickness}.
\end{proof}

Let us denote the triangulated subcategory of $D^b(\ModI(\Lambda_\dX))$ formed by cohomologically algebraic irregular constructible sheaves by $D^{b}_{ic}(\Lambda_X)$.

Let $D^{b}_c(\bk_X)$ be the category of cohomologically algebraic constructible complexes. 
\begin{proposition}
The functor $\frakF$ is restricted to a functor $\frakF\colon D^{b}_{ic}(\Lambda_X)\rightarrow D^{b}_c(\bk_X)$.
\end{proposition}
\begin{proof}
This is clear from the proof of Proposition \ref{ictoc}.
\end{proof}

It is also clear that what we proved in Section 5.4 also holds for $D^{b}(\Lambda_X)$.
Addition to those, we have the following:

Let $f\colon X\rightarrow Y$ be a morphism between algebraic varieties. Since we can always compactify this morphism as $\overline{X}\rightarrow \overline{Y}$, we have a map of topological spaces with boundary $f\colon \dX\rightarrow \dY$. We set
\begin{equation}
\begin{split}
\bR f_*&:=i_\dY^{-1}\circ \bR f_*i_{\dX*}\colon D^b(\ModI(\Lambda_\dX))\rightarrow D^b(\ModI(\Lambda_\dY))\\
\bR f_! &:=i_\dY^{-1}\circ \bR f_!i_{\dX!}\colon D^b(\ModI(\Lambda_\dX))\rightarrow D^b(\ModI(\Lambda_\dY)).
\end{split}
\end{equation}

\begin{proposition}
Let $f\colon X\rightarrow Y$ be a morphism between algebraic varieties. Then $\bR f_*\cV, \bR f_!\cV\in D^b_{ic}(\Lambda_Y)$ for $D^b_{ic}(\Lambda_X)$.
\end{proposition}
\begin{proof}
This can be proved by the same argument used in Proposition \ref{pushic} by using Proposition \ref{commDRalg}.
\end{proof}

\subsection{Algebraic Riemann--Hilbert correspondence}
We first recall the following Malgrange's result.
\begin{theorem}[\cite{Malgrangeirregular}]\label{Malgirreg}
If $X$ is a smooth projective variety, analytic holonomic $\cD_X$-modules are algebraic.
\end{theorem}

By using this, we have the following algebraic version of irregular Riemann--Hilbert correspondence.
\begin{theorem}\label{theoremalgRH}
There exists an exact equivalence
\begin{equation}
\Sol^\Lambda_X\colon D^b_{\hol}(\cD_X)\xrightarrow{\cong} D^{b}_{ic}(\Lambda_X).
\end{equation}
\end{theorem}
\begin{proof}
If $X$ is projective, there is nothing to prove by Theorem \ref{Malgirreg}. We suppose $X$ is quasi-projective and $\overline{X}$ be a compactification of $X$. For $\cM\in D^{b}_{\hol}(\cD_X)$, we have $\int_{i_X}\cM\in D^b_{\hol}(\cD_\oX)$ where $i_X\colon X\hookrightarrow \oX$ is the inclusion. We also set $i_\dX\colon \dX\rightarrow (\oX, \varnothing)$ the canonical morphism. Then we get a functor 
\begin{equation}
\Sol^\Lambda_X:= i_\dX^{-1}\circ \Sol^\Lambda_{\oX^\an}\circ(\cdot)^{\an} \circ\int_{i_X}\colon D^b_\hol(\cD_X)\rightarrow D^{b}_{ic}(\Lambda_{\oX^\an})= D^b_{ic}(\Lambda_\oX)\xrightarrow{i_\dX^{-1}} D^b_{ic}(\Lambda_\dX).
\end{equation}
The middle equality is Chow's lemma. Note that the first three compositions are fully faithful.

Hence, to prove the fully faithfulness of $\Sol^\Lambda_X$, it suffices to show that that image of $\Sol^\Lambda_{\oX^\an}\circ(\cdot)^{\an} \circ\int_{i_X}$ is zero under $i_{D_X}^{-1}$ by Lemma \ref{characterization}. Let $\cS$ be a stratification of $D_X$ such that each stratum is smooth. For $S\in \cS$, we have $i_S^{\dagger}(\int_{i_X}\cM)\simeq 0$ where $i_S\colon S\hookrightarrow \oX$ is the inclusion. Hence we have
\begin{equation}
\begin{split}
i_S^{-1}\Sol^{\Lambda}_X(\cM)&\simeq i_S^{-1}\circ \Sol^\Lambda_{\oX^\an}\circ (\cdot)^{\an}\circ \int_{i_X}\cM\\
&\simeq \Sol^\Lambda_{\oX^\an}\circ i_{S}^{\dagger}\circ \lb\int_{i_X}\cM\rb^\an\\
&\simeq \Sol^\Lambda_{\oX^\an}\circ \lb i_{S}^{\dagger}\circ \int_{i_X}\cM\rb^\an\simeq 0
\end{split}
\end{equation}
Hence we can conclude that $i_{D_X}^{-1}\Sol^\Lambda_{\oX^\an}(\cM)\simeq 0$. Then the fully faithfulness of $\Sol^\Lambda_X$ is evident from Lemma \ref{characterization}. 

To see the essential surjectivity, let us take an object $\cV\in D^b_{ic}(\Lambda_X)$ and consider it as an object of $D^b_{ic}(\Lambda_\oX)=D^b_{ic}(\Lambda_{\oX^\an})$. Then we have $\cM:=(\Sol^{\Lambda}_{\oX^{\an}})^{-1}(\cV)\in D^b_{\hol}(\cD_{\oX^\an})$. We set $\cM^\alg:=((\cdot)^\an)^{-1}(\cM)$. Take a stratification $\cS$ of $D_X$ with smooth strata. To prove $\cM^\alg$ is isomorphic to $\int_{i_X}\cN$ for $\cN\in D^b_{\hol}(\cD_X)$, it is enough to see $i_S^{\dagger}\cM^{\alg}\simeq 0$ for each $S\in \cS$. Note that $i_S^\dagger\cM^\alg\simeq 0$ is equivalent to $(i_S^\dagger\cM^\alg)^{an}\simeq 0$. The latter can be shown as follows:
\begin{equation}
\begin{split}
\Sol_{S^\an}^\Lambda(i_S^\dagger\cM^\alg)^{\an}&\simeq \Sol^\Lambda_{S^\an}(i_{S^\an}^\dagger\cM) \\
&\simeq i_S^{-1}\Sol^\Lambda_{\oX^\an}(\cM)\\
&\simeq i_S^{-1}\cV\simeq 0.
\end{split}
\end{equation}
This completes the proof.
\end{proof}

For the next section, we also prepare the following: We set $\DR^\Lambda_X:=\bD_{\dX}\circ \Sol^\Lambda_X$, which is an equivalence.
\begin{lemma}
There exist isomorphisms
\begin{equation}
\DR^\Lambda_X\simeq i_{\dX}^{-1}\circ \DR^\Lambda_{\oX^\an}\circ (\cdot)^\an\circ \int_{i_X}\simeq i_\dX^{-1}\circ \DR^\Lambda_{\oX^\an}\circ (\cdot)^\an\circ \int_{i_X!}.
\end{equation}
\end{lemma}\label{anothersol}
\begin{proof}
We have
\begin{equation}
\begin{split}
\DR^\Lambda_X&\simeq \bD_{\dX}\circ i_\dX^{-1}\circ \Sol^\Lambda_{\oX^\an}\circ (\cdot)^\an\circ\int_{i_X}\\
&\simeq i^{-1}_\dX\circ \bD_\oX\circ \Sol^\Lambda_{\oX^\an}\circ (\cdot)^\an \circ\int_{i_X}\\
&=i^{-1}_\dX\circ \DR^\Lambda_{\oX^\an}\circ (\cdot)^\an \circ \int_{i_X}.
\end{split}
\end{equation}
The second equality is clear from the composition of $i^{-1}_\dX$.
\end{proof}

\subsection{Commutativities}
Let $X, Y$ be smooth quasi-projective varieties and $f\colon X\rightarrow Y$ be a morphism. Recall that holonomicity of algebraic $\cD$-modules are preserved by six operations.
\begin{proposition}\label{commDRalg}
There exist canonical isomorphisms:
\begin{eqnarray}
\bR f_*\circ \DR_X^\Lambda&\simeq \DR^\Lambda_Y\circ \int_{f}\\
\bR f_!\circ \DR_X^\Lambda &\simeq \DR_Y^\Lambda \circ \int_{f!}\\
f^{-1}\circ \DR_Y^\Lambda&\simeq \DR_X^\Lambda\circ f^\star \\
f^!\circ \DR_Y^\Lambda&\simeq \DR_X^\Lambda\circ f^\dagger.
\end{eqnarray}
\end{proposition}
\begin{proof}
The third and fourth lines follow from the analytic cases. To prove the first and second ones, let us take a projective compactification $i_X\colon X\hookrightarrow \oX$ and $i_Y\colon Y\hookrightarrow \oY$ and a map $\overline{f}\colon \oX\rightarrow \oY$ extending $f$. We have 
\begin{equation}
\begin{split}
\DR^\Lambda_Y\circ \int_{f!} &\simeq i_\dY^{-1}\circ \DR^\Lambda_{\oY^\an}\circ (\cdot)^\an \circ \int_{i_Y!}\circ \int_{f!}\\
&\simeq  i_\dY^{-1}\circ \DR^\Lambda_{\oY^\an}\circ (\cdot)^\an \circ \int_{\overline{f}}\circ \int_{i_X!}\\
&\simeq i_\dY^{-1}\circ \bR \overline{f}_!\circ \DR^\Lambda_{\oY^\an}\circ (\cdot)^\an \circ \int_{i_X!}\\
&\simeq \bR f_!\circ i_\dX^{-1}\circ  \DR^\Lambda_{\oX^\an}\circ (\cdot)^\an\circ  \int_{i_X!}\\
&\simeq \bR f_!\circ \DR^\Lambda_X.
\end{split}
\end{equation}
By using Lemma \ref{anothersol}, one can prove the first formula in the same way.
\end{proof}

\subsection{Algebraic irregular perverse sheaves}
In the same way as in Definition \ref{irregularperversetstr}, we define $(\pD^{\leq 0}_{ic}(\Lambda_X), \pD^{\geq 0}_{ic}(\Lambda_X))$ on $D^{b}_{ic}(\Lambda_X)$. 
\begin{proposition}
The following holds:
\begin{enumerate}
\item The pair $(\pD^{\leq 0}_{ic}(\Lambda_X), \pD^{\geq 0}_{ic}(\Lambda_X))$ forms a t-structure on $D^{b}_{ic}(\Lambda_X)$.
\item The heart $\Ierv(k_X)$ of the t-structure $(\pD^{\leq 0}_{ic}(\Lambda_X), \pD^{\geq 0}_{ic}(\Lambda_X))$ is equivalent to $\Mod_{\hol}(\cD_X)$ under the Riemann--Hilbert correspondence (Theorem \ref{theoremalgRH}).
\item  The heart $\Ierv(\bk_X)$ is stable under the Verdier dual. The t-exactness in Proposition \ref{texact} also holds in this setting without the properness assumption.
\end{enumerate}
\end{proposition} 
\begin{proof}
One can prove in the same way as in the analytic setting except for non-proper setting of 3. Let $f$ be a morphism $X\rightarrow Y$ and a compactification $\oX\rightarrow \oY$. Then we have 
\begin{equation}
\frakF\circ \bR f_!:=\frakF\circ \bR\of_!\circ i_{\dX!}\simeq \bR \of_!\circ i_{X!}\circ \frakF\simeq \bR f_!\circ \frakF
\end{equation}
by Lemma \ref{Fpush} and \ref{Fpush2}. This proves the desired statement for $\bR f_!$. The statement for $\bR f_*$ is obtained by taking the Verdier dual.
\end{proof}

\section{Fukaya categorical Riemann--Hilbert correspondence}
In this section, we speculate some constructions which explain the appearance of $\Lambda$-module more naturally by using Fukaya category. We hope more details or proofs will be discussed in future papers.

Let us recall the following theorem: Let $Z$ be a compact real analytic manifold and $D^b_{\bR\hi c}(Z)$ be the bounded derived category of $\bR$-constructible sheaves. 
\begin{theorem}[\cite{NZ, Nad}]
There exists a Fukaya-type $A_{\infty}$-category $\Fuk(T^*Z)$ of $T^*Z$ and an equivalence
\begin{equation}
D^b_{\bR\hi c}(Z)\simeq D\Fuk(T^*Z)
\end{equation}
where the right hand side is the derived category of $\Fuk (T^*Z)$.
\end{theorem}
For the definition of this kind of Fukaya category, see the original reference \cite{NZ}. We modify this story to explain irregular Riemann--Hilbert correspondence well. There are three ingredients.

The first ingredient is an observation that Nadler--Zalsow's Fukaya category is not enough sensitive for our purpose in the following sense.

\begin{example}
Let $Z=\bR$ and consider functions $f_1=1/x$ and $f_2=1/x^2$ defined over $x>0$. For the sake of simplicity, we only consider locally around $0$. Let us consider $\Lambda^{f_i}$. Then we have 
\begin{equation}
\begin{split}
\Hom_{\ModI(\Lambda_X)}(\Lambda^{f_1}, \Lambda^{f_2})&=\bk\\
\Hom_{\ModI(\Lambda_X)}(\Lambda^{f_2}, \Lambda^{f_1})&=0.
\end{split}
\end{equation}
Let $L_i$ be the graph of differentials in $T^*Z$ for each $f_i$, which are Lagrangian submanifolds in $T^*Z$. 

We would like to consider Lagrangian intersection Floer theory $CF(L_1, L_2)$ and $CF(L_2, L_1)$ and compare them with hom-spaces in $D^b\ModI(\Lambda_X)$. Assume that they are well-defined. As usual, to deal with intersections at infinity, we consider Reeb perturbation for Lagrangians in the right of $CF(\cdot, \cdot)$. Then one can see that 
\begin{equation}\label{correctintersection}
CF(L_1, L_2)=\bk  \text{ and } CF(L_2, L_1)=0
\end{equation}

In the formalism of Nadler--Zaslow, this does not occur. Since both $L_i$'s have the same asymptotic line, it should give the same object $\bk_{(0, \infty)}$ in $D^b_{\bR\hi c}(Z)$.
\end{example}
The first conjecture is the following.
\begin{conjecture}
There exists a version of Fukaya category $\Fuk_{m}(T^*Z)$ of $T^*Z$ which modifies Nadler--Zaslow's one to realize $(\ref{correctintersection})$.
\end{conjecture}

The second ingredient is $\bR$-graded realization of irregular constructible sheaf. For the below, we replace $Z$ by $X$, which is a complex manifold. Let $\cM$ be a meromorphic connection with poles in $D$. Then the corresponding object under the irregular Riemann--Hilbert correspondence is an irregular local system over $(X, D)$. To simplify our explanation, we assume that the formal types of $\cM$ is not ramified. We have a $C^\infty$ function on $X\bs D$ such that $\Lambda^f$ represents the irregular local system. The graph of the derivative of $f$ gives a Lagrangian $L_\cM$ in $T^*X$. Since the choice of $f$ is only up to bounded modification, we can expect the following. One can also modify the above explanation to adopt to any irregular local system.
\begin{conjecture}
The Lagrangian $L_{\cM}$ defines an object of $\Fuk_m(T^*X)$ which does not depend on the choice of $f$. We denote full subcategory of $\Fuk_m(T^*X)$ consisting of such objects by $\Fuk_{icnov}(T^*X)$.
\end{conjecture}

The third ingredient is the fact that the Fukaya category is a priori defined over the Novikov ring, we have to replace $\bk$ with the Novikov ring. In the definition of $\Mod^0(\Lambda_X)$, by not taking $\otimes_\Lambda \bk$, we can define  a category defined over $\Lambda$. By taking a triangulated hull, we denote it by $D^b(\Mod^\bR(\Lambda_X))$.

Building on the above ingredients, we can now state the main conjecture.
\begin{conjecture}
\begin{enumerate}
\item The $\Fuk_{icnov}(T^*X)$ is defined over the finite Novikov ring $\Lambda$ due to the nonexistence of bubbles.
\item There exists a fully faithful embedding $D\Fuk_{icnov}(T^*X)\hookrightarrow D^b(\Mod^\bR(\Lambda_X))$.
\item After reducing coefficients $D\Fuk_{ic}(T^*X):=D\Fuk_{icnov}(T^*X)\otimes_\Lambda \bk$, there exists an equivalence
\begin{equation}
D^b_{ic}(\Lambda_X)\simeq D\Fuk_{ic}(T^*X).
\end{equation}
induced by the embedding in 2. In particular, for $\bk=\bC$, we get a Fukaya-categorical Riemann--Hilbert correspondence
\begin{equation}
D^b_{\hol}(\cD_X)\simeq D\Fuk_{ic}(T^*X).
\end{equation}
\end{enumerate}
\end{conjecture}

\begin{remark}\label{final}
For a fixed formal type of irregular singularity in dimension 1, one can draw an isotopy class of Legendrian knot (e.g \cite{STWZ}). Then the subcategory of Nadler--Zaslow's Fukaya category ending at the Legendrian is equivalent to the derived category of holonomic $\cD$-module of the given formal type. This is a variant of the above conjecture.
\end{remark}

\section*{Appendix: Stackification}
This material is well-known, but we summarize in the form fitted to this paper. We mainly follow the exposition of 
\cite{Moerdijk}. 

Let $(\oX, D_X)$ be a topological space with boundary and $\Open_{\dX}$ be the associated site. Let $\cF$ be a presheaf of categories over $\Open_\dX$ i.e., $\cF$ is a functor $\Open_{\dX}$ (not a 2-functor). For $U\in \Open_\dX$ and $a, b\in \cF(U)$, $\Hom_{\cF}(a,b)\colon V\mapsto \Hom_{\cF(V)}(a|_V, b|_V)$ for $V\subset U$ forms a presheaf of sets.

\begin{definition}
A presheaf of categories is a \emph{prestack} if the presheaves of the form $\Hom(a,b)$ are sheaves.
\end{definition}
\begin{remark}
Usually, a prestack is a condition for fibered categories. We will use a restricted notion in this paper.
\end{remark}

For a given sheaf of categories, we can produce a prestack $\tilde{\cF}$ in the following way: For $U\in \Open_{\dX}$, the set of objects of $\tilde{\cF}(U)$ is the same as $\cF(U)$, For $a, b\in \tilde{\cF}(U)$, $\Hom_{\tilde{\cF}(U)}(a,b)$ is defined by the global section of the sheafification of the presheaf $\Hom(a,b)$. For sheafification over sites, see \cite{Stacksproject}. From the construction, one can see that a section over $U$ of the resulting sheaf is represented by a set of local sections if one takes a covering of $U$ fine enough.

Now we have the sheaf property on the hom-spaces, but not on objects. To remedy this, we will do stackification.
\begin{definition}
A descent data for a cover $\{U_i\}_{i\in I}$ is a set of objects $\{a_i\}_{i \in I} (a_i\in \cF(U_i))$ and a set of morphisms $\theta_{ij}\colon a_j|_{U_i\cap U_j}\rightarrow a_i|_{U_i\cap U_j}$ ($i,j\in I$) such that $\theta_{ii}=1, \theta_{ij}\circ \theta_{jk}=\theta_{ik}$.
\end{definition}
For an object $a\in \cF(\bigcup_{i\in I}U_i)$, we can get a descent data by the restrictions. 
\begin{definition}
A prestack is a \emph{stack} if any descent data comes from the restrictions.
\end{definition}

The existence of stackification (cf. \cite[Theorem 2.1]{Moerdijk}) tells us there exists a stack $\hat{\cF}$ associated to a given prestack. An object of $\hat{\cF}(U)$ for $U\in \Open_\dX$ is represented by a cover $\{U_i\}_{i\in I}$ and a descent data on this cover.

\bibliographystyle{amsalpha}
\bibliography{irregularperversesheavearXiv2.bbl}

\noindent
Tatsuki Kuwagaki

\noindent
Kavli Institute for the Physics and Mathematics of the Universe (WPI), The University of Tokyo Institutes for Advanced Study, The University of Tokyo, Kashiwa, Chiba 277-8583, Japan.

{\em e-mail address}\ : \  tatsuki.kuwagaki@ipmu.jp

\end{document}